\documentclass[a4paper]{article}
\usepackage{arxiv}

\usepackage[utf8]{inputenc}
\usepackage[T1]{fontenc}
\usepackage{lmodern}
\usepackage{microtype}

\usepackage{amsmath}
\usepackage{amssymb}
\usepackage{amsthm}

\usepackage{graphicx}
\usepackage{xcolor}
\usepackage{booktabs}
\usepackage{caption}
\usepackage{subcaption}
\usepackage{tikz}
\usetikzlibrary{positioning, arrows.meta}

\usepackage{enumitem}
\usepackage{lineno}
\usepackage{xspace}

\usepackage{url}
\usepackage{hyperref}
\usepackage{orcidlink}

\hypersetup{
	colorlinks=true,
	linkcolor=cyan,
	urlcolor=blue,
	citecolor=magenta,
	linktoc=all
}

\newtheorem{theorem}{Theorem}[section]
\newtheorem{proposition}[theorem]{Proposition}
\newtheorem{lemma}{Lemma}[section]
\newtheorem{remark}[theorem]{Remark}

\newtheoremstyle{problemstyle}{3pt}{3pt}{\normalfont}{}{\bfseries}{.}{0.5em}{}
\theoremstyle{problemstyle}
\newtheorem{tprob}{Test Problem}

\newenvironment{testproblem}[1]
{\begin{tprob}\textbf{(#1)}\ }
	{\end{tprob}}

\newcommand{\con}{\mathbf{u}}
\newcommand{\prim}{\mathbf{p}}
\newcommand{\flux}{\mathbf{f}}
\newcommand{\fluxx}{\mathbf{f}}
\newcommand{\fluxy}{\mathbf{g}}

\newcommand{\rc}{RC-EOS\xspace}
\newcommand{\id}{ID-EOS\xspace}
\newcommand{\ip}{IP-EOS\xspace}
\newcommand{\tm}{TM-EOS\xspace}

\newcommand{\wenojs}{WENO-JS\xspace}
\newcommand{\wenoz}{WENO-Z\xspace}
\newcommand{\wenoao}{WENO-AO\xspace}
\newcommand{\wenoaoi}{WENO-AOI\xspace}

\title{Constraint Preserving AFD-WENO Schemes for Relativistic Hydrodynamics with General Equations of State}
\author{
	\textbf{Pramodit Mishra}\,\orcidlink{0009-0001-9053-0667}\qquad
	\textbf{Shubham Upadhyay}\,\orcidlink{0009-0000-7846-003X}\\[0.3em]
	\textbf{Rakesh Kumar}\,\orcidlink{0000-0002-5829-0384}\qquad
	\textbf{Biswarup Biswas}\,\orcidlink{0000-0001-7771-6400}\thanks{Corresponding author at: Department of Mathematics, École Centrale School of Engineering, Mahindra University, Hyderabad, 500043, Telangana, India.\newline \textit{E-mail addresses:} \href{mailto:pramoditmishra21@gmail.com}{\texttt{pramoditmishra21@gmail.com}} (P.~Mishra), \href{mailto:shubham.u266@gmail.com}{\texttt{shubham.u266@gmail.com}} (S.~Upadhyay), \href{mailto:rakesh.kumar@mahindrauniversity.edu.in}{\texttt{rakesh.kumar@mahindrauniversity.edu.in}} (R.~Kumar), \href{mailto:biswarupb7@gmail.com}{\texttt{biswarupb7@gmail.com}} (B.~Biswas).}\\[0.8em]
	Department of Mathematics, École Centrale School of Engineering\\
	Mahindra University, Hyderabad 500043, Telangana, India
}
\date{}
\begin{document}
	\maketitle
	
	\begin{abstract}
		We develop a high-order physical-constraint-preserving (PCP) alternative finite difference weighted essentially non-oscillatory (AFD-WENO) scheme for the special relativistic hydrodynamics equations with general equations of state. The proposed scheme comprises two key limiters: a state limiter, which acts after the WENO state interpolation step, and a flux limiter, which acts on the final high-order fluxes. The state limiter ensures that the interpolated states are physically admissible, while the flux limiter ensures that the numerical fluxes are physically admissible. The resulting scheme is rigorously proved to satisfy the physical constraints. Incorporating multiple WENO interpolation techniques, including an improved adaptive-order formulation (\wenoaoi), the method is validated through extensive one- and two-dimensional numerical benchmarks with various equations of state. The numerical results demonstrate high-order accuracy, sharp resolution of discontinuities, and robust stability in extreme relativistic regimes.
	\end{abstract}
	
	\section{Introduction}
	\label{sec:introduction}
	Relativistic hydrodynamics (RHD) provides the fundamental framework for modeling fluid flows in which the fluid velocity approaches the speed of light or the internal energy becomes comparable to the rest-mass energy. Such flows arise naturally in numerous high-energy astrophysical phenomena, including relativistic jets from active galactic nuclei and gamma-ray bursts, pulsar wind nebulae, core-collapse supernovae, and neutron star mergers~\cite{begelman1984theory,bottcher2012relativistic,mirabel1999sources,zensus1997parsec, piran2004physics,gaensler2006evolution, janka2012explosion,baiotti2017binary}.
	The governing equations form a system of nonlinear hyperbolic conservation laws whose solutions contain strong shocks, contact discontinuities, rarefaction waves, and complex multidimensional wave interactions. In a multi-dimensional Cartesian framework, these equations can be compactly expressed as
	\begin{equation*}
		\frac{\partial \con}{\partial t} + \sum_{d} \frac{\partial \flux_d(\con)}{\partial x_d} = \mathbf{0},
	\end{equation*}
	where $\con$ denotes the vector of conserved variables, and $\flux_d(\con)$ represents the corresponding flux vector in the $x_d$-direction. Analytical solutions to these equations are generally unavailable due to their highly nonlinear nature, most notably dictated by the implicit presence of the Lorentz factor and the intricate coupling between the conserved and primitive variables. Consequently, robust numerical methods are indispensable for studying the evolution of relativistic flows.
	
	\par The development of numerical methods for the RHD equations has a long trajectory, beginning with the pioneering work of Wilson in 1972 \cite{wilson1972numerical}, who introduced an explicit finite-difference approach relying on artificial viscosity to capture shocks. While foundational, such artificial viscosity techniques suffer from severe inaccuracies and excessive numerical smearing when dealing with highly relativistic flows characterized by large Lorentz factors \cite{centrella1984planar}. As a result, Eulerian RHD simulations developed slowly until the 1990s. Modern high-resolution shock-capturing Godunov-type methods then led to significant progress. This era successfully introduced approximate and exact Riemann solvers to the relativistic regime \cite{marti1991numerical,marti1994analytical,dai1997iterative,ibanez1999riemann}. 
	
	Following these exact Riemann solvers, researchers progressively developed higher-order spatial reconstructions to improve accuracy in smooth regions. Notable advancements include extensions of the piecewise parabolic method (PPM) \cite{marti1996extension,aloy1999genesis,mignone2005piecewise}, as well as essentially non-oscillatory (ENO) and weighted ENO (WENO) schemes \cite{dolezal1995relativistic,weno2007tchekhovskoy}. An extensive review of these classical formulations and their comparative performance across various relativistic test problems is provided by Font \cite{font2008numerical}.
	
	A physically meaningful solution of the RHD equations must satisfy three fundamental \emph{constraints}: the positivity of rest-mass density, the positivity of pressure, and the subluminal velocity condition. Violations of these constraints can lead to unphysical solutions, numerical instabilities, and consequently, the breakdown of the simulation. Early high-order schemes often failed to maintain these constraints, particularly in the presence of strong shocks or low-density regions. To address this issue, researchers developed various physical-constraint-preserving (PCP) techniques that combine suitable numerical fluxes with convex scaling limiters to ensure that the numerical solution remains within the physically admissible state set under a valid CFL condition \cite{wu2015high,qin2016bound,zhao2017central}.
	
	The equation of state (EOS) provides the essential thermodynamic closure relating pressure, rest-mass density, and specific internal energy, and its choice critically determines both the qualitative and quantitative character of the solution. While the ideal-gas EOS has been widely adopted for its algebraic simplicity, realistic astrophysical applications involving relativistic temperatures, variable composition, or dense nuclear matter require more sophisticated closures. To approximate the thermodynamically consistent but computationally expensive Synge EOS \cite{synge1958relativistic}, several effective EOS models have been proposed that retain computational tractability. Most notable among these are the Taub--Mathews EOS \cite{mathews1971hydromagnetic} and the Ryu--Chattopadhyay EOS \cite{ryu2006equation}. The development of robust numerical schemes applicable to a general class of EOS has consequently become an important requirement for high-fidelity, astrophysically relevant RHD simulations \cite{falle1996upwind,wu2017physical,xu2024high}.
	
	Among high-order methods for hyperbolic conservation laws, WENO schemes occupy a central place \cite{liu1994weighted}. Following the classical fifth-order formulation of Jiang and Shu \cite{jiang1996efficient}, the framework has been extensively extended to improve accuracy and robustness near discontinuities \cite{henrick2005mapped, borges2008improved, castro2011high}. Of particular relevance here are Alternative Finite Difference WENO (AFD-WENO) schemes \cite{jiang2013alternative, balsara2025efficient}. The finite-volume WENO framework for multidimensional hyperbolic systems requires quadrature-based flux integration, which can be computationally expensive because it involves additional evaluations of the flux function at quadrature nodes beyond the cell interfaces. The AFD-WENO framework, similar to the classical finite-volume WENO framework, employs a Riemann solver; however, it avoids these additional flux evaluations by incorporating a high-order correction term to achieve the desired order of accuracy \cite{balsara2025efficient}.
	
	Despite the clear advantages of the AFD-WENO approach, maintaining physical admissibility across complex thermodynamic closures introduces significant mathematical challenges. While physical-constraint-preserving frameworks have been successfully developed for classical WENO and discontinuous Galerkin schemes under a general EOS \cite{wu2017physical, xu2024high}, and PCP extensions of the AFD-WENO framework have recently been introduced for the ideal-gas EOS \cite{balsara2025efficient}, the intersection of these advancements remains unexplored. This highlights an opportunity to extend the AFD-WENO framework to accommodate the thermodynamic requirements of realistic astrophysical simulations by examining the direct applicability of recently developed PCP methods to a general EOS.

	\par In this work, we develop a physical-constraint-preserving framework for the AFD-WENO scheme for relativistic hydrodynamics with a general EOS. The framework utilizes both flux and state limiting. In addition, we introduce efficient state-variable WENO interpolation strategies within the AFD-WENO framework to improve accuracy.
	
	The remainder of this paper is organized as follows. In Section~\ref{sec:preliminaries}, we introduce the governing equations of relativistic hydrodynamics with a general EOS, the admissible state set, and its key mathematical properties. Section~\ref{sec:afd_weno} presents the AFD-WENO framework, including the alternative finite-difference formulation and the WENO interpolation strategies. Section~\ref{sec:extension_2D} extends the scheme to two spatial dimensions. Section~\ref{sec:pcp} develops the PCP limiters and establishes a rigorous proof of constraint preservation. Numerical experiments for one- and two-dimensional benchmark problems are reported in Section~\ref{sec:numerics} to demonstrate the accuracy, robustness, and constraint-preserving properties of the proposed scheme. Finally, conclusions are drawn in Section~\ref{sec:conclusions}.

	\section{Preliminaries}
	\label{sec:preliminaries}
	In this section, we introduce the governing equations of RHD, the equations of state considered in this work, the conservative-to-primitive variable conversion procedure, and the physical constraints that must be preserved by numerical schemes.
	
	\subsection{Governing Equations}
	\label{subsec:governing_equations}
	We consider the two-dimensional RHD equations in the laboratory frame considering the speed of light $c = 1$, which can be written in conservative form as
	\begin{align}\label{SRHD}
		\frac{\partial \con}{\partial t} + \frac{\partial \fluxx(\con)}{\partial x} + \frac{\partial \fluxy(\con)}{\partial y} = 0,
	\end{align}
	where $\con$ is the vector of conserved variables, and $\fluxx(\con)$ and $\fluxy(\con)$ are the flux vectors in the $x$ and $y$ directions, respectively, given by
	\begin{align*}\con = \begin{pmatrix}D \\ m_x \\ m_y \\ E \end{pmatrix}, \quad
		\fluxx(\con) = \begin{pmatrix}D v_x \\ m_x v_x + p \\ m_y v_x \\ m_x \end{pmatrix}, \quad
		\fluxy(\con) = \begin{pmatrix}D v_y \\ m_x v_y \\ m_y v_y + p \\ m_y \end{pmatrix}.
	\end{align*}
	Here, $D$ is the conserved mass density, $m_x$ and $m_y$ are the momentum densities in the $x$ and $y$ directions, $E$ is the total energy density, $p$ is the pressure, and $v_x$ and $v_y$ are the velocity components in the $x$ and $y$ directions, respectively. 
	
	The conserved variables $\con = (D, m_x, m_y, E)^T$ are related to the primitive variables $\prim = (\rho, v_x, v_y, p)^T$, where $\rho$ is the rest-mass density, through the following relations:
	\begin{align*}D = \rho W, \quad m_x = \rho h W^2 v_x, \quad m_y = \rho h W^2 v_y, \quad E = \rho h W^2 - p,
	\end{align*}
	where 
	\begin{align*}W = \frac{1}{\sqrt{1 - v^2}}
	\end{align*}
	is the Lorentz factor with $v^2 = v_x^2 + v_y^2 < 1$. To close the system \eqref{SRHD}, an EOS is required.
	
	\subsection{Equation of State}
	\label{subsec:eos}
	We consider a general EOS of the form
	\begin{align*}h(\rho, p) = 1 + \epsilon(\rho, p) + \frac{p}{\rho}
	\end{align*}
	This formulation simplifies the conservative-to-primitive variable conversion and is commonly adopted in the literature \cite{wu2017physical,basak2025constraints}. For the hyperbolicity of the system \eqref{SRHD}, the sound speed should satisfy $0 < c_s < 1$ where $c_s$ is defined as \begin{align*}c_s^2 = -\frac{\rho}{nh}\frac{\partial h}{\partial \rho},\quad n = \rho\frac{\partial h}{\partial p}-1,
	\end{align*}
	with the specific enthalpy $h$ satisfying the following inequality~\cite{wu2017physical,xu2024high}
	\begin{align*}h(p, \rho) \geq \sqrt{1 + \frac{p^2}{\rho^2}} + \frac{p}{\rho},
	\end{align*}
	as required by relativistic kinetic theory \cite{basak2025constraints}.
	The most commonly used model is the \emph{ideal gas EOS} (\id)
	\begin{align*}h = 1 + \frac{\Gamma}{\Gamma - 1}\frac{p}{\rho},
	\end{align*}
	where $\Gamma \in (1, 2]$ is the adiabatic index. However, as noted in \cite{basak2025constraints}, this EOS is derived from non-relativistic thermodynamics and provides a poor approximation for many relativistic flows, particularly for semi-relativistic or two-component fluids. 
	
	To overcome this limitation, quite a few equations of state are used in the literature that provide better approximations in the relativistic regime. In addition to the \id, we consider the following three equations of state in this work
	\begin{itemize}
		\item \emph{Taub--Matthews} (\tm) \cite{mathews1971hydromagnetic,mignone2005piecewise}
		\begin{align*}h = \frac{5p}{2\rho} + \sqrt{\frac{9p^2}{4\rho^2} + 1}.
		\end{align*}
		
		\item \emph{Ideal polytropic} (\ip) \cite{sokolov2001simple}
		\begin{align*}h = \frac{2p}{\rho} + \sqrt{\frac{4p^2}{\rho^2} + 1}.
		\end{align*}
		
		\item \emph{Ryu--Chattopadhyay} (\rc) \cite{ryu2006equation}
		\begin{align*}h = \frac{2(6p^2 + 4p\rho + \rho^2)}{\rho(3p + 2\rho)}.
		\end{align*}
	\end{itemize}
	
	\begin{remark}
		The inverse transformation from conserved variables $\con$ to primitive variables $\prim$ is nonlinear and typically requires iterative methods for numerical solution, which are well-established in the literature \cite{cai2024provably,basak2025constraints}. We use the method proposed in \cite{cai2024provably} for the conservative-to-primitive variable conversion in this work, which is provably robust and efficient for a wide range of equations of state. For \rc, we use the method proposed in \cite{basak2025constraints} which is specifically designed for this EOS and is also provably robust and efficient.
	\end{remark}
	
	\subsection{Hyperbolicity and Physical Constraints}
	\label{subsec:hyperbolicity_constraints}
	The system \eqref{SRHD} is equipped with a complete set of eigenvalues and eigenvectors of the Jacobian matrices $\partial \fluxx/\partial \con$ and $\partial \fluxy/\partial \con$, confirming that it is a hyperbolic system of conservation laws. Physically admissible solutions must satisfy the following constraints
	\begin{align*}\rho > 0, \quad p > 0, \quad |\mathbf{v}| < 1,
	\end{align*}
	where $\mathbf{v}=(v_x,v_y)^T$ is the fluid velocity. The corresponding admissible set of conservative states is denoted by
	\begin{equation}
		\mathcal{G}_\prim\label{eq:admissible_set_prim}
		=
		\left\{
		\con
		:
		\rho(\con)>0,\;
		p(\con)>0,\;
		|\mathbf{v}(\con)|<1
		\right\}.
	\end{equation}
	
	These constraints ensure positive density and pressure, as well as subluminal fluid velocities. Preserving these physical constraints at the discrete level is crucial for the robustness and stability of numerical schemes for RHD. Violations of these constraints can lead to unphysical solutions, numerical instabilities, or failure of the conservative-to-primitive variable conversion procedure. 
	
	\section{AFD-WENO schemes}\label{sec:afd_weno}
	We start by considering the following one-dimensional system
	\begin{align}\label{SRHD_1D}
		\frac{\partial \con}{\partial t} + \frac{\partial \fluxx(\con)}{\partial x} = \mathbf{0},
	\end{align}
	and system \eqref{SRHD} can be solved by applying the one-dimensional scheme in each direction using the commonly adopted tensor product approach \cite{jiang2013alternative,wu2017physical}.
	
	The computational domain is discretized by intervals $I_j = [x_{j-1/2}, x_{j+1/2}]$ with a uniform grid size $\Delta x$. Using the discretization, we can write the semi-discrete finite difference scheme for \eqref{SRHD_1D} as
	\begin{equation}\label{eq:1d_scheme}
		\frac{d}{dt} \con_j + \frac{1}{\Delta x} \left( \hat{\mathbf{f}}_{j+1/2} - \hat{\mathbf{f}}_{j-1/2} \right) = \mathbf{0},
	\end{equation}
	where $\con_j$ approximates $\con(t, x_j)$ and $\hat{\mathbf{f}}_{j+1/2}$ is the numerical flux at the cell interface $x_{j+1/2}$ and $k$-th order accurate in the sense that
	\begin{align*}
		\frac{1}{\Delta x} \left( \hat{\mathbf{f}}_{j+1/2} - \hat{\mathbf{f}}_{j-1/2} \right) = \mathbf{f}(\con(x))_x|_{x_j} + \mathcal{O}(\Delta x^k).
	\end{align*}
	A numerical flux $\hat{\mathbf{f}}_{j+1/2}$ is consistent with the physical flux $\flux$ in the sense that $\hat{\mathbf{f}}(\con, \con, \dots, \con) = \flux(\con)$ for any $\con$. A classical finite difference WENO scheme is a high-order scheme that computes the numerical flux $\hat{\mathbf{f}}_{j+1/2}$ using a convex combination of lower-order fluxes computed on different stencils. The weights in this convex combination are designed to adaptively select the smoothest stencil, thereby achieving high-order accuracy in smooth regions while avoiding spurious oscillations near discontinuities.
	
	\subsection{Alternative WENO Formulations}
	
	The AFD-WENO scheme differs from the classical finite difference WENO formulation in that it reconstructs the point values of the conserved variables and evaluates the numerical flux as the sum of a low-order Riemann flux and a high-order correction term. Accordingly, the numerical flux at the interface $x_{j+\frac12}$ is expressed as
	\begin{equation*}
		\hat{\mathbf{f}}_{j+\frac12}
		=
		\hat{\mathbf{f}}^{\,\mathrm{low}}_{j+\frac12}
		+
		\hat{\mathbf{f}}^{\,\mathrm{cor}}_{j+\frac12},
	\end{equation*}
	where
	\begin{equation*}
		\hat{\mathbf{f}}^{\,\mathrm{low}}_{j+\frac12}
		=
		\mathcal{F}
		\left(
		\con^-_{j+\frac12},
		\con^+_{j+\frac12}
		\right)
	\end{equation*}
	is obtained using an approximate Riemann solver, while the correction term
	$\hat{\mathbf{f}}^{\,\mathrm{cor}}_{j+\frac12}$ is constructed from the physical flux values
	$\mathbf{f}_j=\mathbf{f}(\con_j)$.
	
	The interface states are reconstructed by applying the WENO interpolation component-wise to the conserved variables,
	\begin{align*}
		\con^-_{j+\frac12}
		&=
		\mathrm{WENO}
		\left(
		\con_{j-r+1},
		\ldots,
		\con_{j+r-1}
		\right),\\
		\con^+_{j+\frac12}
		&=
		\mathrm{WENO}
		\left(
		\con_{j+r},
		\ldots,
		\con_{j-r+2}
		\right).
	\end{align*}
	
	For a $(2r-1)$th-order scheme, the correction term ensures the desired order of accuracy. The correction coefficients for different orders are given in \cite{jiang2013alternative,balsara2025efficient}. In this work, we employ the fifth-order ($r=3$) correction
	\begin{equation*}
		\hat{\mathbf{f}}^{\,\mathrm{cor},3}_{j+\frac12}
		=
		\frac{19}{3840}
		\left(
		\mathbf{f}_{j-2}
		+
		\mathbf{f}_{j+3}
		\right)
		-
		\frac{137}{3840}
		\left(
		\mathbf{f}_{j-1}
		+
		\mathbf{f}_{j+2}
		\right)
		+
		\frac{59}{1920}
		\left(
		\mathbf{f}_{j}
		+
		\mathbf{f}_{j+1}
		\right).
	\end{equation*}
	
	When the WENO interpolation is applied directly to the conserved variables, the resulting method is  computationally efficient, however, may produce spurious oscillations near strong discontinuities. To improve robustness, the reconstruction can instead be performed in the local characteristic space.
	
	In the local characteristic decomposition (LCD) approach, the characteristic basis is constructed independently at each interface.
	
	The interface state is first approximated by an appropriate average $\con_{j+\frac12}$, and the corresponding flux Jacobian is diagonalized as
	\begin{align*}
		\mathbf{A}_{j+\frac12}
		=
		\mathbf{A}(\con_{j+\frac12})
		=
		\mathbf{R}_{j+\frac12}
		\mathbf{\Lambda}_{j+\frac12}
		\mathbf{L}_{j+\frac12}
	\end{align*}
	where $\mathbf{R}_{j+\frac12}$ and $\mathbf{L}_{j+\frac12}$ are the matrices of right and left eigenvectors, respectively.
	
	The conserved variables on the reconstruction stencil are projected onto the local characteristic variables,
	\begin{align*}
		\mathbf{w}_k
		=
		\mathbf{L}_{j+\frac12}\con_k,
		\qquad
		k=j-r+1,\ldots,j+r,
	\end{align*}
	and the WENO interpolation is then applied component-wise in the characteristic space,
	\begin{align*}
		\mathbf{w}^-_{j+\frac12}
		&=
		\mathrm{WENO}
		\left(
		\mathbf{w}_{j-r+1},
		\ldots,
		\mathbf{w}_{j+r-1}
		\right),\\
		\mathbf{w}^+_{j+\frac12}
		&=
		\mathrm{WENO}
		\left(
		\mathbf{w}_{j+r},
		\ldots,
		\mathbf{w}_{j-r+2}
		\right).
	\end{align*}
	Finally, the reconstructed interface values are transformed back to the physical space according to
	\begin{align*}
		\con^\pm_{j+\frac12}
		=
		\mathbf{R}_{j+\frac12}
		\mathbf{w}^\pm_{j+\frac12}.
	\end{align*}
	\subsection{Choice of WENO interpolation}
	In the previous subsection, we introduced the AFD-WENO framework for solving systems of hyperbolic conservation laws. The AFD-WENO framework requires the numerical flux at the cell interface, which is computed using reconstructed solution values obtained through WENO interpolation. Over the years, several WENO interpolation techniques have been proposed in the literature to improve the accuracy, robustness, and resolution of the original WENO formulation. In this section, we briefly review three widely used interpolation procedures: the classical \wenojs interpolation \cite{liu1994weighted}, the \wenoz interpolation \cite{borges2008improved}, and the WENO-AO(5,3) \cite{balsara2025efficient} interpolation. Furthermore, we propose an improved WENO-AOI(5,3) interpolation, which enhances the performance of the WENO-AO(5,3) framework while preserving its high-order accuracy and non-oscillatory properties.

	For polynomial interpolation, we employ the Legendre polynomial basis on the cell $I_i = [x_{i-\frac{1}{2}}, x_{i+\frac{1}{2}}]$ because of its orthogonality property (see \cite{balsara2025efficient} for more details). The Legendre polynomials up to degree four over the cell $I_i$ are defined by
	\begin{align*}
		\mathcal{L}_0(x) &=1,~~
		\mathcal{L}_1(x) = \left(\frac{x-x_i}{\Delta x}\right),\\
		\mathcal{L}_2(x) &= \left(\frac{x-x_i}{\Delta x}\right)^2-\frac{1}{12},\\
		\mathcal{L}_3(x) &= \left(\frac{x-x_i}{\Delta x}\right)^3-\frac{3}{20}\left(\frac{x-x_i}{\Delta x}\right),\\
		\mathcal{L}_4(x) &= \left(\frac{x-x_i}{\Delta x}\right)^4
		-\frac{3}{14}\left(\frac{x-x_i}{\Delta x}\right)^2
		+\frac{3}{560}.
	\end{align*}
	Consider the five-point stencil
	\begin{align*}
		\mathcal{S}_0^5=\{i-2,i-1,i,i+1,i+2\},
	\end{align*}
	and its associated three-point substencils
	\begin{align*}
		\mathcal{S}_0^3=\{i-2,i-1,i\},\qquad
		\mathcal{S}_1^3=\{i-1,i,i+1\},\qquad
		\mathcal{S}_2^3=\{i,i+1,i+2\}.
	\end{align*}
	Let $\mathbb{P}_0^5(x)$ denote the fifth-order reconstruction polynomial over the stencil $\mathcal{S}_0^5$, and let
	$\mathbb{P}_0^3(x)$,
	$\mathbb{P}_1^3(x)$,
	and
	$\mathbb{P}_2^3(x)$
	denote the third-order reconstruction polynomials on the sub-stencil
	$\mathcal{S}_0^3$,
	$\mathcal{S}_1^3$,
	and
	$\mathcal{S}_2^3$,
	respectively. These interpolation polynomials are expressed in terms of the Legendre basis as
	\begin{equation*}
		\left.
		\begin{aligned}
			\mathbb{P}_0^5(x)
			&=
			a_{00}^{5}\mathcal{L}_0(x)
			+a_{01}^{5}\mathcal{L}_1(x)
			+a_{02}^{5}\mathcal{L}_2(x)
			+a_{03}^{5}\mathcal{L}_3(x)
			+a_{04}^{5}\mathcal{L}_4(x),\\[1ex]
			\mathbb{P}_0^3(x)
			&=
			a_{00}^{3}\mathcal{L}_0(x)
			+a_{01}^{3}\mathcal{L}_1(x)
			+a_{02}^{3}\mathcal{L}_2(x),\\
			\mathbb{P}_1^3(x)
			&=
			a_{10}^{3}\mathcal{L}_0(x)
			+a_{11}^{3}\mathcal{L}_1(x)
			+a_{12}^{3}\mathcal{L}_2(x),\\
			\mathbb{P}_2^3(x)
			&=
			a_{20}^{3}\mathcal{L}_0(x)
			+a_{21}^{3}\mathcal{L}_1(x)
			+a_{22}^{3}\mathcal{L}_2(x),
		\end{aligned}
		\right\}
	\end{equation*}
	where the coefficients $a_{jk}^{r}$ are obtained by enforcing the interpolation conditions on polynomials. In order to measure the smoothness of function over a given stencil, we need a smoothness indicator. The smoothness indicators associated with the stencil $\mathcal{S}_k^m$ and polynomial interpolation $\mathbb{P}_k^m$ are defined as (see \cite{jiang1996efficient} for more details).
	\begin{equation*}
		\beta_k^m
		=
		\sum_{\ell=1}^{m-1}
		\Delta x^{2\ell-1}
		\int_{I_i}
		\left(
		\frac{d^\ell}{dx^\ell}\mathbb{P}_k^m(x)
		\right)^2\,dx,
	\end{equation*}
	We now present the construction of various WENO interpolation schemes based on Legendre polynomial expansions and the associated smoothness indicators.
	\subsection{WENO-JS5}\label{sec:wenojs}
	The WENO-JS5 interpolation at the cell interface $x_{i+\frac12}$ is defined as
	\begin{equation*}
		u_{i+\frac12}
		=
		\sum_{k=0}^{2}\omega_k^3\,\mathbb{P}_k^3(x_{i+\frac12}),
	\end{equation*}
	where $\mathbb{P}_k^3$, $k=0,1,2$, denote the quadratic interpolation
	polynomials constructed on the sub-stencils
	$\mathcal{S}_0^3$, $\mathcal{S}_1^3$, and $\mathcal{S}_2^3$, respectively. The nonlinear weights, denoted by $\omega_k^m$, are defined as
	\begin{equation*}
		\omega_k^3
		=
		\frac{\alpha_k^3}
		{\displaystyle\sum_{j=0}^{2}\alpha_j^3},
		\qquad k=0,1,2,
	\end{equation*}
	where
	\begin{equation*}
		\alpha_k^3
		=
		\frac{\gamma_k^3}
		{\left(\beta_k^3+\varepsilon\right)^p},
		\qquad k=0,1,2.
	\end{equation*}
	Here, $\varepsilon$ is a small positive number introduced to avoid division by
	zero, and the parameter $p=2$. The optimal linear weights, denoted by $\gamma_k^m$, are given by
	\begin{equation*}
		\gamma_0^3=\frac{1}{10},
		\qquad
		\gamma_1^3=\frac{6}{10},
		\qquad
		\gamma_2^3=\frac{3}{10},
	\end{equation*}
	which satisfy
	\begin{align*}
		\gamma_0^3+\gamma_1^3+\gamma_2^3=1.
	\end{align*}
	\subsection{WENO-Z5}\label{sec:wenoz}
	The WENO-Z5 interpolation at the cell interface $x_{i+\frac12}$ is defined as
	\begin{equation*}
		u_{i+\frac12}
		=
		\sum_{k=0}^{2}\omega_k^3\,\mathbb{P}_k^3(x_{i+\frac12}),
	\end{equation*}
	where $\mathbb{P}_k^3$, $k=0,1,2$, denote the quadratic interpolation
	polynomials constructed on the substencils
	$\mathcal{S}_0^3$, $\mathcal{S}_1^3$, and $\mathcal{S}_2^3$, respectively. The nonlinear weights are defined by
	\begin{equation*}
		\omega_k^3
		=
		\frac{\alpha_k^3}
		{\displaystyle\sum_{j=0}^{2}\alpha_j^3},
		\qquad k=0,1,2,
	\end{equation*}
	where
	\begin{equation*}
		\alpha_k^3
		=
		\gamma_k^3
		\left(
		1+
		\left(
		\frac{\tau}{\beta_k^3+\varepsilon}
		\right)^p
		\right),
		\qquad k=0,1,2.
	\end{equation*}
	Here, $\varepsilon$ is a small positive number introduced to avoid division by
	zero, and the parameter $p=2$. The optimal linear weights are same as in WENO-JS5 schemes  and $\tau$ is the global
	smoothness indicator. The global smoothness indicator is defined as
	\begin{equation*}
		\tau
		=
		|\beta_0^3-\beta_2^3|.
	\end{equation*}
	\subsection{WENO-AO(5,3)}\label{sec:wenoao}
	The WENO-AO reconstruction at the cell interface $x_{i+\frac12}$ is defined as (see \cite{balsara2025efficient} for more details)
	\begin{equation*}
		u_{i+\frac12}
		=
		\frac{\omega_0^5}{\gamma_0^5}
		\left(
		\mathbb{P}_0^5(x_{i+\frac12})
		-
		\sum_{k=0}^{2}\gamma_k^3\mathbb{P}_k^3(x_{i+\frac12})
		\right)
		+
		\sum_{k=0}^{2}\omega_k^3\mathbb{P}_k^3(x_{i+\frac12}),
	\end{equation*}
	where $\mathbb{P}_0^5$ denotes the fourth-degree polynomial constructed on the stencil
	$\mathcal{S}_0^5$, while
	$\mathbb{P}_k^3$, $k=0,1,2$, are the quadratic polynomials constructed on the sub-stencils
	$\mathcal{S}_k^3$, respectively. The linear weights satisfy
	\begin{equation*}
		\gamma_0^5+\sum_{k=0}^{2}\gamma_k^3=1,
		\qquad
		\gamma_0^5,\;\gamma_k^3>0.
	\end{equation*}
	The nonlinear weights are defined by
	\begin{equation*}
		\omega_0^5=\frac{\alpha_0^5}
		{\alpha_0^5+\displaystyle\sum_{k=0}^{2}\alpha_k^3},
		\qquad
		\omega_k^3=\frac{\alpha_k^3}
		{\alpha_0^5+\displaystyle\sum_{j=0}^{2}\alpha_j^3},
		\quad k=0,1,2,
	\end{equation*}
	where
	\begin{equation*}
		\alpha_0^5
		=
		\gamma_0^5
		\left(
		1+\left(\frac{\tau}{\beta_0^5+\varepsilon}\right)^2
		\right),
		\qquad
		\alpha_k^3
		=
		\gamma_k^3
		\left(
		1+\left(\frac{\tau}{\beta_k^3+\varepsilon}\right)^2
		\right),
		\quad k=0,1,2.
	\end{equation*}
	Here, $\varepsilon$ is a small positive number introduced to avoid division by zero,
	$\beta_0^5$ and $\beta_k^3$ denote the smoothness indicators corresponding to
	$\mathbb{P}_0^5$ and $\mathbb{P}_k^3$, respectively, and $\tau$ is the global
	smoothness indicator. The global smoothness indicator is defined as
	\begin{equation*}
		\tau
		=
		\frac{1}{3}
		\left(
		\left|\beta_0^5-\beta_0^3\right|
		+
		\left|\beta_0^5-\beta_1^3\right|
		+
		\left|\beta_0^5-\beta_2^3\right|
		\right).
	\end{equation*}
	\subsection{WENO-AOI}\label{sec:wenoaoi}
	The \wenoaoi interpolation has the same reconstruction structure as the \wenoao interpolation. Thus, the reconstructed value at the cell interface $x_{i+\frac12}$ is given by
	\begin{equation*}
		u_{i+\frac12}
		=
		\frac{\omega_0^5}{\gamma_0^5}
		\left(
		\mathbb{P}_0^5(x_{i+\frac12})
		-
		\sum_{k=0}^{2}\gamma_k^3\mathbb{P}_k^3(x_{i+\frac12})
		\right)
		+
		\sum_{k=0}^{2}\omega_k^3\mathbb{P}_k^3(x_{i+\frac12}),
	\end{equation*}
	where the nonlinear weights are computed in the same manner as in the \wenoao scheme, except that a different global smoothness indicator is employed. Specifically, the global smoothness indicator for the \wenoaoi scheme is defined by
	\begin{equation*}
		\tau
		=
		\left|
		\beta_0^5
		-
		\frac{1}{3}
		\left(
		\beta_0^3+\beta_1^3+\beta_2^3
		\right)
		\right|.
	\end{equation*}
	\section{Extension to Two Dimensions}\label{sec:extension_2D}
	
	Consider the two-dimensional computational domain
	$\Omega=[x_{\min},x_{\max}]\times[y_{\min},y_{\max}]$,
	which is partitioned into a uniform Cartesian mesh consisting of rectangular cells
	\begin{equation*}
		I_{i,j}
		=
		[x_{i-\frac12},x_{i+\frac12}]
		\times
		[y_{j-\frac12},y_{j+\frac12}],
	\end{equation*}
	with mesh sizes $\Delta x$ and $\Delta y$ in the $x$- and $y$-directions, respectively. Let $\con_{i,j}$ denote the approximation to $\con(t,x_i,y_j)$ at the grid point $(x_i,y_j)$.
	
	The extension of the AFD-WENO scheme to two dimensions is performed in a dimension-by-dimension manner by applying the one-dimensional reconstruction and flux evaluation independently in each coordinate direction. The resulting semi-discrete finite difference scheme for the two-dimensional system \eqref{SRHD} is
	\begin{equation}\label{semi_discrete_scheme_2D}
		\frac{d}{dt}\con_{i,j}
		+\frac{1}{\Delta x}
		\left(
		\hat{\mathbf{f}}_{i+\frac12,j}
		-
		\hat{\mathbf{f}}_{i-\frac12,j}
		\right)
		+\frac{1}{\Delta y}
		\left(
		\hat{\mathbf{g}}_{i,j+\frac12}
		-
		\hat{\mathbf{g}}_{i,j-\frac12}
		\right)
		=\mathbf{0}.
	\end{equation}
	Here, $\hat{\mathbf{f}}_{i+\frac12,j}$ and $\hat{\mathbf{g}}_{i,j+\frac12}$ denote the numerical fluxes in the $x$- and $y$-directions, respectively, obtained by applying the one-dimensional AFD-WENO reconstruction described in the previous section. The corresponding reconstructed left and right interface states are denoted by $\con^{\pm}_{i+\frac12,j}$ and $\con^{\pm}_{i,j+\frac12}$ in the $x$- and $y$-directions, respectively. This semi-discrete formulation forms the basis for the physical-constraint-preserving analysis presented in the next section.
	
	\section{Physical Constraint Preservation by AFD-WENO Schemes}
	\label{sec:pcp}
	
	This section develops the PCP framework for the AFD-WENO scheme. We begin by reformulating the admissible set $\mathcal{G}_{\prim}$ in terms of conservative variables and establishing its convexity. We then specify the wave-speed estimates and CFL condition. The complete algorithmic procedure, including the state-limiting and two-stage flux-limiting procedures, is presented next, followed by the supporting lemmas and the main constraint-preservation theorem. The section closes with a remark on the necessity of the CFL bound, which is made transparent by the proof of the theorem.
	
	\subsection{Admissible Set in Conservative Variables}
	\label{subsec:admissible_set}
	
	The physically admissible set $\mathcal{G}_{\prim}$ was defined in~\eqref{eq:admissible_set_prim} in terms of the primitive variables. For the development of PCP schemes it is essential to characterise admissibility directly in terms of the conservative variables. To this end we define the \emph{$q$-function}
	\begin{equation}
		\label{eq:q_function}
		q(\con) := E - \sqrt{D^2 + m_x^2 + m_y^2},
	\end{equation}
	and the conservative admissible set
	\begin{equation}
		\label{eq:admissible_set}
		\mathcal{G}
		=
		\Bigl\{
		\con = (D, m_x, m_y, E)^T \in \mathbb{R}^4
		:
		D > 0
		\quad\text{and}\quad
		q(\con) > 0
		\Bigr\}.
	\end{equation}
	The condition $D > 0$ ensures positivity of the conserved mass density, and together with $q(\con) > 0$ it encodes pressure positivity and the subluminal velocity constraint, as confirmed by the following proposition.
	
	\begin{proposition}[Equivalence of admissible sets \cite{wu2015high,wu2017physical}]
		\label{prop:equivalence}
		The admissible set $\mathcal{G}$ defined in~\eqref{eq:admissible_set} is equivalent to the physically admissible set $\mathcal{G}_{\prim}$ defined in~\eqref{eq:admissible_set_prim}.
	\end{proposition}
	
	\begin{remark}
		Proposition~\ref{prop:equivalence} confirms that $\mathcal{G}$ is the correct conservative-variable reformulation of the physical admissibility conditions. In particular, $q(\con) > 0$ serves as a single scalar
		certificate for both pressure positivity and the subluminal velocity constraint, making it the natural quantity to monitor and preserve in the numerical scheme.
	\end{remark}
	
	\begin{lemma}[Convexity of $\mathcal{G}$]
		\label{lem:convexity}
		The admissible set $\mathcal{G}$ is convex: for any $\con^A, \con^B \in \mathcal{G}$ and $\theta \in [0,1]$,
		\begin{align*}
			\con^\theta
			:=
			\theta\,\con^A + (1-\theta)\,\con^B
			\;\in\;
			\mathcal{G}.
		\end{align*}
	\end{lemma}
	
	\begin{proof}
		Linearity of $D$ gives
		$D(\con^\theta) = \theta D(\con^A) + (1-\theta) D(\con^B) > 0$. Since $E$ is linear in $\con$ and $\sqrt{D^2 + m_x^2 + m_y^2}$ is convex, the $q$-function defined in~\eqref{eq:q_function} is concave. Jensen's
		inequality for concave functions then yields
		\begin{align*}
			q(\con^\theta)
			\ge
			\theta\,q(\con^A) + (1-\theta)\,q(\con^B)
			> 0.
		\end{align*}
		Hence $\con^\theta \in \mathcal{G}$.
	\end{proof}
	
	\begin{remark}
		\label{rem:convexity_role}
		The convexity of $\mathcal{G}$ implies that any convex combination of admissible states is again admissible. This property underpins both the state-limiting procedure in Step~2 of the algorithm below and the
		constraint-preservation proof for the multi-dimensional update in Theorem~\ref{thm:main}.
	\end{remark}
	Before proceeding to the algorithmic details, we introduce the following notations.
	\subsection{Wave-Speed Estimates and Grid Ratios}
	\label{subsec:wave_speeds}
	
	Let $S_x(\con)$ and $S_y(\con)$ denote the spectral radii of the flux Jacobians $\partial\fluxx/\partial\con$ and $\partial\fluxy/\partial\con$, respectively, i.e.\ the maximum absolute eigenvalues in each coordinate direction. At each time level $t^n$ we define the \emph{global} wave-speed estimates
	\begin{equation}
		\label{eq:wave_speeds}
		\alpha_x = \max_{i,j}\, S_x\!\left(\con_{i,j}^n\right),
		\qquad
		\alpha_y = \max_{i,j}\, S_y\!\left(\con_{i,j}^n\right),
	\end{equation}
	and the directional splitting weights
	\begin{equation*}
		\beta_x = \frac{\alpha_x/\Delta x}{\alpha_x/\Delta x + \alpha_y/\Delta y},
		\qquad
		\beta_y = \frac{\alpha_y/\Delta y}{\alpha_x/\Delta x + \alpha_y/\Delta y},
	\end{equation*}
	which satisfy $\beta_x + \beta_y = 1$. The time step $\Delta t$ is chosen to satisfy the CFL condition
	\begin{equation}
		\label{eq:cfl}
		\mathrm{CFL}
		:=
		\Delta t \left( \frac{\alpha_x}{\Delta x} + \frac{\alpha_y}{\Delta y} \right)
		\le \frac{1}{2}.
	\end{equation}
	We define the mesh ratios $\lambda_x, \lambda_y$ and the scaled parameters $\Lambda_x, \Lambda_y$ as
	\begin{equation*}
		\lambda_x = \frac{\Delta t}{\Delta x}, \quad \lambda_y = \frac{\Delta t}{\Delta y},
		\qquad
		\Lambda_x = \frac{\lambda_x}{\beta_x} = \frac{\mathrm{CFL}}{\alpha_x}, \quad \Lambda_y = \frac{\lambda_y}{\beta_y} = \frac{\mathrm{CFL}}{\alpha_y}.
	\end{equation*}
	
	\begin{remark}[Local vs.\ global wave speeds]
		\label{rem:local_global}
		The global wave-speed estimates~\eqref{eq:wave_speeds} define the uniform time step $\Delta t$ and scaling parameters $\Lambda_x, \Lambda_y$, ensuring that the global CFL bound~\eqref{eq:cfl} holds. In the local Lax--Friedrichs (LLF) flux~\eqref{eq:llf_x}--\eqref{eq:llf_y}, the local wave speeds $\alpha_{x,i+\frac{1}{2},j}$ and $\alpha_{y,i,j+\frac{1}{2}}$ are employed. As established in Lemma~\ref{lem:llf_positivity}, since $\alpha_{x,i+\frac{1}{2},j} \le \alpha_x$, the local coefficient satisfies $\mu_{i+\frac{1}{2},j} \le 2\,\mathrm{CFL} \le 1$, which rigorously guarantees physical admissibility of the LLF candidate states.
	\end{remark}
	
	\subsection{The PCP Algorithm}
	\label{subsec:pcp_algorithm}
	
	The complete numerical procedure at each time step is described below. All notation refers to the grid point $(i,j)$ and time level $t^n$. Operations in the $y$-direction are symmetric to those in the $x$-direction and are stated concisely.
	
	\subsubsection*{Step~1: High-Order Interface Reconstruction}
	
	From the grid-point values $\{\con_{i,j}^n\}$, apply the WENO interpolation component-wise (or via local characteristic decomposition; see Section~\ref{sec:afd_weno}) to obtain the left and right reconstructed
	interface states
	\begin{align*}
		\con^{\mathrm{rec},-}_{i+\frac{1}{2},j},\quad
		\con^{\mathrm{rec},+}_{i+\frac{1}{2},j},\quad
		\con^{\mathrm{rec},-}_{i,j+\frac{1}{2}},\quad
		\con^{\mathrm{rec},+}_{i,j+\frac{1}{2}}.
	\end{align*}
	
	\subsubsection*{Step~2: State-Limiting Procedure}
	
	High-order reconstruction does not in general preserve admissibility of the interface states. We therefore limit each reconstructed state toward a provably admissible \emph{anchor state} before flux evaluation.
	
	Define the anchor states as arithmetic averages of neighboring grid-point values:
	\begin{equation*}
		\con^A_{i+\frac{1}{2},j}
		= \tfrac{1}{2}\!\left(\con_{i,j}^n + \con_{i+1,j}^n\right),
		\qquad
		\con^A_{i,j+\frac{1}{2}}
		= \tfrac{1}{2}\!\left(\con_{i,j}^n + \con_{i,j+1}^n\right).
	\end{equation*}
	By Lemma~\ref{lem:convexity}, $\con^A \in \mathcal{G}$ whenever the neighboring grid-point values are admissible. Each reconstructed state $\con^{\mathrm{rec}}$ is then limited toward $\con^A$ in two stages to obtain the \emph{PCP-limited interface state} $\con^{\mathrm{pcp}}$:
	
	\begin{enumerate}[label=(\roman*)]
		\item \textbf{$D$-stage.}
		Set $\con^D = \theta_D\,\con^{\mathrm{rec}} + (1-\theta_D)\,\con^A$, where $\theta_D \in [0,1]$ is the largest value such that $D(\con^D) \ge \varepsilon_D > 0$.
		
		\item \textbf{$q$-constraint stage.}
		Set $\con^{\mathrm{pcp}} = \theta_Q\,\con^D + (1-\theta_Q)\,\con^A$, where $\theta_Q \in [0,1]$ is the largest value such that $q(\con^{\mathrm{pcp}}) \ge \varepsilon_q > 0$.
	\end{enumerate}
	
	Here $\varepsilon_D$ and $\varepsilon_q$ are small user-specified tolerances (typically $10^{-13}$). The $D$-component limiting parameter $\theta_D$ is computed analytically by linear interpolation:
	\begin{equation*}
		\theta_D = \begin{cases}
			\frac{D(\con^A) - \varepsilon_D}{D(\con^A) - D(\con^{\mathrm{rec}})}, & \text{if } D(\con^{\mathrm{rec}}) < \varepsilon_D, \\
			1, & \text{otherwise}.
		\end{cases}
	\end{equation*}
	Since the relation $q(\con^{\mathrm{pcp}}) \ge \varepsilon_q$ is nonlinear, the $q$-constraint limiting parameter $\theta_Q$ is computed using a bisection search in the interval $[0, 1]$ to solve $q(\theta \con^{D} + (1-\theta)\con^A) - \varepsilon_q = 0$. The PCP-limited states $\con^{\mathrm{pcp},\pm}$ replace the reconstructed states $\con^{\mathrm{rec},\pm}$ in all subsequent steps.
	
	\subsubsection*{Step~3: Low-Order Local Lax--Friedrichs Flux}
	
	Evaluate the low-order LLF numerical fluxes using the grid-point values:
	\begin{align}
		\hat{\fluxx}^{\mathrm{LLF}}_{i+\frac{1}{2},j}
		&= \tfrac{1}{2}\!\left[
		\fluxx\!\left(\con_{i,j}^n\right)
		+ \fluxx\!\left(\con_{i+1,j}^n\right)
		- \alpha_{x,i+\frac{1}{2},j}
		\left(
		\con_{i+1,j}^n
		- \con_{i,j}^n
		\right)
		\right], \label{eq:llf_x} \\
		\hat{\fluxy}^{\mathrm{LLF}}_{i,j+\frac{1}{2}}
		&= \tfrac{1}{2}\!\left[
		\fluxy\!\left(\con_{i,j}^n\right)
		+ \fluxy\!\left(\con_{i,j+1}^n\right)
		- \alpha_{y,i,j+\frac{1}{2}}
		\left(
		\con_{i,j+1}^n
		- \con_{i,j}^n
		\right)
		\right], \label{eq:llf_y}
	\end{align}
	where $\alpha_{x,i+\frac{1}{2},j}$ and $\alpha_{y,i,j+\frac{1}{2}}$ are local wave-speed estimates at each interface (see Remark~\ref{rem:local_global}).
	
	\subsubsection*{Step~4: High-Order AFD-WENO Flux}
	
	Evaluate the LLF flux on the PCP-limited reconstructed interface states, and add the high-order correction term (Section~\ref{sec:afd_weno}) to obtain the high-order AFD-WENO fluxes:
	\begin{align*}
		\hat{\fluxx}^{\mathrm{AWENO}}_{i+\frac{1}{2},j}
		&= \hat{\fluxx}^{\mathrm{LLF,rec}}_{i+\frac{1}{2},j}
		+ \hat{\fluxx}^{\mathrm{cor}}_{i+\frac{1}{2},j}, \\
		\hat{\fluxy}^{\mathrm{AWENO}}_{i,j+\frac{1}{2}}
		&= \hat{\fluxy}^{\mathrm{LLF,rec}}_{i,j+\frac{1}{2}}
		+ \hat{\fluxy}^{\mathrm{cor}}_{i,j+\frac{1}{2}}, \end{align*}
	where $\hat{\fluxx}^{\mathrm{LLF,rec}}_{i+\frac{1}{2},j}$ and $\hat{\fluxy}^{\mathrm{LLF,rec}}_{i,j+\frac{1}{2}}$ are the LLF fluxes evaluated on the PCP-limited interface states (e.g., using $\con^{\mathrm{pcp},\pm}$ instead of grid-point values in~\eqref{eq:llf_x}--\eqref{eq:llf_y}). The correction terms $\hat{\fluxx}^{\mathrm{cor}}$ and $\hat{\fluxy}^{\mathrm{cor}}$ are computed solely from the physical flux values at the grid points and carry no admissibility guarantee on their own. The high-order AFD-WENO fluxes may therefore produce inadmissible forward-Euler candidate states, which motivates the flux-limiting step that follows.
	
	\subsubsection*{Step~5: Two-Stage Flux Limiter}
	
	For each interface, the AFD-WENO flux is blended toward the LLF flux by the
	smallest amount necessary to restore admissibility of the one-sided
	forward-Euler candidate states. We describe the procedure for the
	$x$-direction; the $y$-direction is handled symmetrically.
	
	For the interface between grid points $i$ and $i{+}1$ (suppressing the
	$j$-index), define the \emph{one-sided forward-Euler candidate states}
	associated with a generic interface flux $\mathbf{F}$:
	\begin{equation}
		\label{eq:euler_states}
		\con^+[\mathbf{F}]
		= \con_{i,j}^n   - 2\Lambda_x\,\mathbf{F},
		\qquad
		\con^-[\mathbf{F}]
		= \con_{i+1,j}^n + 2\Lambda_x\,\mathbf{F}.
	\end{equation}
	Denote $\con^{\pm,W} = \con^\pm[\hat{\fluxx}^{\mathrm{AWENO}}]$ and
	$\con^{\pm,L} = \con^\pm[\hat{\fluxx}^{\mathrm{LLF}}]$.
	
	The key observation is that the LLF flux computed from the grid-point values
	$\con_{i,j}^n$ and $\con_{i+1,j}^n$ produces admissible candidate
	states $\con^{\pm,L} \in \mathcal{G}$ under the CFL condition~\eqref{eq:cfl},
	as proved in Lemma~\ref{lem:llf_positivity}. The flux limiter is designed
	to blend the high-order AFD-WENO flux toward this LLF flux to restore admissibility
	of the candidate states. We first choose sufficiently small positive numbers $\tilde{\varepsilon}_D$ and $\tilde{\varepsilon}_q$ such that $D(\con^{\pm,L})\geq\tilde{\varepsilon}_D$ and $q(\con^{\pm,L})\geq\tilde{\varepsilon}_q$ for all interfaces ($10^{-13}$ is used in the numerical simulations). This is possible because Lemma~\ref{lem:llf_positivity} guarantees that for the LLF flux $D(\con^{\pm,L})> 0$ and $q(\con^{\pm,L})> 0$. We then have the following two-stage flux-limiting procedure.
	
	\medskip
	\noindent\textbf{Stage~I ($D$).}
	Construct the intermediate flux by limiting only the $D$-component of
	the flux vector, leaving the momentum and energy components at their
	high-order AFD-WENO values:
	\begin{equation}
		\label{eq:fd_flux}
		\hat{\fluxx}^D
		=
		\Bigl(
		(1-\theta_D)\,\hat{f}^{\mathrm{LLF}}_{[0]}
		+ \theta_D\,\hat{f}^{\mathrm{AWENO}}_{[0]},\;\;
		\hat{f}^{\mathrm{AWENO}}_{[1]},\;\;
		\hat{f}^{\mathrm{AWENO}}_{[2]},\;\;
		\hat{f}^{\mathrm{AWENO}}_{[3]}
		\Bigr)^T,
	\end{equation}
	where subscript $[\,\cdot\,]$ denotes the corresponding component and
	\begin{equation}
		\label{eq:thetaD_flux}
		\theta_D
		=
		\min\!\Biggl(1,\;
		\Bigl[
		\frac{D(\con^{+,L}) - \tilde{\varepsilon}_D}
		{D(\con^{+,L}) - D(\con^{+,W})}
		\Bigr]_{D(\con^{+,W}) < \tilde{\varepsilon}_D},\;
		\Bigl[
		\frac{D(\con^{-,L}) - \tilde{\varepsilon}_D}
		{D(\con^{-,L}) - D(\con^{-,W})}
		\Bigr]_{D(\con^{-,W}) < \tilde{\varepsilon}_D}
		\Biggr).
	\end{equation}
	Denote $\con^{\pm,D} = \con^\pm[\hat{\fluxx}^D]$.
	
	\medskip
	\noindent\textbf{Stage~II ($q$-constraint).}
	Define the final PCP flux by blending $\hat{\fluxx}^D$ with the LLF flux:
	\begin{equation}
		\label{eq:pcp_flux_final}
		\hat{\fluxx}^{\mathrm{PCP}}
		=
		(1-\theta_Q)\,\hat{\fluxx}^{\mathrm{LLF}}
		+ \theta_Q\,\hat{\fluxx}^D,
	\end{equation}
	where
	\begin{equation}
		\label{eq:thetaQ}
		\theta_Q
		=
		\min\!\Biggl(1,\;
		\Bigl[
		\frac{q(\con^{+,L}) - \tilde{\varepsilon}_q}
		{q(\con^{+,L}) - q(\con^{+,D})}
		\Bigr]_{q(\con^{+,D}) < \tilde{\varepsilon}_q},\;
		\Bigl[
		\frac{q(\con^{-,L}) - \tilde{\varepsilon}_q}
		{q(\con^{-,L}) - q(\con^{-,D})}
		\Bigr]_{q(\con^{-,D}) < \tilde{\varepsilon}_q}
		\Biggr).
	\end{equation}
	A symmetric procedure yields $\hat{\fluxy}^{\mathrm{PCP}}$.
	
	\subsubsection*{Step~6: Conservative Update}
	Finally, we can discretize the semi-discrete scheme~\eqref{semi_discrete_scheme_2D} in time using a time-stepping method that preserves admissibility. For simplicity, we present the update using the forward-Euler method. Specifically, the grid-point values are advanced to time level $t^{n+1}$ by
	\begin{equation}
		\label{eq:update}
		\con_{i,j}^{n+1}
		=
		\con_{i,j}^n
		- \frac{\Delta t}{\Delta x}\!\left(
		\hat{\fluxx}^{\mathrm{PCP}}_{i+\frac{1}{2},j}
		- \hat{\fluxx}^{\mathrm{PCP}}_{i-\frac{1}{2},j}
		\right)
		- \frac{\Delta t}{\Delta y}\!\left(
		\hat{\fluxy}^{\mathrm{PCP}}_{i,j+\frac{1}{2}}
		- \hat{\fluxy}^{\mathrm{PCP}}_{i,j-\frac{1}{2}}
		\right).
	\end{equation}
	An overview of the complete algorithm is provided in Figure~\ref{fig:pcp_flowchart}.
	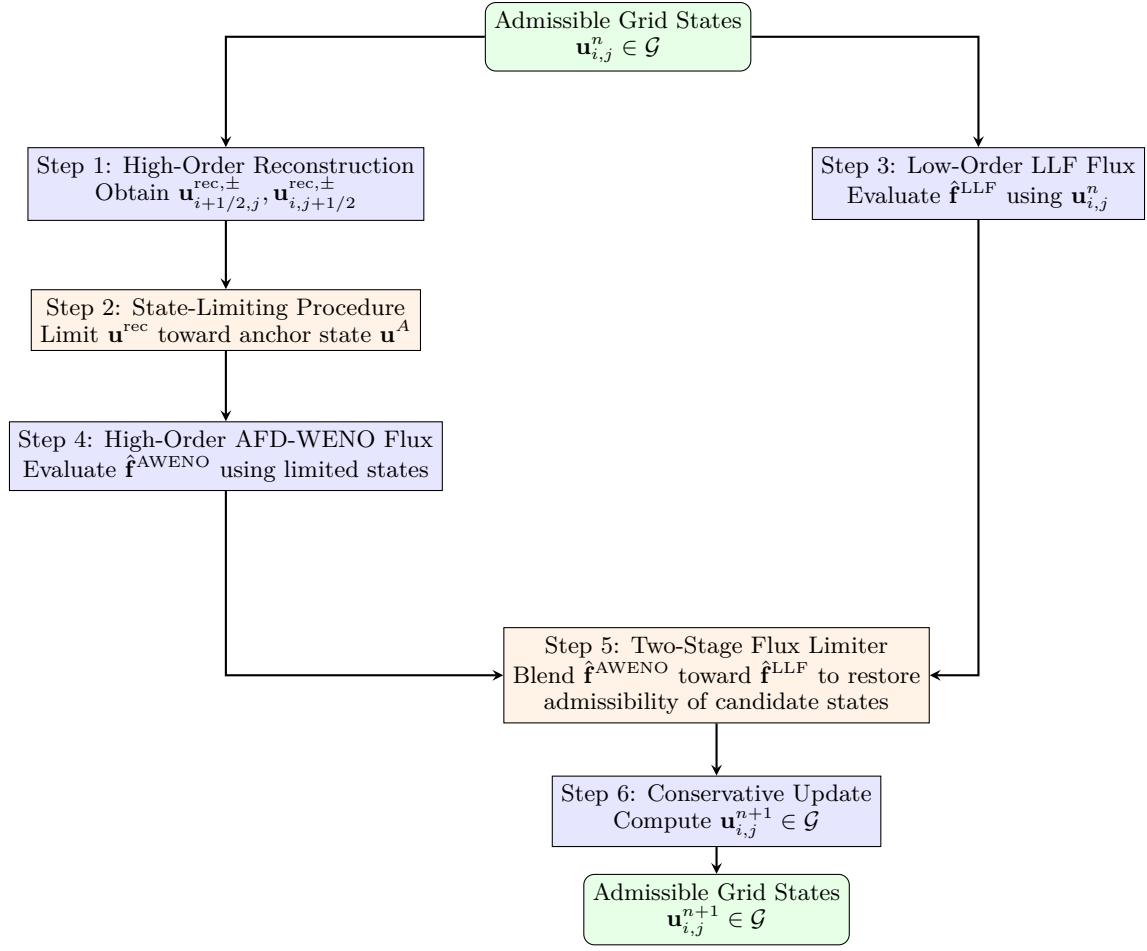
\begin{figure}[htbp]
		\centering
		\begin{tikzpicture}[
			node distance=1.5cm and 2.5cm,
			align=center,
			every node/.style={font=\footnotesize},
			startstop/.style={rectangle, rounded corners, minimum width=3cm, minimum height=0.7cm, text centered, draw=black, fill=green!10},
			process/.style={rectangle, minimum width=3.8cm, minimum height=0.7cm, text centered, draw=black, fill=blue!10},
			limiter/.style={rectangle, minimum width=3.8cm, minimum height=0.7cm, text centered, draw=black, fill=orange!10},
			arrow/.style={thick,->,>=stealth}
			]
			
			\node (start) [startstop] {Admissible Grid States\\$\con_{i,j}^n \in \mathcal{G}$};
			
			\node (step1) [process, below left=1.0cm and 0.8cm of start] {Step 1: High-Order Reconstruction\\Obtain $\con^{\mathrm{rec},\pm}_{i+1/2,j}, \con^{\mathrm{rec},\pm}_{i,j+1/2}$};
			\node (step2) [limiter, below of=step1, node distance=1.8cm] {Step 2: State-Limiting Procedure\\Limit $\con^{\mathrm{rec}}$ toward anchor state $\con^A$};
			\node (step4) [process, below of=step2, node distance=1.8cm] {Step 4: High-Order AFD-WENO Flux\\Evaluate $\hat{\fluxx}^{\mathrm{AWENO}}$ using limited states};
			
			\node (step3) [process, below right=1.0cm and 0.8cm of start] {Step 3: Low-Order LLF Flux\\Evaluate $\hat{\fluxx}^{\mathrm{LLF}}$ using $\con_{i,j}^n$};
			
			\node (step5) [limiter, below right=1.8cm and 0.8cm of step4] {Step 5: Two-Stage Flux Limiter\\Blend $\hat{\fluxx}^{\mathrm{AWENO}}$ toward $\hat{\fluxx}^{\mathrm{LLF}}$ to restore\\admissibility of candidate states};
			
			\node (step6) [process, below of=step5, node distance=1.8cm] {Step 6: Conservative Update\\Compute $\con_{i,j}^{n+1} \in \mathcal{G}$};
			
			\node (end) [startstop, below of=step6, node distance=1.3cm] {Admissible Grid States\\$\con_{i,j}^{n+1} \in \mathcal{G}$};
			
			\draw [arrow] (start) -| (step1.north);
			\draw [arrow] (start) -| (step3.north);
			\draw [arrow] (step1) -- (step2);
			\draw [arrow] (step2) -- (step4);
			\draw [arrow] (step3.south) |- (step5.east);
			\draw [arrow] (step4.south) |- (step5.west);
			\draw [arrow] (step5) -- (step6);
			\draw [arrow] (step6) -- (end);
			
		\end{tikzpicture}
		\caption{Flowchart of the Physical Constraint Preserving (PCP) algorithm for AFD-WENO schemes.}
		\label{fig:pcp_flowchart}
	\end{figure}
	\subsection{Theoretical Analysis}
	\label{subsec:theory}
	
	We now prove that the algorithm preserves $\mathcal{G}$ at every time step.
	The argument proceeds through three lemmas before culminating in the main
	theorem.
	
	\begin{lemma}[LLF forward-Euler positivity]
		\label{lem:llf_positivity}
		Suppose $\mathrm{CFL} \le \tfrac{1}{2}$ and the grid-point values $\con_{i,j}^n \in \mathcal{G}$ for all $(i,j)$. Let the local wave speed satisfy $\alpha_{x,i+\frac{1}{2},j} = \max\left(S_x(\con_{i,j}^n),\, S_x(\con_{i+1,j}^n)\right) \le \alpha_x$. Then the one-sided forward-Euler candidate states $\con^{\pm,L}$ computed with the local Lax--Friedrichs (LLF) flux~\eqref{eq:llf_x} belong to $\mathcal{G}$.
	\end{lemma}
	
	\begin{proof}
		Define the local parameter $\mu_{i+\frac{1}{2},j} := 2\Lambda_x\,\alpha_{x,i+\frac{1}{2},j} = 2\,\mathrm{CFL}\left(\frac{\alpha_{x,i+\frac{1}{2},j}}{\alpha_x}\right)$. Since $\alpha_{x,i+\frac{1}{2},j} \le \alpha_x$ and $\mathrm{CFL} \le \frac{1}{2}$, we have $0 \le \mu_{i+\frac{1}{2},j} \le 2\,\mathrm{CFL} \le 1$. Substituting the low-order LLF flux~\eqref{eq:llf_x} (evaluated on grid-point values) into $\con^{+,L} = \con_{i,j}^n - 2\Lambda_x\,\hat{\fluxx}^{\mathrm{LLF}}_{i+\frac{1}{2},j}$ gives
		\begin{align*}
			\con^{+,L}
			=
			(1-\mu_{i+\frac{1}{2},j})\,\con_{i,j}^n
			+ \frac{\mu_{i+\frac{1}{2},j}}{2}\underbrace{\!\left(
				\con_{i,j}^n
				- \frac{\fluxx(\con_{i,j}^n)}{\alpha_{x,i+\frac{1}{2},j}}
				\right)}_{\displaystyle=:\,\con^{\ominus}_{i,j}}
			+ \frac{\mu_{i+\frac{1}{2},j}}{2}\underbrace{\!\left(
				\con_{i+1,j}^n
				- \frac{\fluxx(\con_{i+1,j}^n)}{\alpha_{x,i+\frac{1}{2},j}}
				\right)}_{\displaystyle=:\,\con^{\ominus}_{i+1,j}}.
		\end{align*}
		The coefficients $(1-\mu_{i+\frac{1}{2},j})$, $\mu_{i+\frac{1}{2},j}/2$, $\mu_{i+\frac{1}{2},j}/2$ are non-negative and sum to one, so $\con^{+,L}$ is a convex combination of $\con_{i,j}^n$, $\con^{\ominus}_{i,j}$, and $\con^{\ominus}_{i+1,j}$. Under the local wave-speed bound $\alpha_{x,i+\frac{1}{2},j} \ge \max\left(S_x(\con_{i,j}^n), S_x(\con_{i+1,j}^n)\right)$, the states $\con \pm \fluxx(\con)/\alpha_{x,i+\frac{1}{2},j}$ lie in $\mathcal{G}$ whenever $\con \in \mathcal{G}$~\cite{wu2017physical}; hence $\con^{\ominus}_{i,j}, \con^{\ominus}_{i+1,j} \in \mathcal{G}$. Since $\con_{i,j}^n \in \mathcal{G}$ by assumption, Lemma~\ref{lem:convexity} implies $\con^{+,L} \in \mathcal{G}$.
		
		Similarly, substituting the LLF flux into $\con^{-,L} = \con_{i+1,j}^n + 2\Lambda_x\,\hat{\fluxx}^{\mathrm{LLF}}_{i+\frac{1}{2},j}$ yields
		\begin{align*}
			\con^{-,L}
			=
			(1-\mu_{i+\frac{1}{2},j})\,\con_{i+1,j}^n
			+ \frac{\mu_{i+\frac{1}{2},j}}{2}\underbrace{\!\left(
				\con_{i,j}^n
				+ \frac{\fluxx(\con_{i,j}^n)}{\alpha_{x,i+\frac{1}{2},j}}
				\right)}_{\displaystyle=:\,\con^{\oplus}_{i,j}}
			+ \frac{\mu_{i+\frac{1}{2},j}}{2}\underbrace{\!\left(
				\con_{i+1,j}^n
				+ \frac{\fluxx(\con_{i+1,j}^n)}{\alpha_{x,i+\frac{1}{2},j}}
				\right)}_{\displaystyle=:\,\con^{\oplus}_{i+1,j}},
		\end{align*}
		which is a convex combination of $\con_{i+1,j}^n$, $\con^{\oplus}_{i,j}$, and $\con^{\oplus}_{i+1,j}$. Since these states all lie in $\mathcal{G}$, Lemma~\ref{lem:convexity} implies $\con^{-,L} \in \mathcal{G}$.
	\end{proof}
	
	\begin{lemma}[Stage~I: $D$ preservation]
		\label{lem:stage1}
		The intermediate candidate states
		$\con^{\pm,D} = \con^\pm[\hat{\fluxx}^D]$ satisfy
		$D(\con^{\pm,D}) \ge \tilde{\varepsilon}_D$.
	\end{lemma}
	
	\begin{proof}
		Since $\hat{\fluxx}^D$ modifies only the $D$-component of the flux
		(see~\eqref{eq:fd_flux}), the $D$-component of $\con^{+,D}$ satisfies
		\begin{align*}
			D(\con^{+,D})
			=
			(1-\theta_D)\,D(\con^{+,L}) + \theta_D\,D(\con^{+,W}).
		\end{align*}
		If $D(\con^{+,W}) \ge \tilde{\varepsilon}_D$, then $\theta_D = 1$ and the result is
		immediate. Otherwise, substituting the explicit formula for $\theta_D$
		from~\eqref{eq:thetaD_flux} yields $D(\con^{+,D}) = \tilde{\varepsilon}_D$.
		An identical argument applies to $\con^{-,D}$.
	\end{proof}
	
	\begin{lemma}[Stage~II: $q$-constraint and $D$ preservation]
		\label{lem:stage2}
		The final candidate states
		$\con^{\pm,\mathrm{PCP}} = \con^\pm[\hat{\fluxx}^{\mathrm{PCP}}]$
		satisfy $q(\con^{\pm,\mathrm{PCP}}) \ge \tilde{\varepsilon}_q$ and
		$D(\con^{\pm,\mathrm{PCP}}) \ge \tilde{\varepsilon}_D$.
	\end{lemma}
	
	\begin{proof}
		From~\eqref{eq:euler_states} and~\eqref{eq:pcp_flux_final},
		\begin{align*}
			\con^{+,\mathrm{PCP}}
			=
			(1-\theta_Q)\,\con^{+,L} + \theta_Q\,\con^{+,D}.
		\end{align*}
		\textit{$q$-constraint.}
		If $q(\con^{+,D}) \ge \tilde{\varepsilon}_q$, then $\theta_Q = 1$ and the claim
		follows immediately. Otherwise $\theta_Q \in [0,1)$ and the concavity of $q$
		(Lemma~\ref{lem:convexity}) yields
		\begin{align*}
			q(\con^{+,\mathrm{PCP}})
			\ge
			(1-\theta_Q)\,q(\con^{+,L}) + \theta_Q\,q(\con^{+,D})
			=
			\tilde{\varepsilon}_q,
		\end{align*}
		where the equality is obtained by substituting $\theta_Q$
		from~\eqref{eq:thetaQ}.
		
		\textit{$D$-component.}
		Since $D(\con^{+,L}) \ge \tilde{\varepsilon}_D$	and $D(\con^{+,D}) \ge \tilde{\varepsilon}_D$ (Lemma~\ref{lem:stage1}), their convex
		combination satisfies $D(\con^{+,\mathrm{PCP}}) \ge \tilde{\varepsilon}_D$. The
		argument for $\con^{-,\mathrm{PCP}}$ is identical.
	\end{proof}
	
	\begin{theorem}[Constraint preservation]
		\label{thm:main}
		Suppose the grid-point values $\con_{i,j}^n \in \mathcal{G}$ for all
		$(i,j)$ and $\mathrm{CFL} \le \tfrac{1}{2}$. Then the updated grid-point
		values $\con_{i,j}^{n+1}$ produced by~\eqref{eq:update} with the two-stage
		PCP flux limiter also belong to $\mathcal{G}$.
	\end{theorem}
	
	\begin{proof}
		Using the relations $\lambda_x = \beta_x \Lambda_x$ and $\lambda_y = \beta_y \Lambda_y$, we define the four one-sided forward-Euler states for grid point $(i,j)$:
		\begin{align*}
			\con^{+,\mathrm{PCP}}_x
			&:= \con_{i,j}^n - 2\Lambda_x\,\hat{\fluxx}^{\mathrm{PCP}}_{i+\frac{1}{2},j},
			&
			\con^{-,\mathrm{PCP}}_x
			&:= \con_{i,j}^n + 2\Lambda_x\,\hat{\fluxx}^{\mathrm{PCP}}_{i-\frac{1}{2},j},\\
			\con^{+,\mathrm{PCP}}_y
			&:= \con_{i,j}^n - 2\Lambda_y\,\hat{\fluxy}^{\mathrm{PCP}}_{i,j+\frac{1}{2}},
			&
			\con^{-,\mathrm{PCP}}_y
			&:= \con_{i,j}^n + 2\Lambda_y\,\hat{\fluxy}^{\mathrm{PCP}}_{i,j-\frac{1}{2}}.
		\end{align*}
		By Lemma~\ref{lem:stage2}, all four states belong to $\mathcal{G}$. Substituting these definitions into the conservative update~\eqref{eq:update} yields:
		\begin{align*}
			\con_{i,j}^{n+1}
			&= \con_{i,j}^n
			- \beta_x \Lambda_x \left( \hat{\fluxx}^{\mathrm{PCP}}_{i+\frac{1}{2},j} - \hat{\fluxx}^{\mathrm{PCP}}_{i-\frac{1}{2},j} \right)
			- \beta_y \Lambda_y \left( \hat{\fluxy}^{\mathrm{PCP}}_{i,j+\frac{1}{2}} - \hat{\fluxy}^{\mathrm{PCP}}_{i,j-\frac{1}{2}} \right) \\
			&= \con_{i,j}^n
			- \frac{\beta_x}{2} \left( 2\Lambda_x\,\hat{\fluxx}^{\mathrm{PCP}}_{i+\frac{1}{2},j} \right)
			+ \frac{\beta_x}{2} \left( 2\Lambda_x\,\hat{\fluxx}^{\mathrm{PCP}}_{i-\frac{1}{2},j} \right)
			- \frac{\beta_y}{2} \left( 2\Lambda_y\,\hat{\fluxy}^{\mathrm{PCP}}_{i,j+\frac{1}{2}} \right)
			+ \frac{\beta_y}{2} \left( 2\Lambda_y\,\hat{\fluxy}^{\mathrm{PCP}}_{i,j-\frac{1}{2}} \right).
		\end{align*}
		By grouping terms, this can be rewritten as the convex combination
		\begin{equation}
			\label{eq:convex_update}
			\con_{i,j}^{n+1}
			=
			\frac{\beta_x}{2}\,\con^{+,\mathrm{PCP}}_x
			+ \frac{\beta_x}{2}\,\con^{-,\mathrm{PCP}}_x
			+ \frac{\beta_y}{2}\,\con^{+,\mathrm{PCP}}_y
			+ \frac{\beta_y}{2}\,\con^{-,\mathrm{PCP}}_y.
		\end{equation}
		The coefficients $\beta_x/2$, $\beta_x/2$, $\beta_y/2$, $\beta_y/2$ are
		non-negative and sum to $\beta_x + \beta_y = 1$. Since all four states belong to $\mathcal{G}$,
		Lemma~\ref{lem:convexity} implies $\con_{i,j}^{n+1} \in \mathcal{G}$.
	\end{proof}
	
	\section{Numerical Results}\label{sec:numerics}
	
	In the previous sections, we developed PCP AFD-WENO schemes for the one- and two-dimensional RHD equations and established that the proposed schemes preserve the admissible set $\mathcal{G}$ under the CFL condition~\eqref{eq:cfl}. In this section, we present numerical results to validate the theoretical analysis and demonstrate the accuracy and robustness of the proposed PCP AFD-WENO schemes. We also assess the performance of the proposed \wenoaoi interpolation through comparisons with the \wenojs, \wenoz, and \wenoao interpolations.
	
	The PCP property established in the previous sections is based on the forward Euler time discretization. To obtain a high-order fully discrete scheme while retaining this property, we employ the third-order strong stability preserving (SSP) Runge--Kutta method. Since this method can be expressed as a convex combination of forward Euler steps, the PCP property can be preserved by applying the PCP algorithm at each Runge--Kutta stage. Writing the semi-discrete scheme~\eqref{eq:1d_scheme} or~\eqref{semi_discrete_scheme_2D} in the form
	\begin{equation*}
		\frac{d\con}{dt} = \mathcal{L}(\con),
	\end{equation*}
	the third-order SSP Runge--Kutta method is given by
	\begin{align*}
		\con^{(1)} &= \con^n + \Delta t\,\mathcal{L}(\con^n), \\
		\con^{(2)} &= \frac{3}{4}\,\con^n + \frac{1}{4}\,\con^{(1)} + \frac{1}{4}\,\Delta t\,\mathcal{L}(\con^{(1)}), \\
		\con^{n+1} &= \frac{1}{3}\,\con^n + \frac{2}{3}\,\con^{(2)} + \frac{2}{3}\,\Delta t\,\mathcal{L}(\con^{(2)}).
	\end{align*}
	To keep the CFL number well within the required bound $0 < \text{CFL} < \frac{1}{2}$, it is set to $0.4$ for all the test problems.
	
	We use the following abbreviations for the different schemes in presenting the numerical results:
	\begin{itemize}
		\item \wenojs: AFD-WENO scheme with the PCP algorithm and WENO-JS5 interpolation defined in Section~\ref{sec:wenojs}.
		\item \wenoz: AFD-WENO scheme with the PCP algorithm and WENO-Z5 interpolation defined in Section~\ref{sec:wenoz}.
		\item \wenoao: AFD-WENO scheme with the PCP algorithm and WENO-AO(5,3) interpolation defined in Section~\ref{sec:wenoao}.
		\item \wenoaoi: AFD-WENO scheme with the PCP algorithm and WENO-AOI interpolation defined in Section~\ref{sec:wenoaoi}.
	\end{itemize}
	
	The parameters for the \wenoao and \wenoaoi interpolations are taken as $\gamma_0^5 = 0.85$ and $\gamma_0^3 = \gamma_1^3 = \gamma_2^3 = 0.05$. The value of $\varepsilon$ in the \wenojs and \wenoz interpolations is set to $10^{-6}$, whereas that in the \wenoao and \wenoaoi interpolations is set to $10^{-12}$.
		
	\subsection{One dimensional test problems}
	In this subsection, we present a series of one-dimensional test problems to examine the accuracy of the proposed schemes for smooth solutions, their resolution of discontinuities, and their ability to preserve the PCP property under challenging flow conditions. We first verify the accuracy of the proposed schemes for smooth solutions.
	
	\begin{testproblem}{Accuracy Tests\label{tp:st}}
		One-dimensional smooth advection problems are used to verify the spatial accuracy of the proposed numerical scheme with periodic boundary conditions. The following two test cases are considered.
		
		\noindent\textbf{Case I.} The computational domain is $[0,1]$. The exact solution is given by
		\begin{equation*}
			\rho(x,t)=2+\sin\!\left(2\pi(x-0.5t)\right),\quad
			v_x(x,t)=0.5,\quad
			p(x,t)=1,
		\end{equation*}
		where the initial condition is obtained by setting $t=0$. The solution is evolved until $t=2.0$.
		
		\noindent\textbf{Case II.} The computational domain is $[0,2\pi]$. The exact solution is given by
		\begin{equation*}
			\rho(x,t)=1+0.99999\sin(x-0.99t),\quad
			v_x(x,t)=0.99,\quad
			p(x,t)=0.005,
		\end{equation*}
		where the corresponding initial condition is obtained by setting $t=0$. The solution is evolved up to the final time $t=0.01$.
		
		Case-I is a standard test problem commonly used to assess the accuracy and convergence properties of numerical schemes. In contrast, Case-II represents a low-density and low-pressure test problem proposed in \cite{wu2017physical}, which is particularly useful for assessing the robustness of numerical schemes under challenging physical conditions. The convergence results for Case-I and Case-II are presented in Tables~\ref{tab:accuracy_test_case1} and~\ref{tab:accuracy_test_case2}, respectively.
		In both cases, all the considered schemes converge to the exact solution with the expected convergence rates. WENO-Z, WENO-AO, and WENO-AOI consistently achieve higher accuracy than WENO-JS. In particular, WENO-AOI provides the smallest $L_{\infty}$ errors on coarse grids as compared to WENO-Z, WENO-AO schemes.

	\begin{table*}[ht!]
		\centering
		\caption{Comparison of WENO-JS, WENO-Z, WENO-AO, and WENO-AOI for Example~\ref{tp:st}, Case-I, in terms of the $L_1$ and $L_{\infty}$ errors and the corresponding convergence rates.
		}
		\label{tab:accuracy_test_case1}
		
		\renewcommand{\arraystretch}{1.15}
		
		\begin{tabular}{ccccccccc}
			\toprule
			&
			\multicolumn{4}{c}{WENO-JS} &
			\multicolumn{4}{c}{WENO-Z} \\
			\cmidrule(lr){2-5}\cmidrule(lr){6-9}
			$N$
			& $L_1$ Error & Order
			& $L^\infty$ Error & Order
			& $L_1$ Error & Order
			& $L^\infty$ Error & Order \\
			\midrule
			
			8
			& $8.2531\times10^{-2}$ & --
			& $1.2346\times10^{-1}$ & --
			& $2.3862\times10^{-2}$ & --
			& $3.7186\times10^{-2}$ & -- \\
			
			24
			& $7.5593\times10^{-4}$ & 4.272
			& $1.3554\times10^{-3}$ & 4.107
			& $9.6713\times10^{-5}$ & 5.014
			& $1.5821\times10^{-4}$ & 4.970 \\
			
			40
			& $5.8521\times10^{-5}$ & 5.009
			& $1.1310\times10^{-4}$ & 4.862
			& $7.5905\times10^{-6}$ & 4.982
			& $1.2400\times10^{-5}$ & 4.985 \\
			
			56
			& $1.0879\times10^{-5}$ & 5.001
			& $2.1965\times10^{-5}$ & 4.871
			& $1.4150\times10^{-6}$ & 4.992
			& $2.3111\times10^{-6}$ & 4.993 \\
			
			72
			& $3.0963\times10^{-6}$ & 5.000
			& $6.2992\times10^{-6}$ & 4.970
			& $4.0312\times10^{-7}$ & 4.996
			& $6.5848\times10^{-7}$ & 4.996 \\
			
		\end{tabular}
		
		\vspace{0.4cm}
		
		\begin{tabular}{ccccccccc}
			\toprule
			&
			\multicolumn{4}{c}{WENOAO} &
			\multicolumn{4}{c}{WENOAOI} \\
			\cmidrule(lr){2-5}\cmidrule(lr){6-9}
			$N$
			& $L_1$ Error & Order
			& $L^\infty$ Error & Order
			& $L_1$ Error & Order
			& $L^\infty$ Error & Order \\
			\midrule
			
			8
			& $2.3359\times10^{-2}$ & --
			& $3.4154\times10^{-2}$ & --
			& $2.1775\times10^{-2}$ & --
			& $3.2192\times10^{-2}$ & -- \\
			
			24
			& $9.7021\times10^{-5}$ & 4.992
			& $1.5756\times10^{-4}$ & 4.896
			& $9.6791\times10^{-5}$ & 4.930
			& $1.5763\times10^{-4}$ & 4.842 \\
			
			40
			& $7.5944\times10^{-6}$ & 4.987
			& $1.2392\times10^{-5}$ & 4.978
			& $7.5923\times10^{-6}$ & 4.983
			& $1.2392\times10^{-5}$ & 4.979 \\
			
			56
			& $1.4152\times10^{-6}$ & 4.993
			& $2.3106\times10^{-6}$ & 4.992
			& $1.4151\times10^{-6}$ & 4.993
			& $2.3106\times10^{-6}$ & 4.992 \\
			
			72
			& $4.0314\times10^{-7}$ & 4.997
			& $6.5842\times10^{-7}$ & 4.995
			& $4.0313\times10^{-7}$ & 4.996
			& $6.5842\times10^{-7}$ & 4.995 \\
			
			\bottomrule
		\end{tabular}
	\end{table*}
	\begin{table*}[ht!]
		\centering
		\caption{Comparison of WENO-JS, WENO-Z, WENO-AO, and WENO-AOI for Example~\ref{tp:st}, Case-II, in terms of the $L_1$ and $L_{\infty}$ errors and the corresponding convergence rates.}
		\label{tab:accuracy_test_case2}
		
		\renewcommand{\arraystretch}{1.15}
		
		\begin{tabular}{ccccccccc}
			\toprule
			&
			\multicolumn{4}{c}{WENO-JS} &
			\multicolumn{4}{c}{WENO-Z} \\
			\cmidrule(lr){2-5}\cmidrule(lr){6-9}
			$N$
			& $L_1$ Error & Order
			& $L^\infty$ Error & Order
			& $L_1$ Error & Order
			& $L^\infty$ Error & Order \\
			\midrule
			
			8  & $1.8947\times10^{-4}$ & --    & $3.2892\times10^{-4}$ & --    &
			$2.9198\times10^{-5}$ & --    & $6.0073\times10^{-5}$ & --    \\
			
			24 & $9.1389\times10^{-7}$ & 4.855 & $1.5166\times10^{-6}$ & 4.896 &
			$9.0715\times10^{-8}$ & 5.256 & $1.4529\times10^{-7}$ & 5.484 \\
			
			40 & $6.4967\times10^{-8}$ & 5.176 & $1.1834\times10^{-7}$ & 4.993 &
			$7.3127\times10^{-9}$ & 4.929 & $1.1769\times10^{-8}$ & 4.920 \\
			
			56 & $1.1586\times10^{-8}$ & 5.124 & $2.2055\times10^{-8}$ & 4.993 &
			$1.4233\times10^{-9}$ & 4.864 & $2.4882\times10^{-9}$ & 4.618 \\
			
			72 & $3.2138\times10^{-9}$ & 5.102 & $6.2855\times10^{-9}$ & 4.995 &
			$3.9918\times10^{-10}$ & 5.059 & $7.1690\times10^{-10}$ & 4.951 \\
			
		\end{tabular}
		
		\vspace{0.4cm}
		
		\begin{tabular}{ccccccccc}
			\toprule
			&
			\multicolumn{4}{c}{WENOAO} &
			\multicolumn{4}{c}{WENOAOI} \\
			\cmidrule(lr){2-5}\cmidrule(lr){6-9}
			$N$
			& $L_1$ Error & Order
			& $L^\infty$ Error & Order
			& $L_1$ Error & Order
			& $L^\infty$ Error & Order \\
			\midrule
			
			8  & $2.4857\times10^{-5}$ & --    & $6.1101\times10^{-5}$ & --    &
			$2.3364\times10^{-5}$ & --    & $5.8126\times10^{-5}$ & --    \\
			
			24 & $9.1035\times10^{-8}$ & 5.106 & $1.4176\times10^{-7}$ & 5.522 &
			$9.0835\times10^{-8}$ & 5.052 & $1.4178\times10^{-7}$ & 5.476 \\
			
			40 & $7.3163\times10^{-9}$ & 4.935 & $1.1713\times10^{-8}$ & 4.881 &
			$7.3144\times10^{-9}$ & 4.932 & $1.1719\times10^{-8}$ & 4.880 \\
			
			56 & $1.4235\times10^{-9}$ & 4.865 & $2.4853\times10^{-9}$ & 4.608 &
			$1.4234\times10^{-9}$ & 4.865 & $2.4855\times10^{-9}$ & 4.609 \\
			
			72 & $3.9918\times10^{-10}$ & 5.059 & $7.1658\times10^{-10}$ & 4.949 &
			$3.9920\times10^{-10}$ & 5.059 & $7.1660\times10^{-10}$ & 4.949 \\
			
			\bottomrule
		\end{tabular}
	\end{table*}	
	\end{testproblem}
	\begin{figure}[ht!]
		\centering
		\begin{subfigure}{0.48\textwidth}
			\centering
			\includegraphics[width=\textwidth]{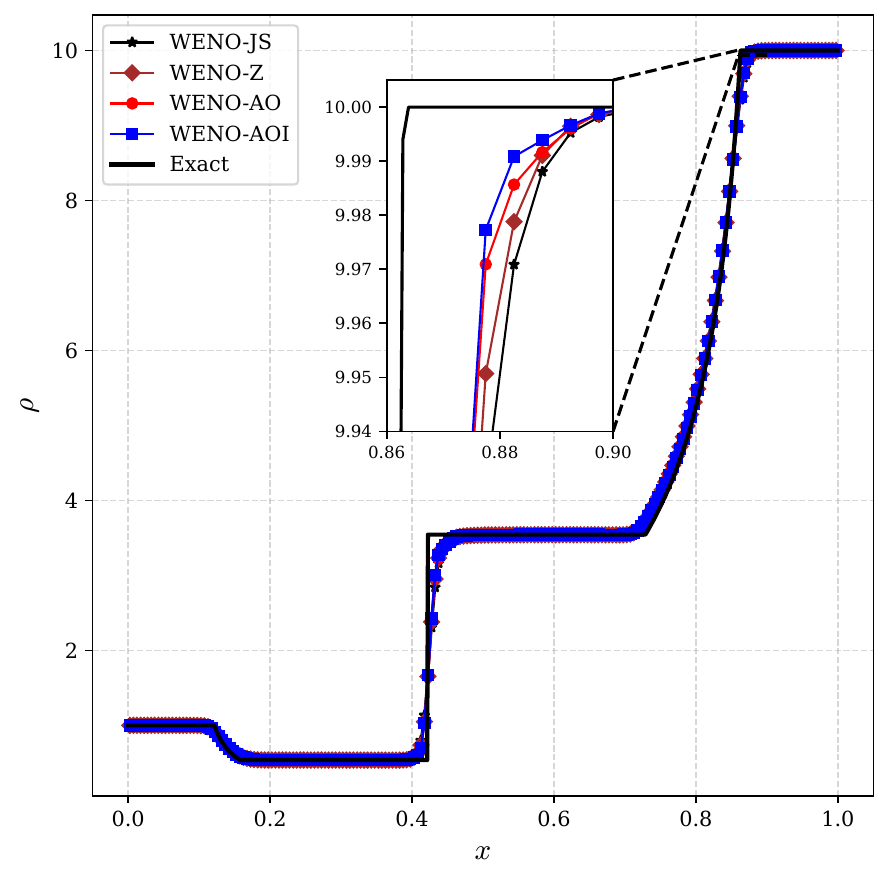}
			\caption{\id}
		\end{subfigure}
		\hfill
		\begin{subfigure}{0.48\textwidth}
			\centering
			\includegraphics[width=\textwidth]{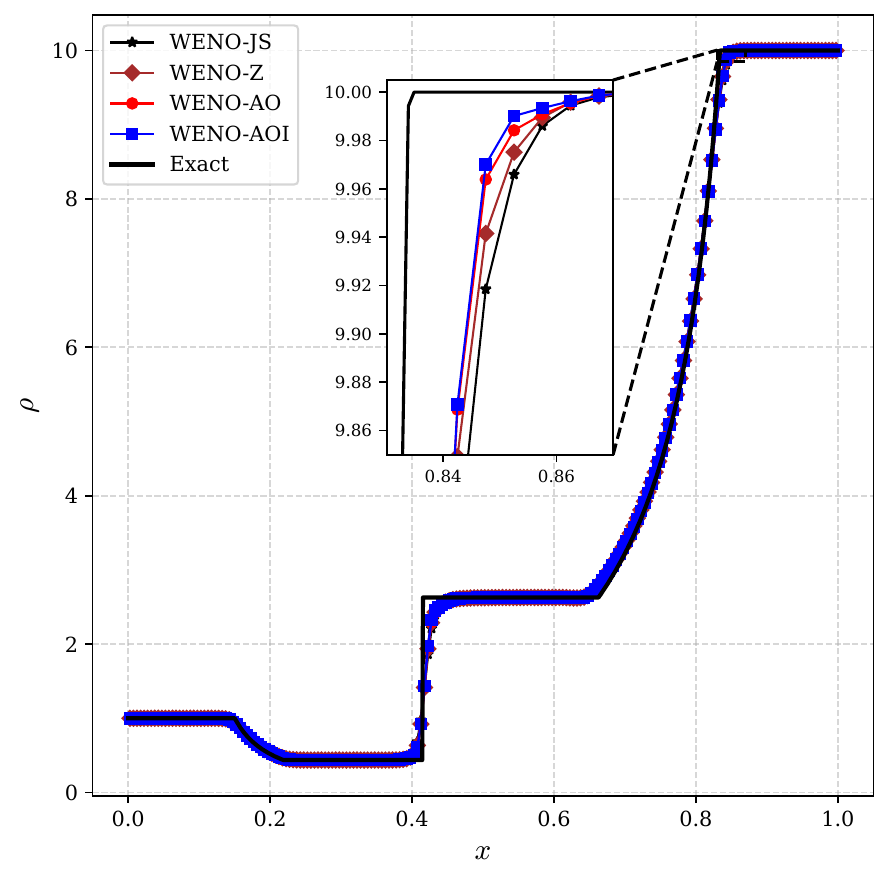}
			\caption{\ip}
		\end{subfigure}
		\\[6pt]
		\begin{subfigure}{0.48\textwidth}
			\centering
			\includegraphics[width=\textwidth]{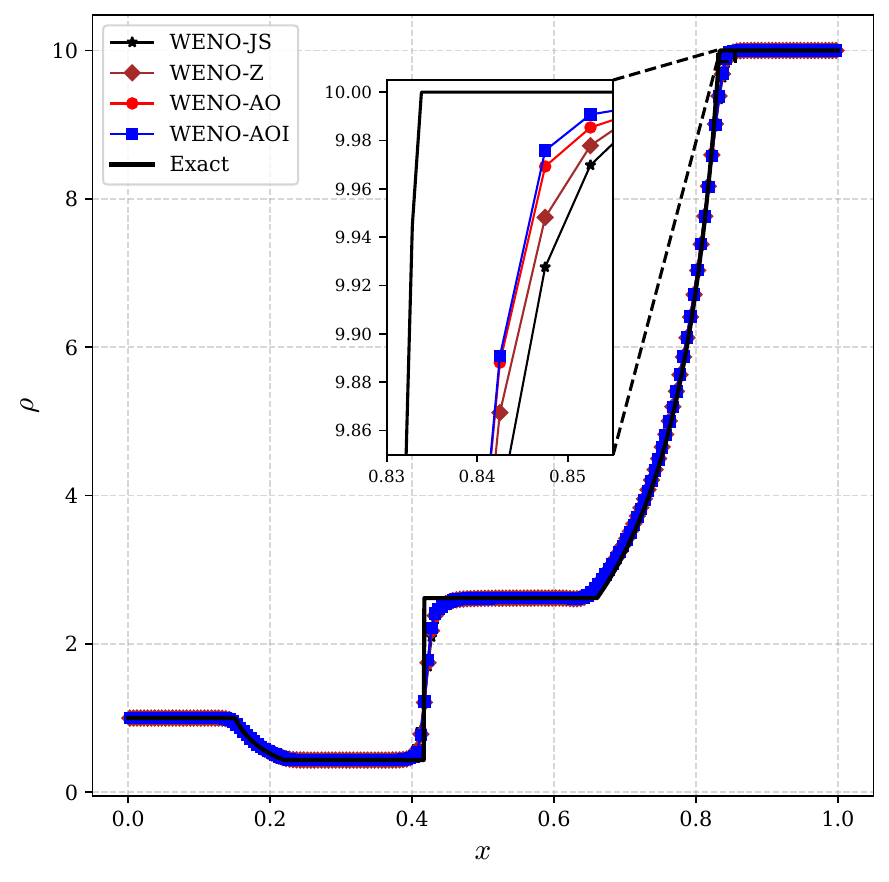}
			\caption{\rc}
		\end{subfigure}
		\hfill
		\begin{subfigure}{0.48\textwidth}
			\centering
			\includegraphics[width=\textwidth]{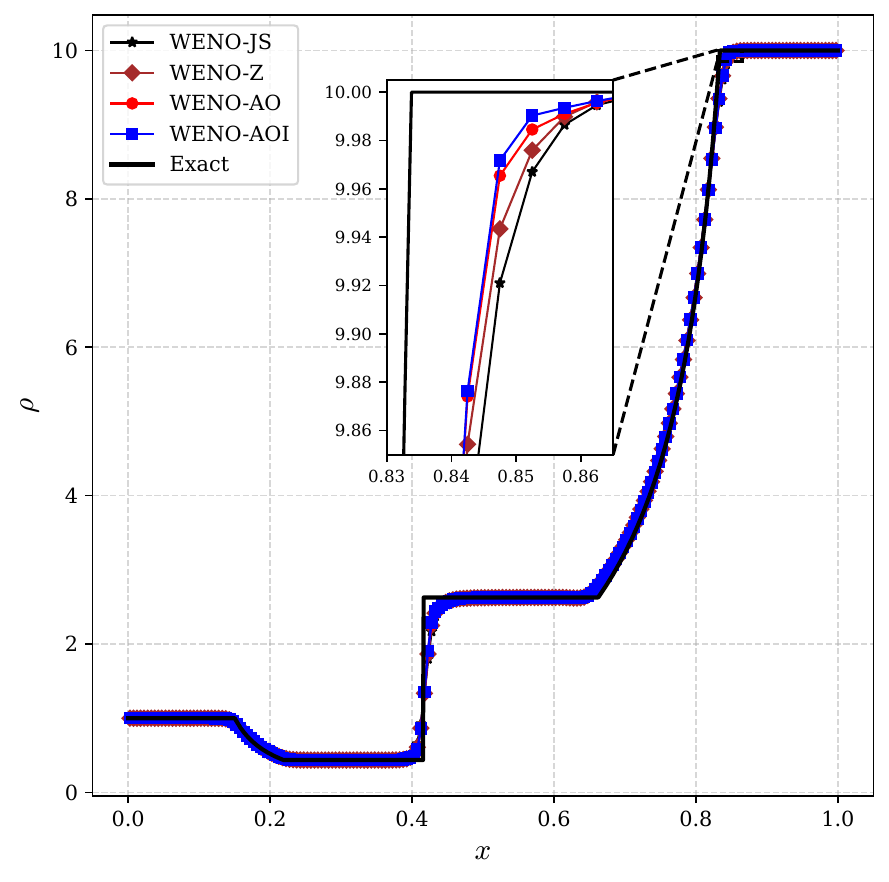}
			\caption{\tm}
		\end{subfigure}
		\caption{Numerical solutions for Test Problem~\ref{tp:rp1} at $T=0.4$ on a $200$-point mesh.}
		\label{fig:rp1_200}
	\end{figure}
	\begin{figure}[ht!]
		\centering
		\begin{subfigure}{0.48\textwidth}
			\centering
			\includegraphics[width=\textwidth]{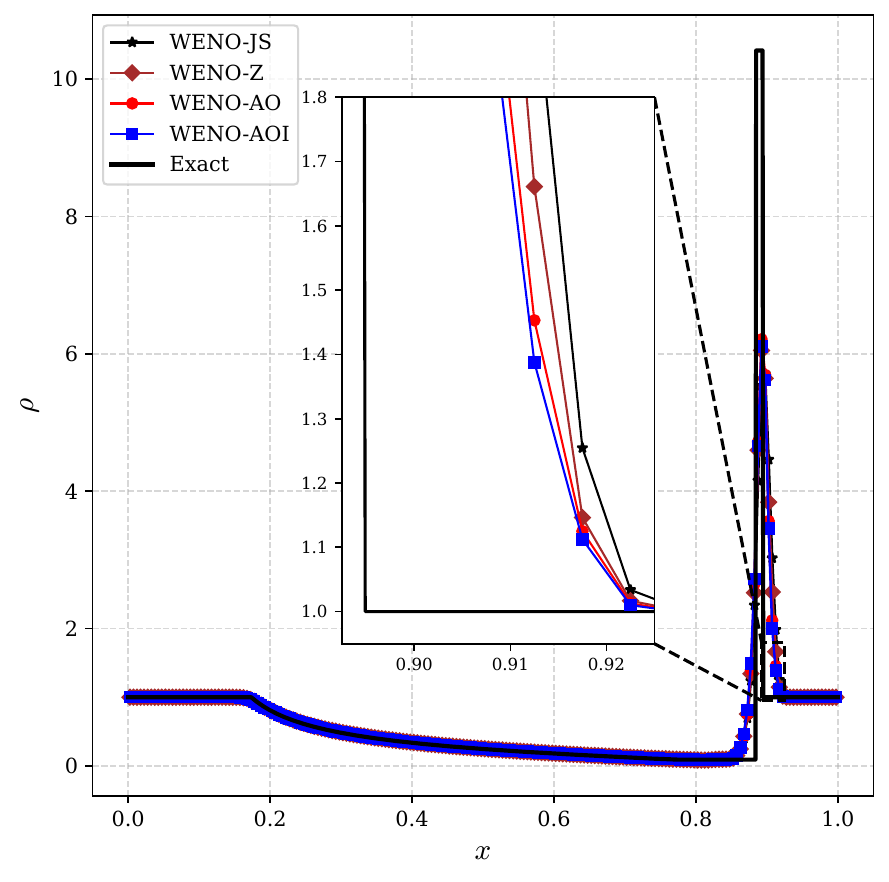}
			\caption{\id}
		\end{subfigure}
		\hfill
		\begin{subfigure}{0.48\textwidth}
			\centering
			\includegraphics[width=\textwidth]{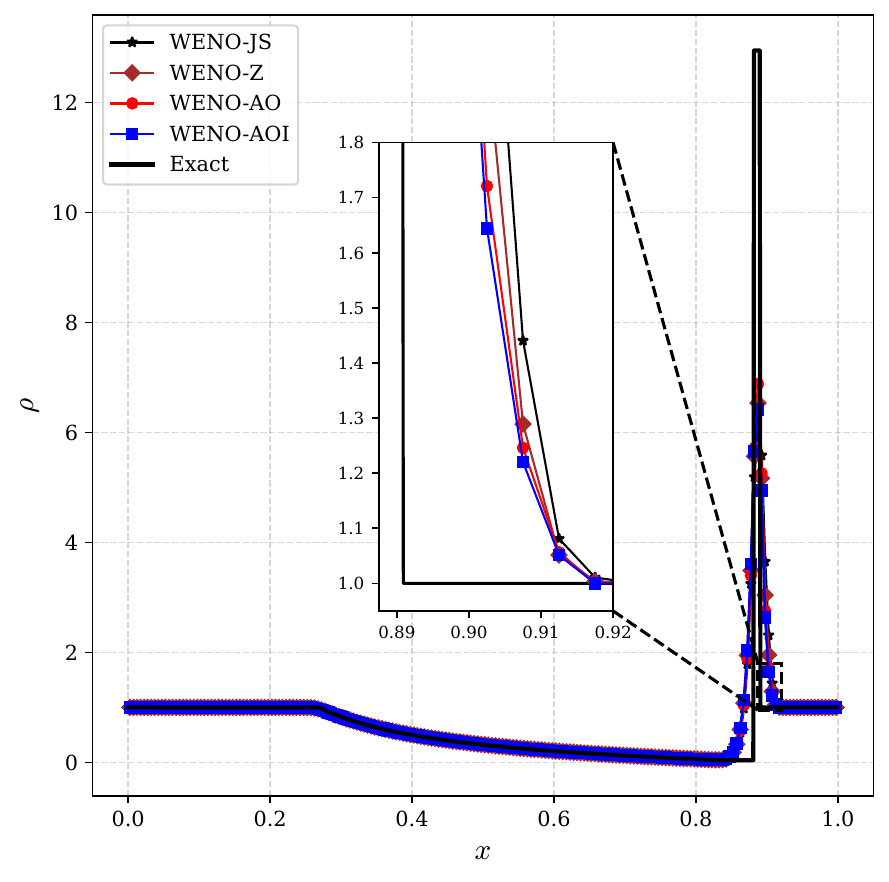}
			\caption{\ip}
		\end{subfigure}
		\\[6pt]
		\begin{subfigure}{0.48\textwidth}
			\centering
			\includegraphics[width=\textwidth]{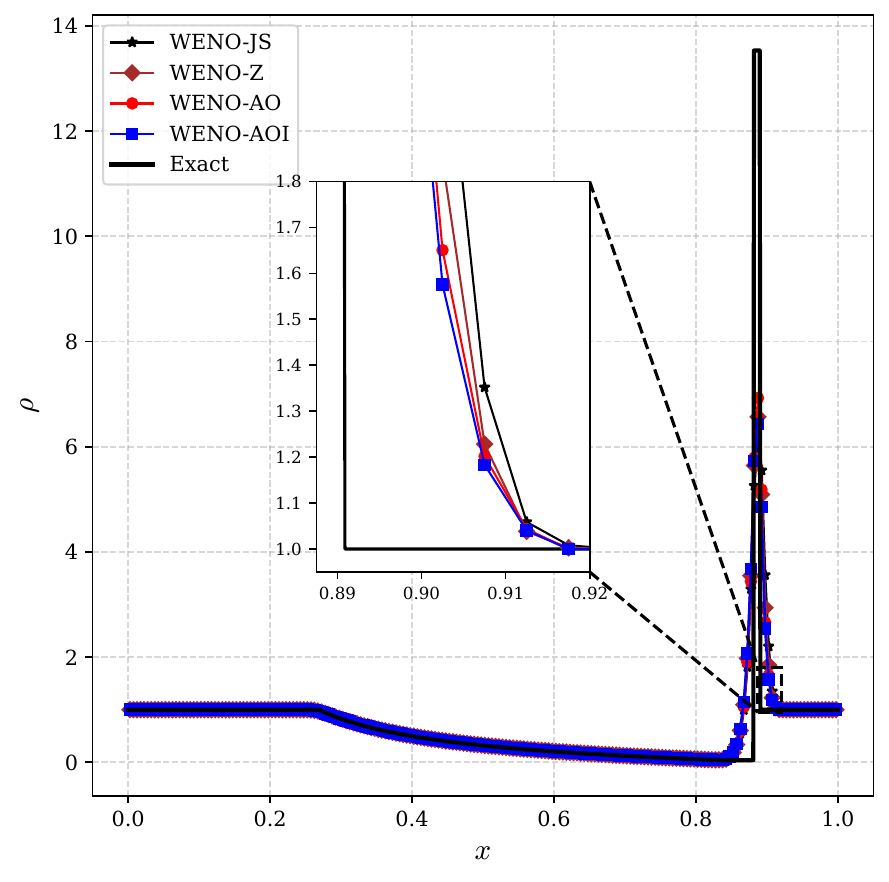}
			\caption{\rc}
		\end{subfigure}
		\hfill
		\begin{subfigure}{0.48\textwidth}
			\centering
			\includegraphics[width=\textwidth]{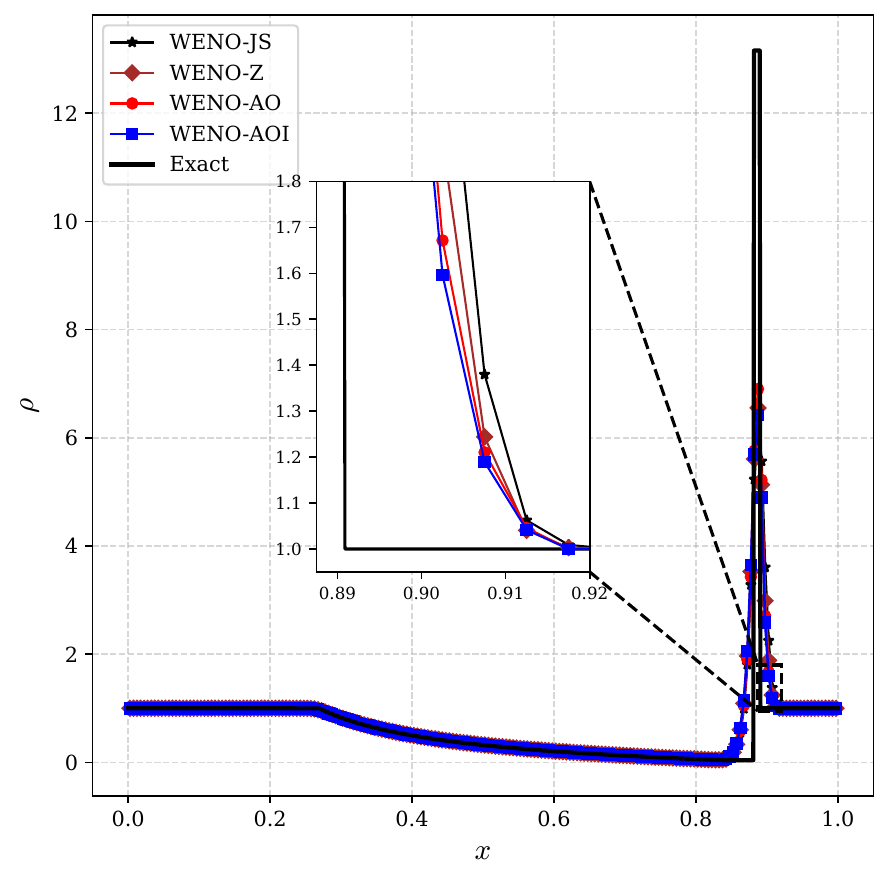}
			\caption{\tm}
		\end{subfigure}
		\caption{Numerical solutions for Test Problem~\ref{tp:rp2} at $T=0.4$ on a $200$-point mesh.}
		\label{fig:rp2_200}
	\end{figure}
	\begin{figure}[ht!]
		\centering
		\begin{subfigure}{0.48\textwidth}
			\centering
			\includegraphics[width=\textwidth]{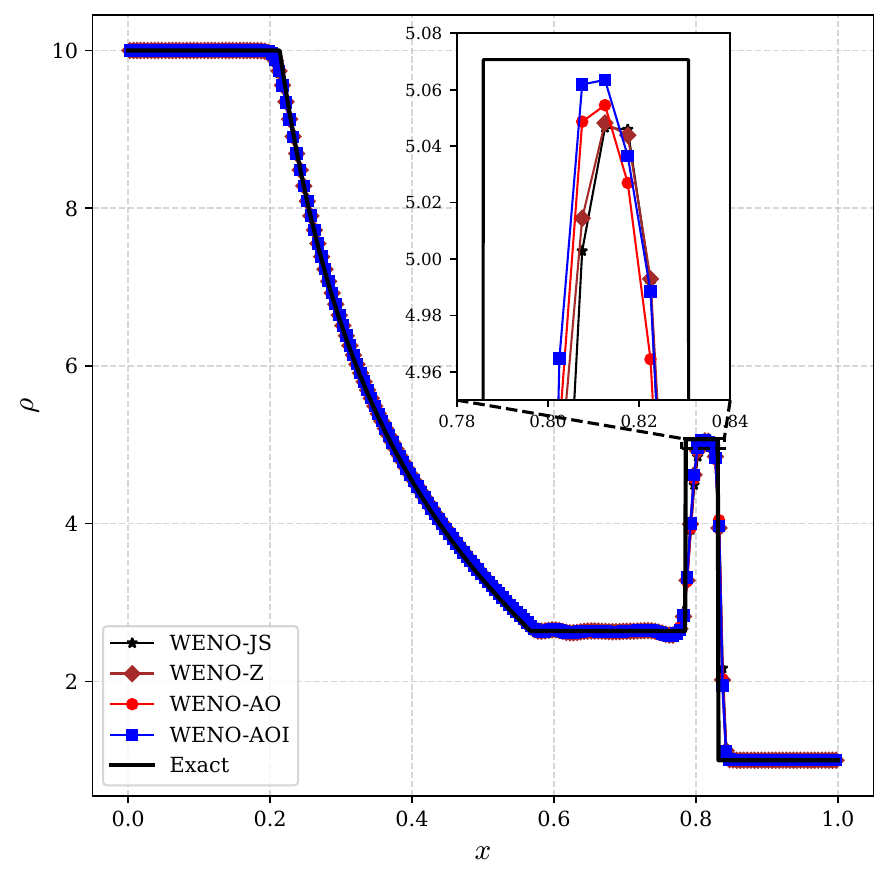}
			\caption{\id}
		\end{subfigure}
		\hfill
		\begin{subfigure}{0.48\textwidth}
			\centering
			\includegraphics[width=\textwidth]{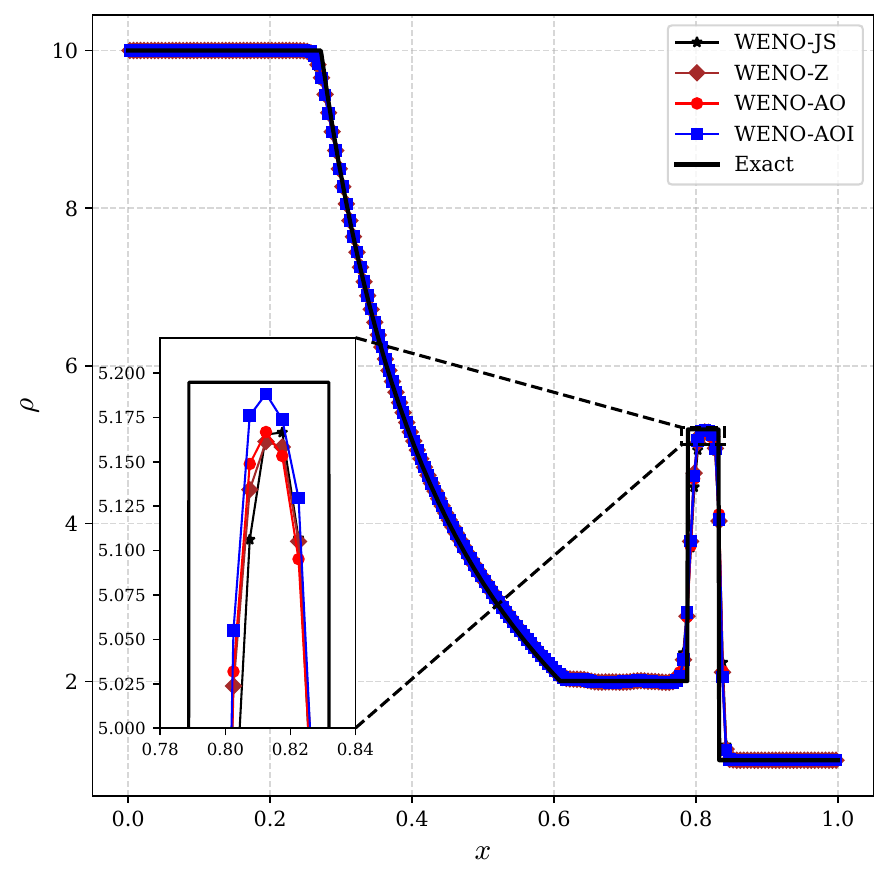}
			\caption{\ip}
		\end{subfigure}
		\\[6pt]
		\begin{subfigure}{0.48\textwidth}
			\centering
			\includegraphics[width=\textwidth]{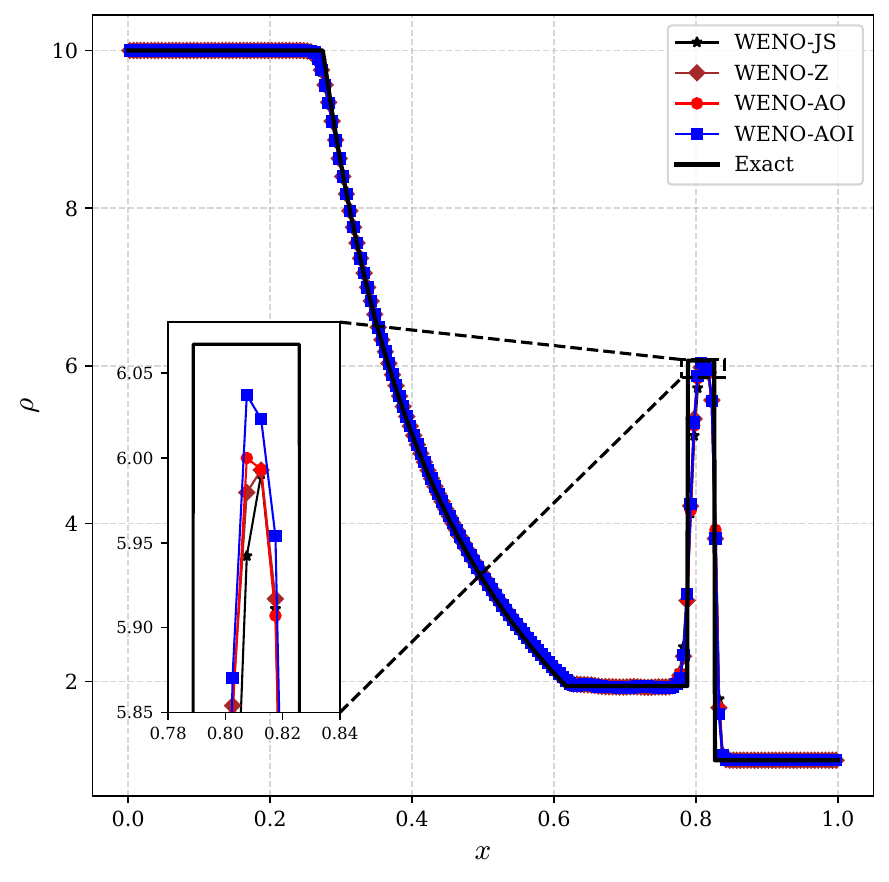}
			\caption{\rc}
		\end{subfigure}
		\hfill
		\begin{subfigure}{0.48\textwidth}
			\centering
			\includegraphics[width=\textwidth]{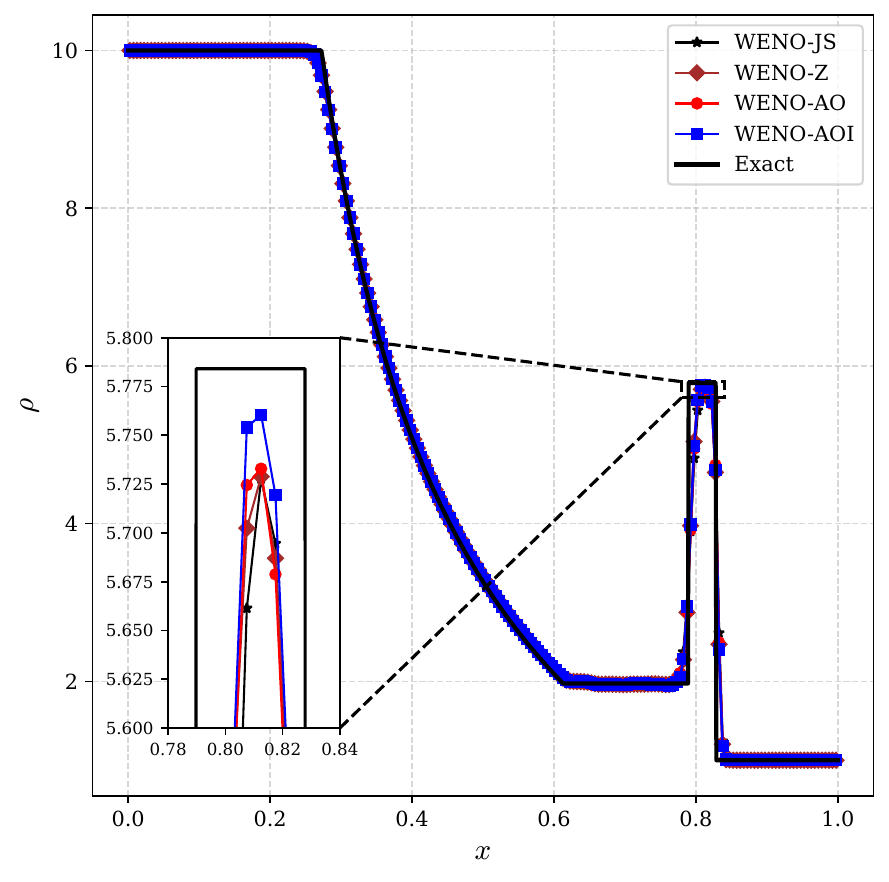}
			\caption{\tm}
		\end{subfigure}
		\caption{Numerical solutions for Test Problem~\ref{tp:rp3} at $T=0.4$ on a $200$-point mesh.}
		\label{fig:rp3}
	\end{figure}
	\begin{figure}[ht!]
		\centering
		\begin{subfigure}{0.48\textwidth}
			\centering
			\includegraphics[width=\textwidth]{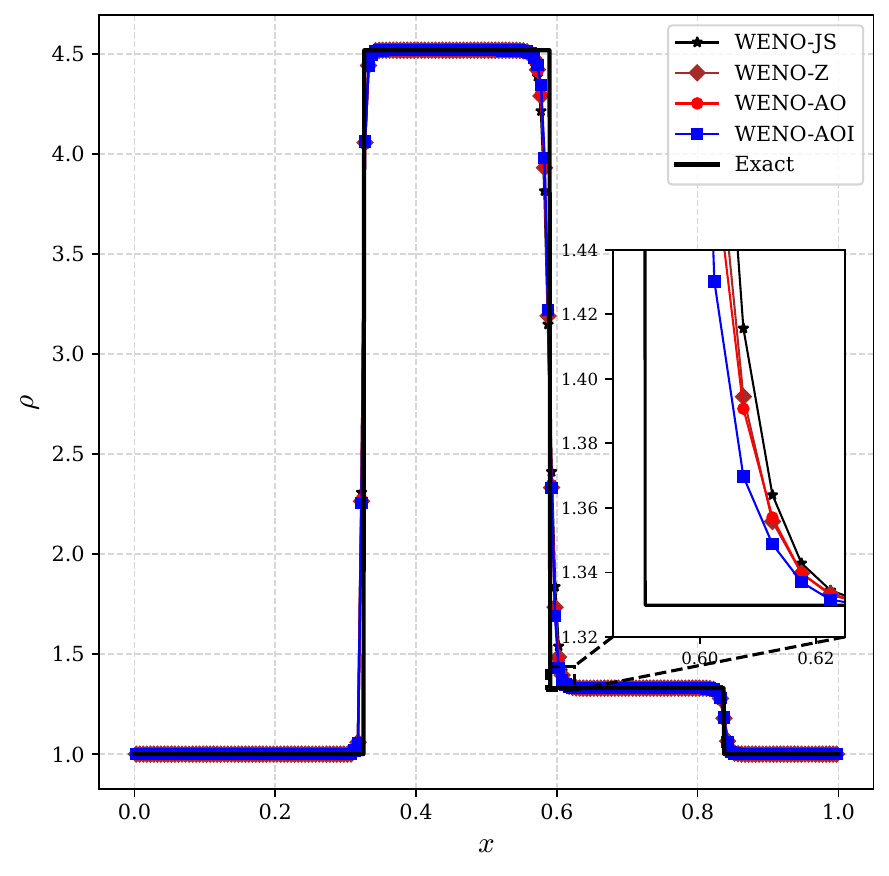}
			\caption{\id}
		\end{subfigure}
		\hfill
		\begin{subfigure}{0.48\textwidth}
			\centering
			\includegraphics[width=\textwidth]{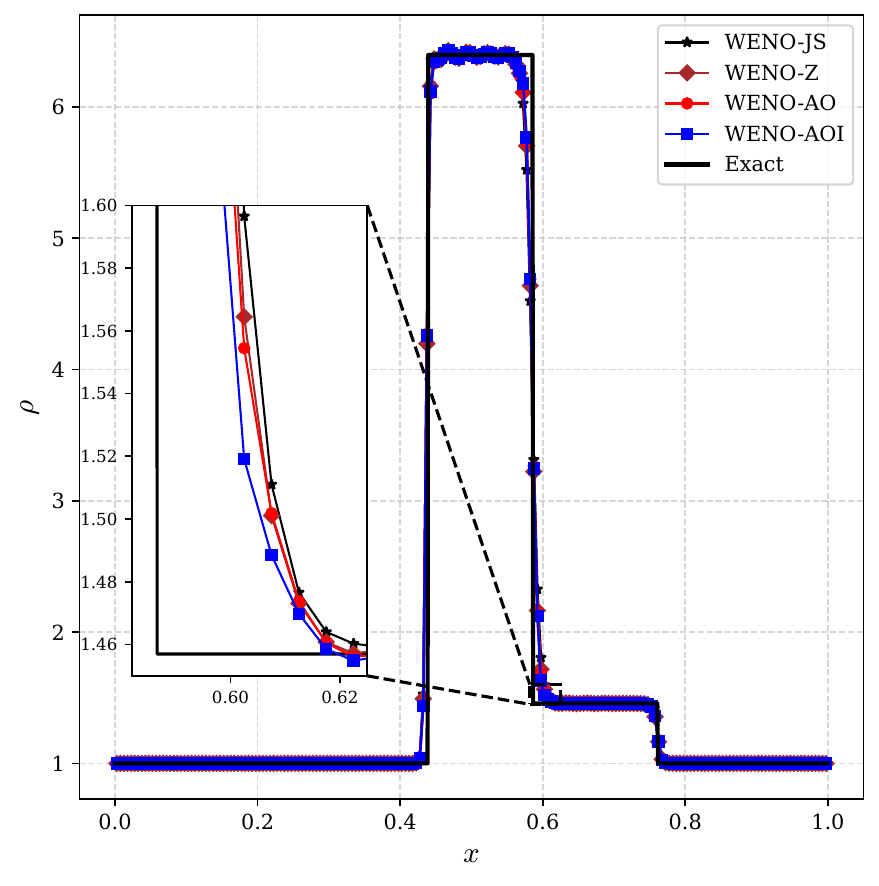}
			\caption{\ip}
		\end{subfigure}
		\\[6pt]
		\begin{subfigure}{0.48\textwidth}
			\centering
			\includegraphics[width=\textwidth]{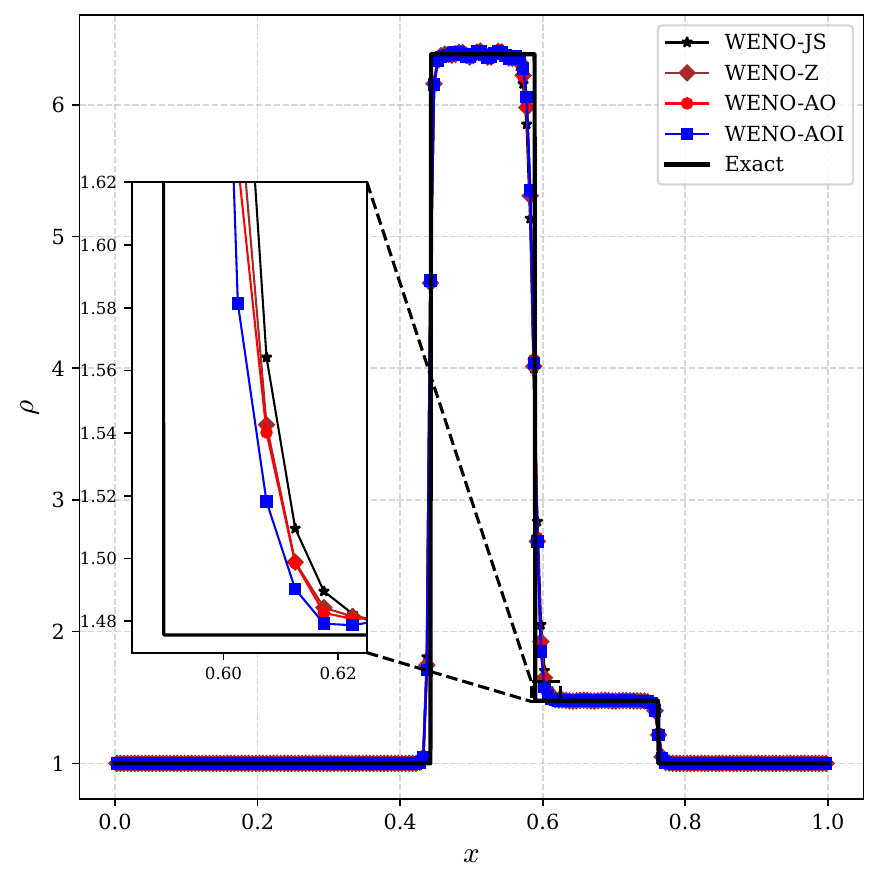}
			\caption{\rc}
		\end{subfigure}
		\hfill
		\begin{subfigure}{0.48\textwidth}
			\centering
			\includegraphics[width=\textwidth]{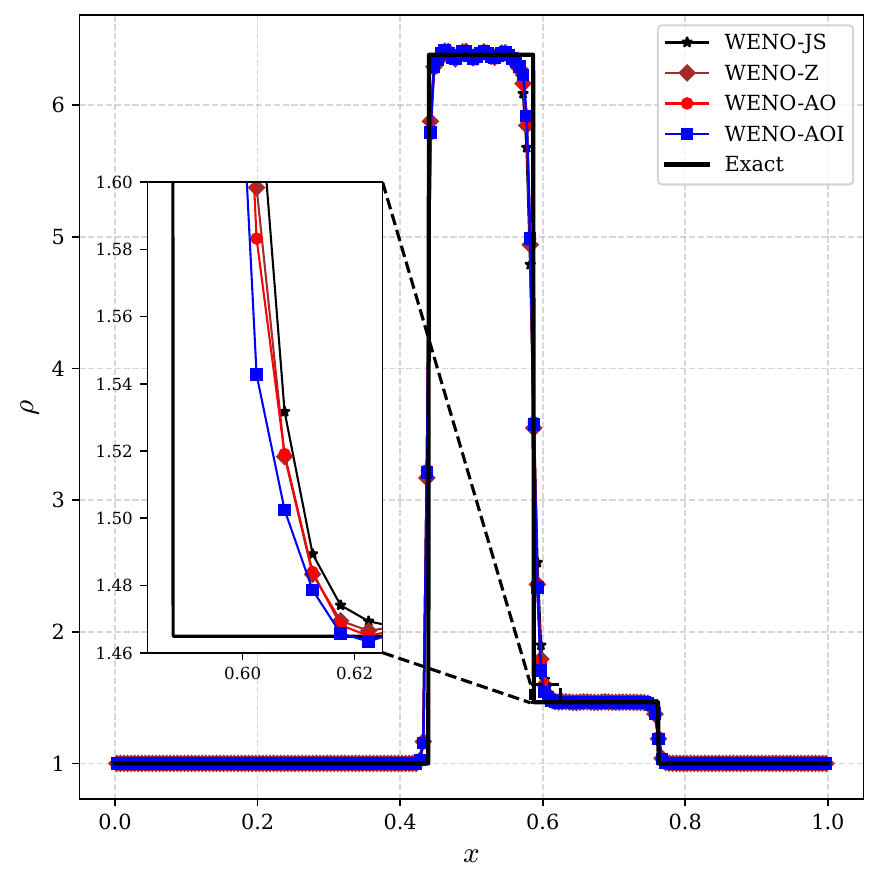}
			\caption{\tm}
		\end{subfigure}
		\caption{Numerical solutions for Test Problem~\ref{tp:rp4} at $T=0.4$ on a $200$-point mesh.}
		\label{fig:rp4}
	\end{figure}
	\begin{testproblem}{Riemann Problem 1}
		\label{tp:rp1}
		In this test, we consider a Riemann problem discussed in \cite{mignone2005hllc}. The computational domain is $[0,1]$ and an initial discontinuity is placed at $x=0.5$. The initial states are given by
		\begin{equation*}
			\left(\rho,\,u,\,p\right) =
			\begin{cases}
				\left(1,-0.6,10\right), & \text{if } x < 0.5, \\
				\left(10,0.5,20\right), & \text{if } x > 0.5.
			\end{cases}
		\end{equation*}
		The exact solution consists of two oppositely propagating rarefaction waves separated by a contact discontinuity. Outflow boundary conditions are imposed at both ends of the computational domain. The numerical results at $t=0.4$ are presented in Figure~\ref{fig:rp1_200}.
		
		The numerical solutions obtained using the \wenojs, \wenoz, \wenoao, and \wenoaoi schemes successfully capture the key features of the solution for all four equations of state (EOS). In comparison with \wenojs and \wenoz, the \wenoao and \wenoaoi schemes provide sharper resolution of the solution features. The enlarged view near the peak of the rarefaction wave further shows that \wenoaoi achieves a slightly sharper resolution than \wenojs, \wenoz, and \wenoao
	\end{testproblem}
	\begin{figure}[ht]
		\centering
		\begin{subfigure}{0.48\textwidth}
			\centering
			\includegraphics[width=\textwidth]{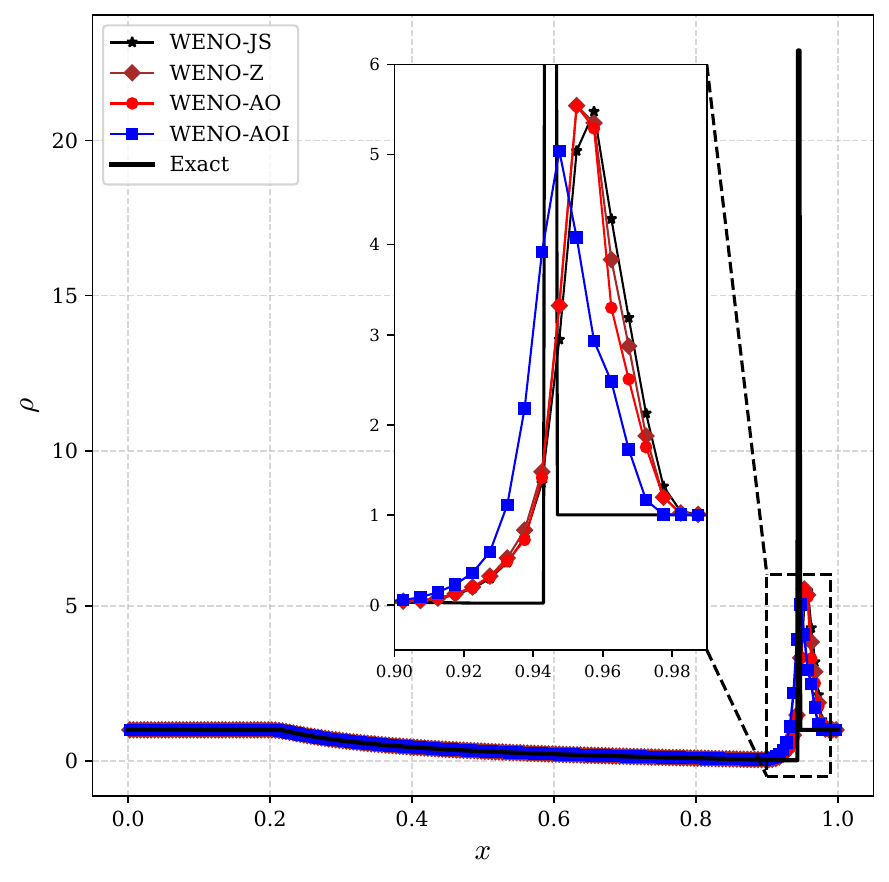}
			\caption{\id}
		\end{subfigure}
		\hfill
		\begin{subfigure}{0.48\textwidth}
			\centering
			\includegraphics[width=\textwidth]{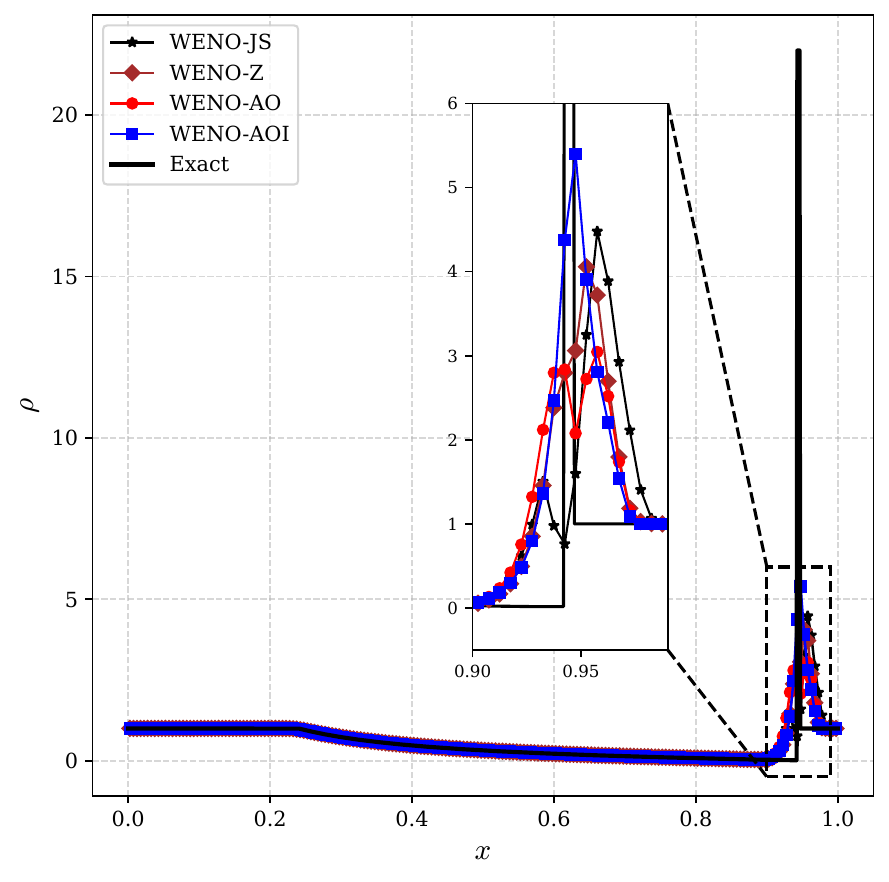}
			\caption{\ip}
		\end{subfigure}
		\\[6pt]
		\begin{subfigure}{0.48\textwidth}
			\centering
			\includegraphics[width=\textwidth]{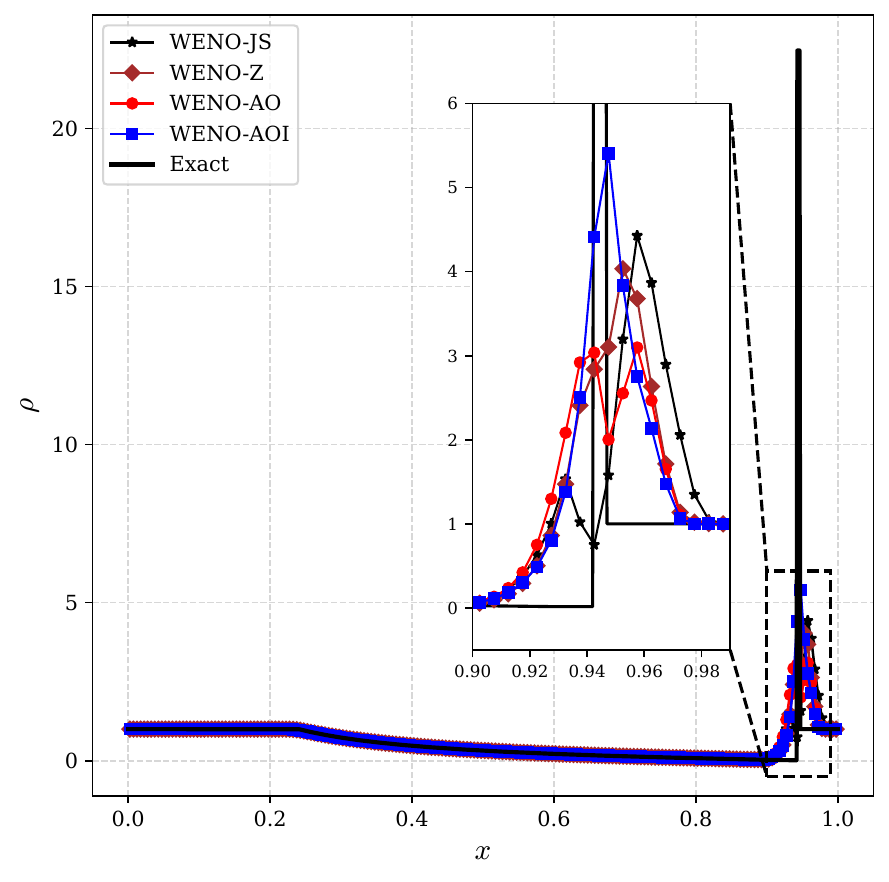}
			\caption{\rc}
		\end{subfigure}
		\hfill
		\begin{subfigure}{0.48\textwidth}
			\centering
			\includegraphics[width=\textwidth]{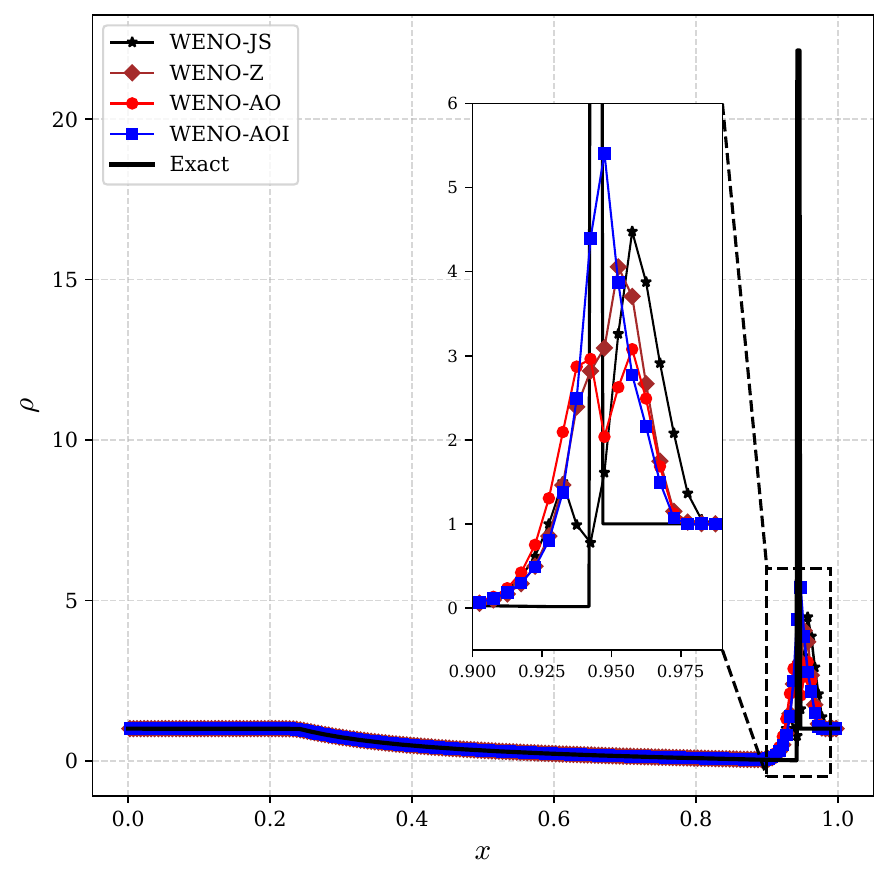}
			\caption{\tm}
		\end{subfigure}
		\caption{Numerical solutions for Test Problem~\ref{tp:rp5} at $T=0.45$ on a $200$-point mesh.}
		\label{fig:rp5}
	\end{figure}
	\begin{figure}[ht]
		\centering
		\begin{subfigure}{0.48\textwidth}
			\centering
			\includegraphics[width=\textwidth]{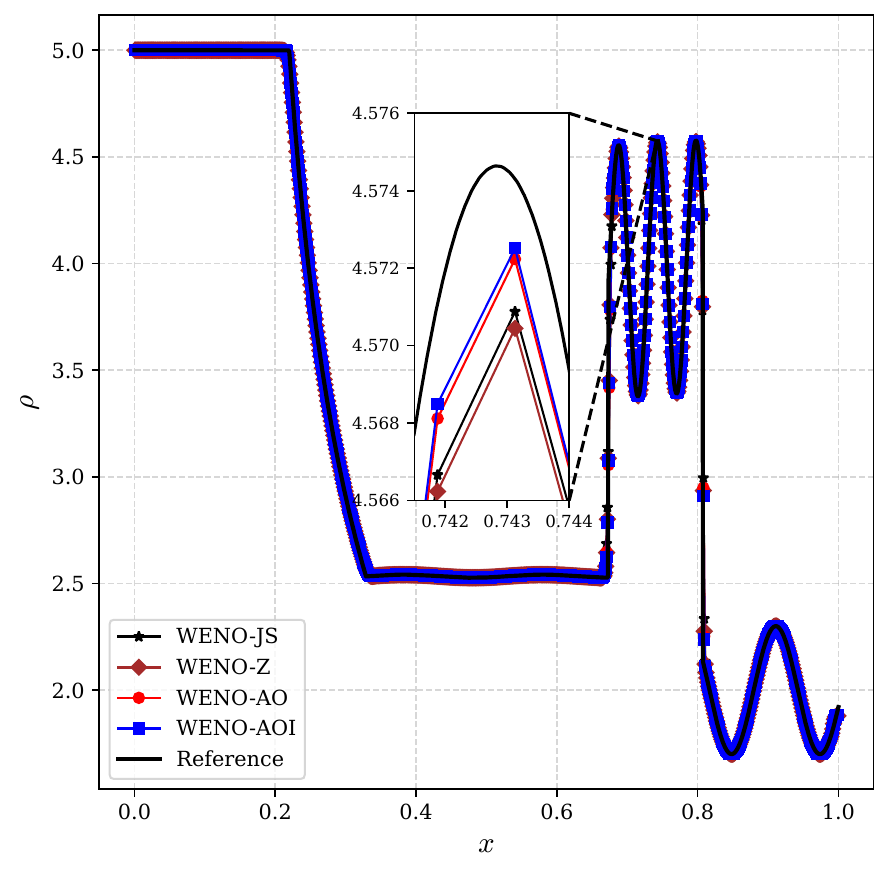}
			\caption{\id}
		\end{subfigure}
		\hfill
		\begin{subfigure}{0.48\textwidth}
			\centering
			\includegraphics[width=\textwidth]{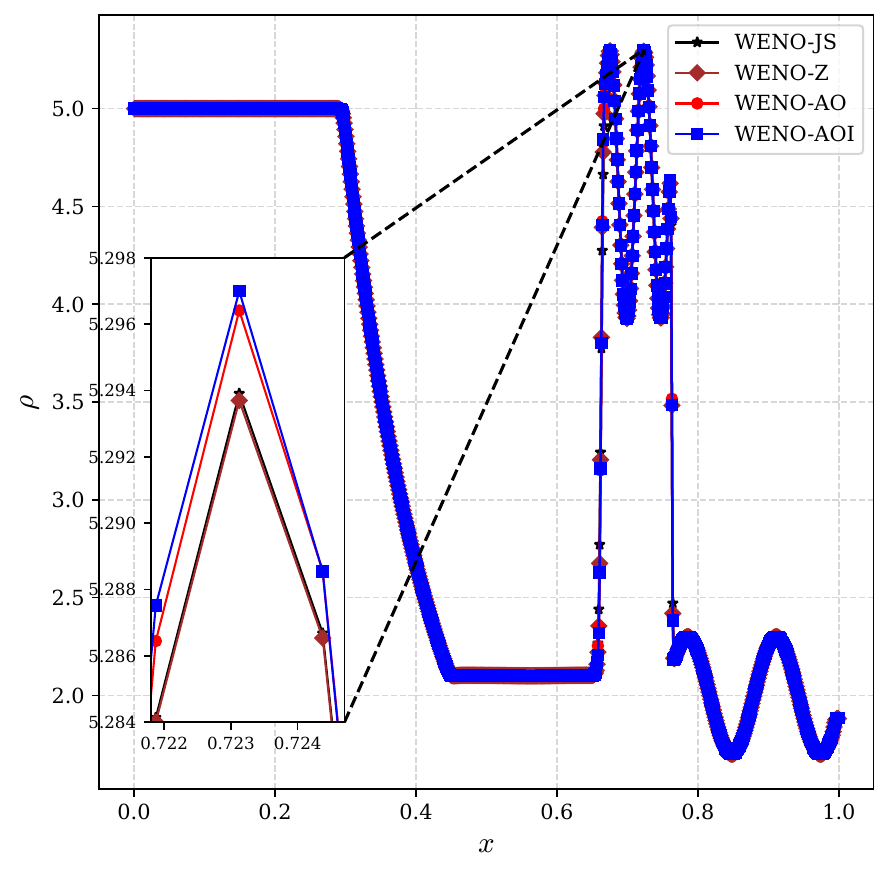}
			\caption{\ip}
		\end{subfigure}
		\\[6pt]
		\begin{subfigure}{0.48\textwidth}
			\centering
			\includegraphics[width=\textwidth]{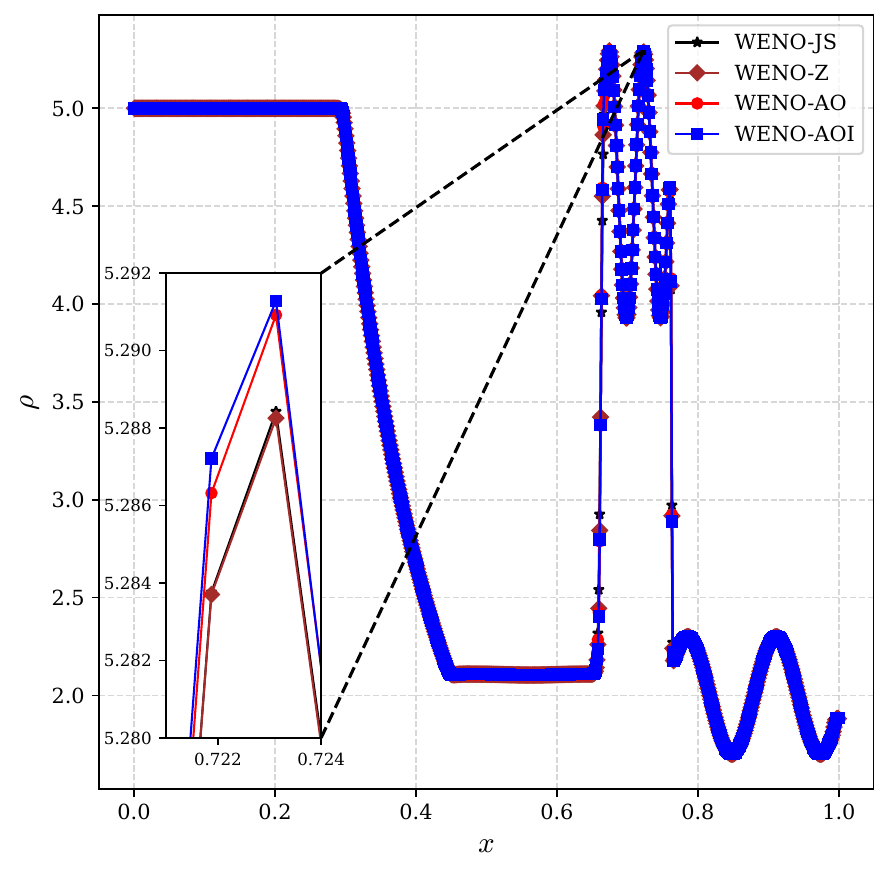}
			\caption{\rc}
		\end{subfigure}
		\hfill
		\begin{subfigure}{0.48\textwidth}
			\centering
			\includegraphics[width=\textwidth]{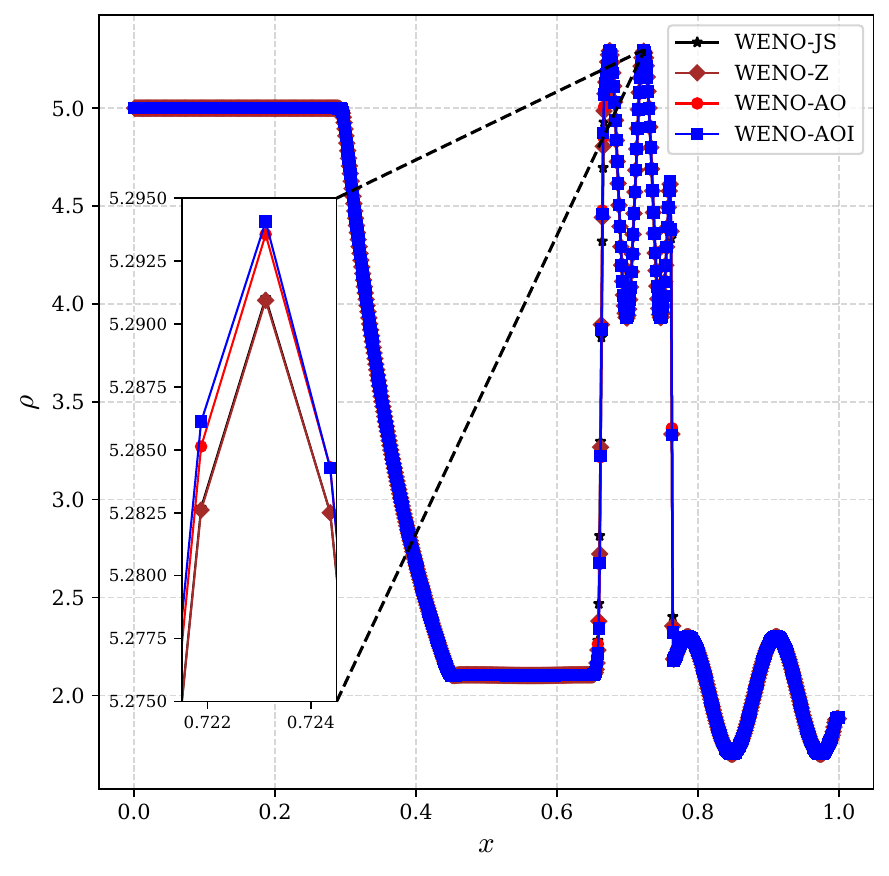}
			\caption{\tm}
		\end{subfigure}
		\caption{Numerical solutions for Test Problem~\ref{tp:SO} at $T=0.35$ on a $800$-point mesh.}
		\label{fig:so}
	\end{figure}

	\begin{testproblem}{Riemann Problem 2}
		\label{tp:rp2}
		We test the scheme on a shock tube Riemann problem from \cite{marti2003numerical}. The computational domain is $[0,1]$ with outflow boundary conditions. The initial conditions are specified as follows:
		\begin{equation*}\left(\rho,\,u,\,p\right) =
			\begin{cases}
				\left(1,0,10^3\right), & \text{if } x < 0.5, \\
				\left(1,0,10^{-2}\right), & \text{if } x > 0.5.
			\end{cases}
		\end{equation*}
		The exact solution contains all major wave types, making it an effective test for assessing the wave-capturing ability of the schemes. The numerical results at $t = 0.4$ obtained using $200$ grid points are shown in Figures \ref{fig:rp2_200}. We compare the numerical solutions obtained using the \wenojs, \wenoz, \wenoao, and \wenoaoi schemes with the exact solution. The numerical results accurately capture all the essential solution features for all four EOS. 
	\end{testproblem}
	\begin{testproblem}{Riemann Problem 3}
		\label{tp:rp3}
		We consider another Riemann problem from \cite{mignone2005hllc}. The computational domain is $[0,1]$, with the initial states given by
		\begin{align*}\label{eq:rp3}
			\left(\rho,\,u,\,p\right) =
			\begin{cases}
				\left(10,0,\dfrac{40}{3}\right), & \text{if } x < 0.5, \\
				\left(1,0,\dfrac{2}{3}\times10^{-5}\right), & \text{if } x > 0.5.
			\end{cases}
		\end{align*}
		The right state corresponds to an extremely low pressure, making this test case an excellent benchmark for examining the robustness of numerical schemes. Outflow boundary conditions are applied on both sides. The numerical results at $t=0.4$ obtained using $200$ grid points are shown in Figure \ref{fig:rp3}. We compare the numerical solutions obtained using the \wenojs, \wenoz, \wenoao, and \wenoaoi schemes with the exact solution. The numerical results accurately capture the essential flow features. \wenoaoi scheme provides slightly better resolution than the other schemes for all four EOS.
	\end{testproblem}
	\begin{figure}[ht]
		\centering
		\begin{subfigure}{0.48\textwidth}
			\centering
			\includegraphics[width=\textwidth]{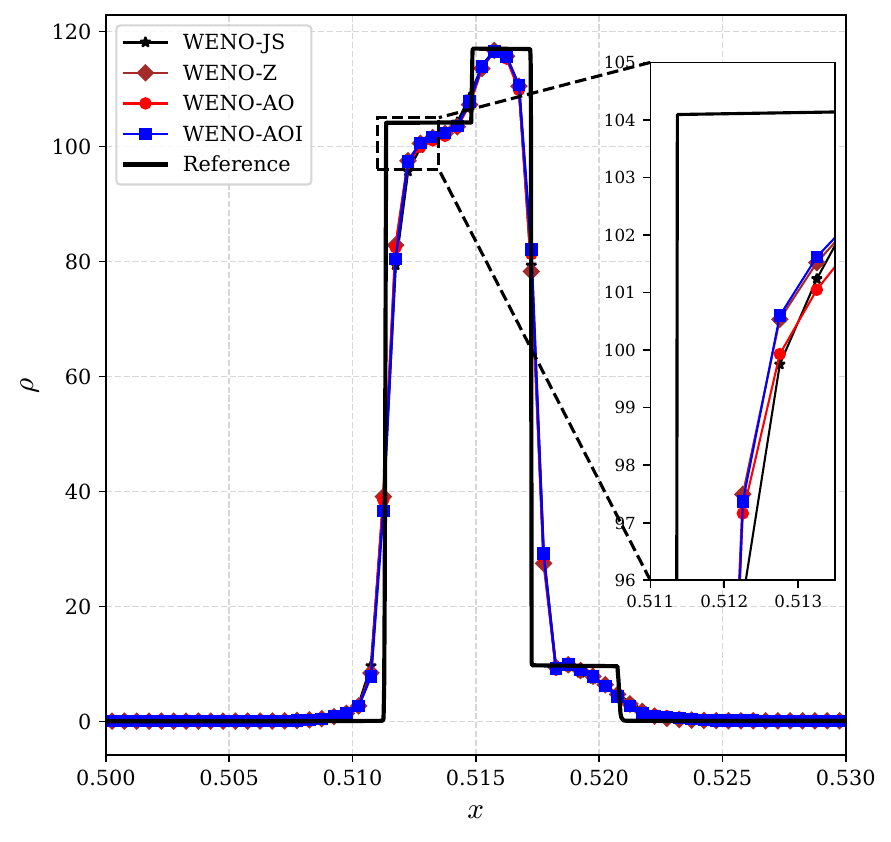}
			\caption{\id}
		\end{subfigure}
		\hfill
		\begin{subfigure}{0.48\textwidth}
			\centering
			\includegraphics[width=\textwidth]{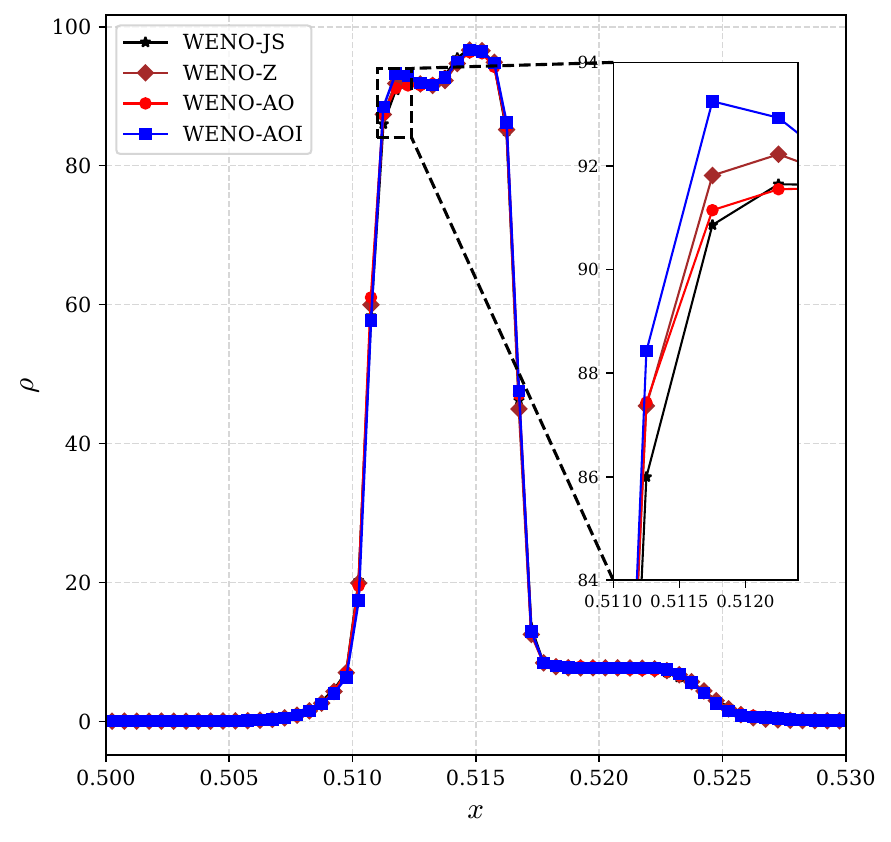}
			\caption{\ip}
		\end{subfigure}
		\\[6pt]
		\begin{subfigure}{0.48\textwidth}
			\centering
			\includegraphics[width=\textwidth]{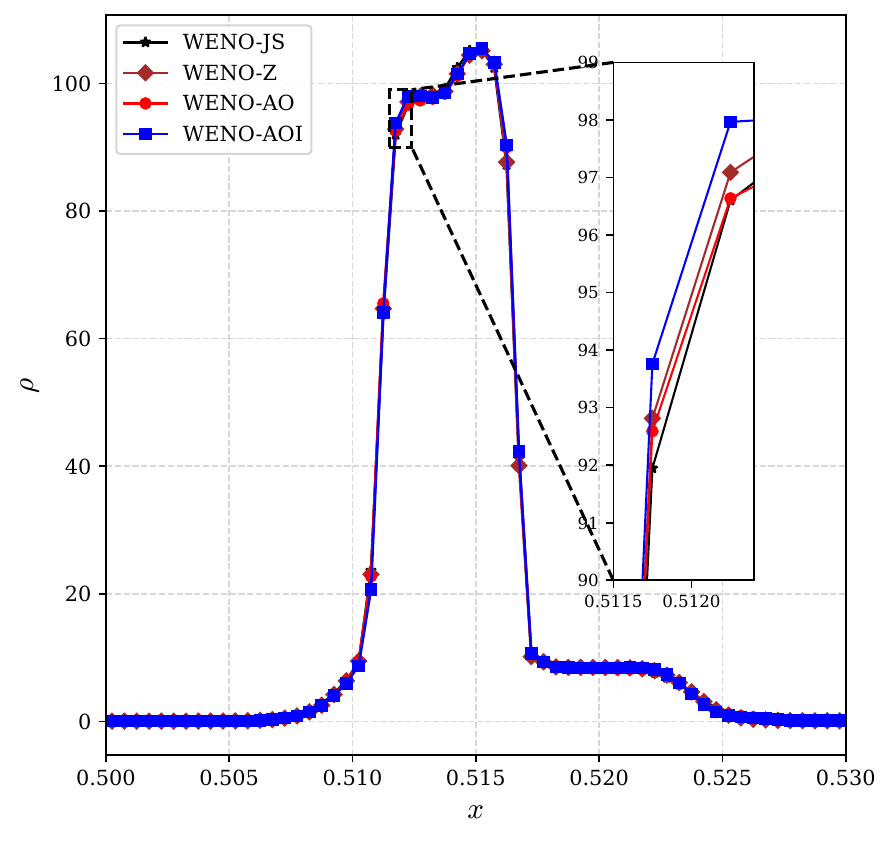}
			\caption{\rc}
		\end{subfigure}
		\hfill
		\begin{subfigure}{0.48\textwidth}
			\centering
			\includegraphics[width=\textwidth]{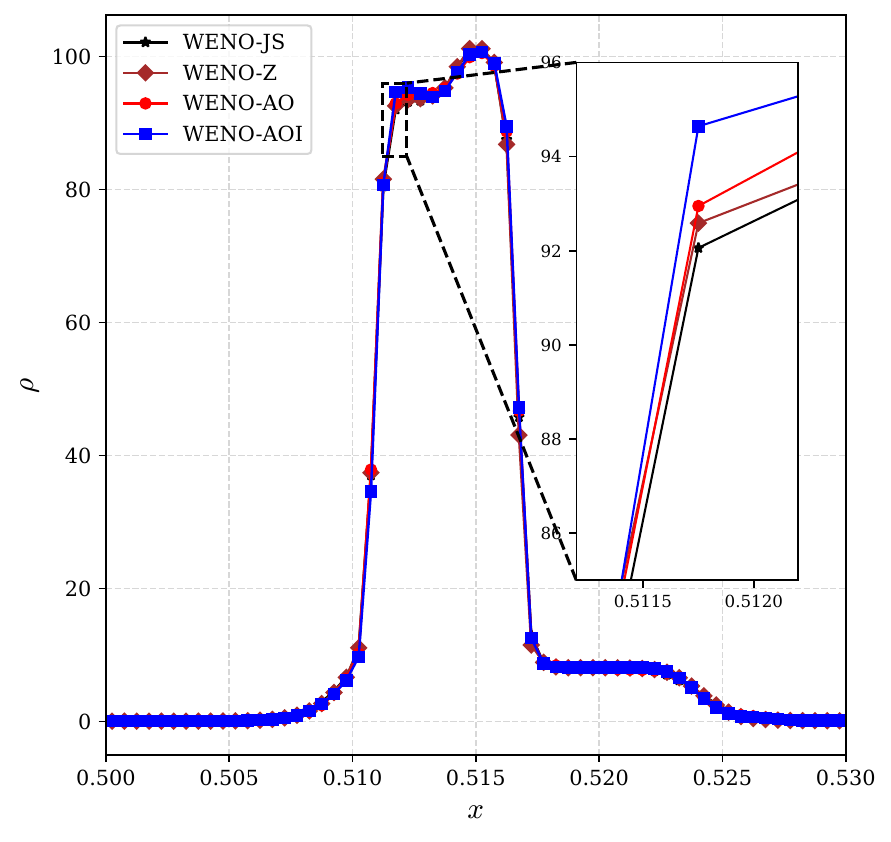}
			\caption{\tm}
		\end{subfigure}
		\caption{Numerical solutions for Test Problem~\ref{tp:BW} at $T=0.43$ on a $2000$-point mesh.}
		\label{fig:bw}
	\end{figure}
	\begin{figure}[ht]
		\centering
		\begin{subfigure}{0.48\textwidth}
			\centering
			\includegraphics[width=\textwidth]{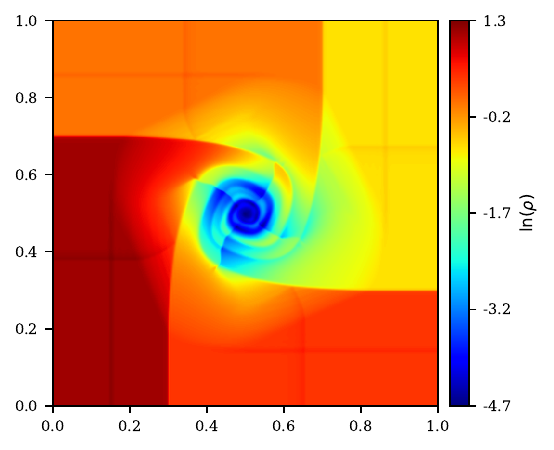}
			\caption{\id}
		\end{subfigure}
		\hfill
		\begin{subfigure}{0.48\textwidth}
			\centering
			\includegraphics[width=\textwidth]{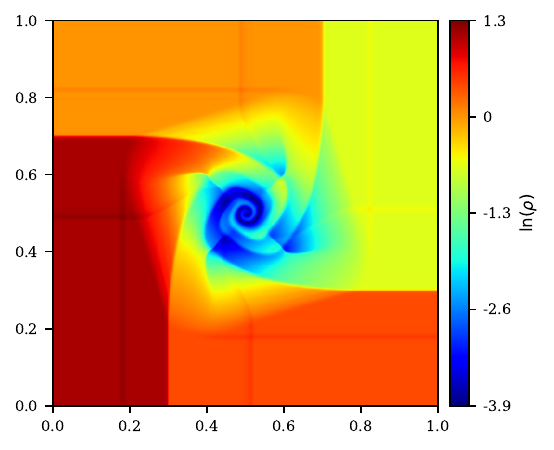}
			\caption{\ip}
		\end{subfigure}
		\\[6pt]
		\begin{subfigure}{0.48\textwidth}
			\centering
			\includegraphics[width=\textwidth]{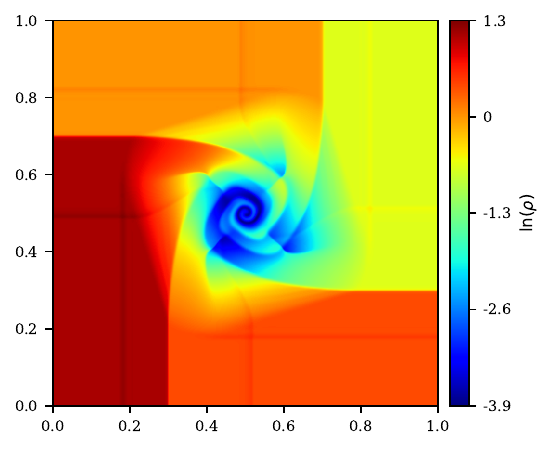}
			\caption{\rc}
		\end{subfigure}
		\hfill
		\begin{subfigure}{0.48\textwidth}
			\centering
			\includegraphics[width=\textwidth]{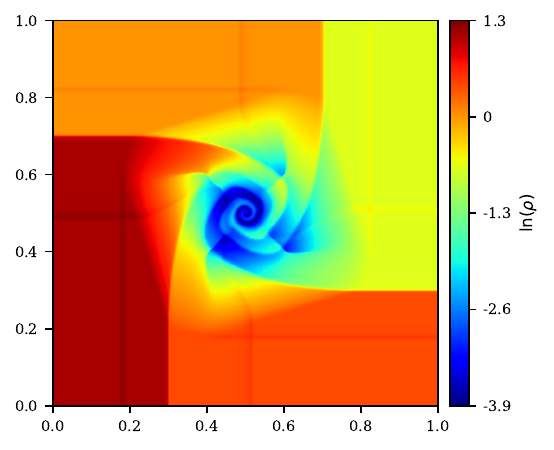}
			\caption{\tm}
		\end{subfigure}
		\caption{Plots of $\ln(\rho)$ for Test Problem~\ref{tp:rp2d1} at time $t=0.4$ on a $400\times400$ mesh for different equations of state.}
		\label{fig:rp2d1}
	\end{figure}
	\begin{testproblem}{Riemann Problem 4}
		\label{tp:rp4}
		We next consider another Riemann problem from \cite{zhang2006ram}. The computational domain is $[0,1]$, with outflow boundary conditions. The initial states are specified by
		\begin{align*}
			\left(\rho,\,u,\,p\right) =
			\begin{cases}
				\left(1,0.9,1\right), & \text{if } x < 0.5, \\
				\left(1,0,10\right), & \text{if } x > 0.5.
			\end{cases}
		\end{align*}
		The numerical results at $t = 0.4$ with $200$ grid points are shown in Figure \ref{fig:rp4}. We compare the numerical solutions obtained using the \wenojs, \wenoz, \wenoao, and \wenoaoi schemes with the exact solution. The numerical results accurately capture the shock waves and contact discontinuity present in the solution. Again, the \wenoaoi scheme provides slightly better resolution than the other schemes for all four EOS.
	\end{testproblem}
	\begin{testproblem}{Riemann Problem 5}
		\label{tp:rp5}
		We next test another Riemann problem from \cite{wu2015high}. The computational domain is $[0,1]$ with outflow boundary conditions. The initial states of this Riemann problem are
		\begin{align*}
			(\rho,u,p)=
			\begin{cases}
				(1,\,0,\,10^{4}), & \text{if } x<0.5,\\[2mm]
				(1,\,0,\,10^{-8}), & \text{if } x>0.5.
			\end{cases}
		\end{align*}
		The right state has an extremely low pressure. When we compare the left and right pressures, there is a very large pressure jump. The numerical result at $t=0.45$ is obtained using $200$ grid points and presented in Figure \ref{fig:rp5}. We compare the obtained numerical solution using the \wenojs, \wenoz, \wenoao, and \wenoaoi schemes with the exact solution. For all the EOS, the numerical results accurately capture all the essential solution features.
	\end{testproblem}
	\begin{figure}[ht]
		\centering
		\begin{subfigure}{0.48\textwidth}
			\centering
			\includegraphics[width=\textwidth]{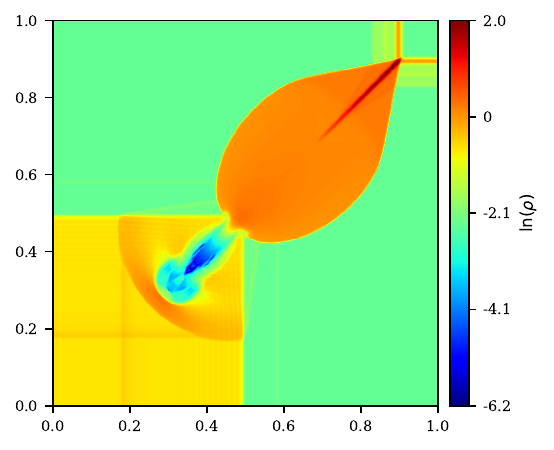}
			\caption{\id}
		\end{subfigure}
		\hfill
		\begin{subfigure}{0.48\textwidth}
			\centering
			\includegraphics[width=\textwidth]{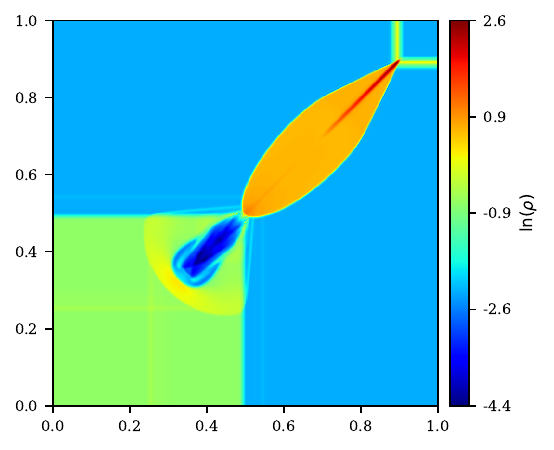}
			\caption{\ip}
		\end{subfigure}
		\\[6pt]
		\begin{subfigure}{0.48\textwidth}
			\centering
			\includegraphics[width=\textwidth]{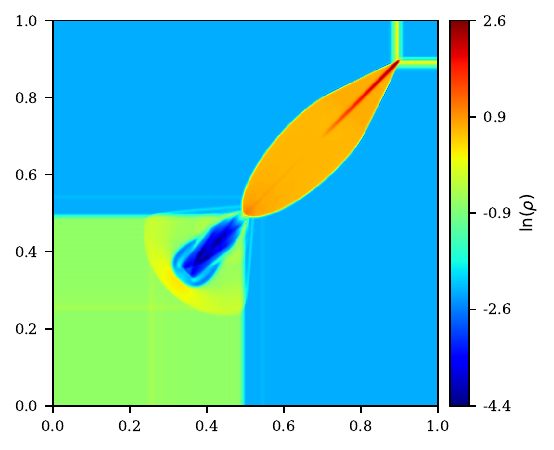}
			\caption{\rc}
		\end{subfigure}
		\hfill
		\begin{subfigure}{0.48\textwidth}
			\centering
			\includegraphics[width=\textwidth]{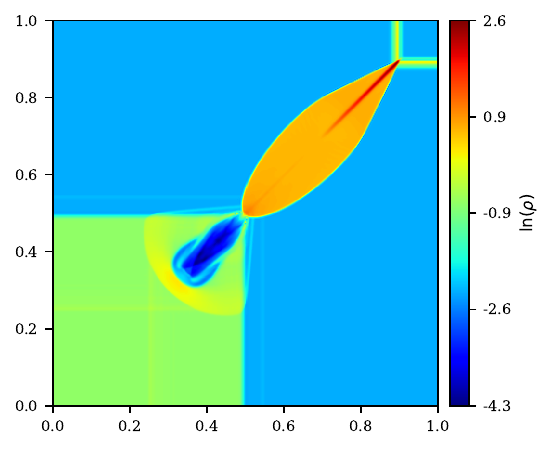}
			\caption{\tm}
		\end{subfigure}
		\caption{Results for Example~\ref{tp:rp2d2} at time $t=0.4$ on a $400\times400$ mesh. Panels show surface plot of $\ln(\rho)$ obtained using the WENO-AOI scheme with different equations of state.}
		\label{fig:rp2d2}
	\end{figure}
	\begin{testproblem}{Shu--Osher test}
		\label{tp:SO}
		The Shu--Osher problem is a classical benchmark used to assess the ability of high-order numerical methods to resolve small-scale sinusoidal wave structures. The computational domain is $[0,1]$ with outflow boundary conditions. The initial conditions are given by
		\begin{align*}
			(\rho,u,p)=
			\begin{cases}
				(5,\,0,\,50), & \text{if } x<0.5,\\[2mm]
				(2+0.3\sin(50x),\,0,\,5), & \text{if } x>0.5.
			\end{cases}
		\end{align*}
		The numerical solution at $t=0.35$ is computed using $800$ grid points and compared with a high-resolution reference solution in Figure~\ref{fig:so}. For all the considered EOS, the \wenojs, \wenoz, \wenoao, and \wenoaoi schemes capture the shock sharply while resolving the small-scale oscillatory structures in the post-shock region. The consistent results across all EOS demonstrate the robustness of the schemes, with \wenoaoi providing the better resolution among the schemes considered.
	\end{testproblem}
	\begin{testproblem}{Blast--Wave test}
		\label{tp:BW}
		This test problem is taken from \cite{wu2015high,wu2017physical}. It involves the interaction of waves generated by two strong initial discontinuities, making it a challenging benchmark for assessing the robustness and resolution of high-order numerical schemes. The computational domain is $[0,1]$, and the initial conditions are given by
		\begin{align*}
			(\rho,u,p)=
			\begin{cases}
				(1,\,0,\,10^{3}), & \text{if } x<0.1,\\[2mm]
				(1,\,0,\,10^{-2}), & \text{if } 0.1\le x<0.9,\\[2mm]
				(1,\,0,\,10^{2}), & \text{if } x\ge 0.9.
			\end{cases}
		\end{align*}
		with outflow boundary conditions. The numerical solution is computed at time $t=0.43$ on a mesh with $2000$ grid points and presented in Figure~\ref{fig:bw}. The numerical solutions obtained using the \wenojs, \wenoz, \wenoao, and \wenoaoi schemes are compared with the reference solution. For all the considered EOS, the schemes accurately capture the strong shocks and complex wave interactions, showing close agreement with the reference. In particular, all schemes maintain PCP throughout the computation, demonstrating their robustness for strong discontinuities and challenging wave interactions across different EOS.
	\end{testproblem}
	\begin{figure}[ht]
		\centering
		\begin{subfigure}{0.48\textwidth}
			\centering
			\includegraphics[width=\textwidth]{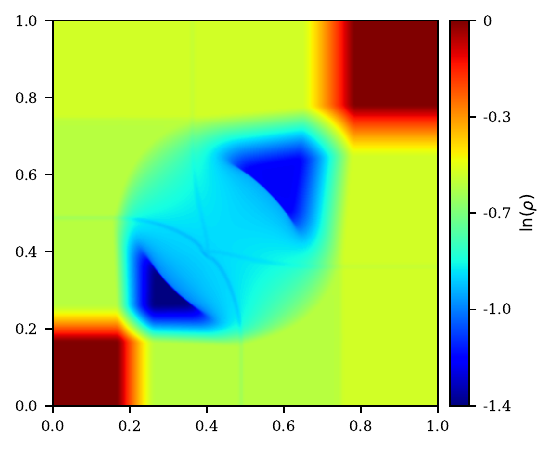}
			\caption{\id}
		\end{subfigure}
		\hfill
		\begin{subfigure}{0.48\textwidth}
			\centering
			\includegraphics[width=\textwidth]{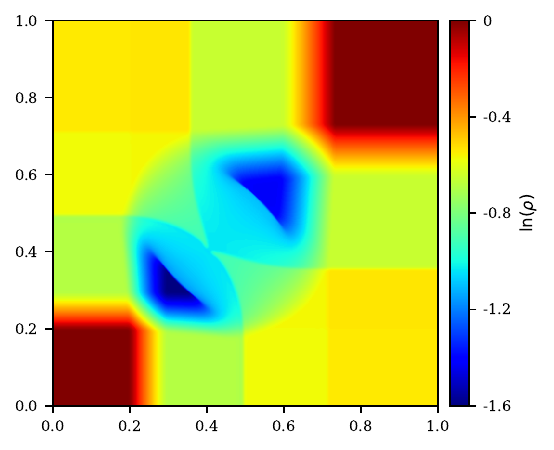}
			\caption{\ip}
		\end{subfigure}
		\\[6pt]
		\begin{subfigure}{0.48\textwidth}
			\centering
			\includegraphics[width=\textwidth]{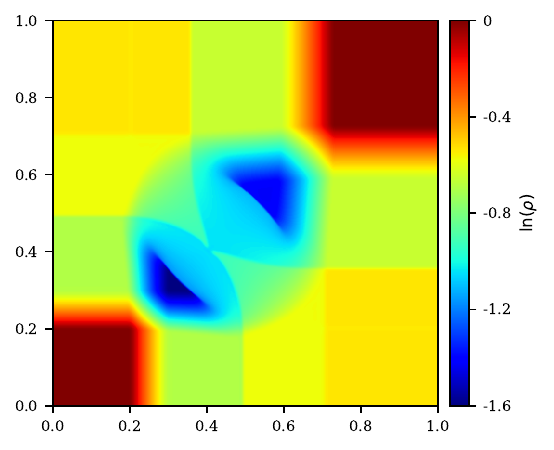}
			\caption{\rc}
		\end{subfigure}
		\hfill
		\begin{subfigure}{0.48\textwidth}
			\centering
			\includegraphics[width=\textwidth]{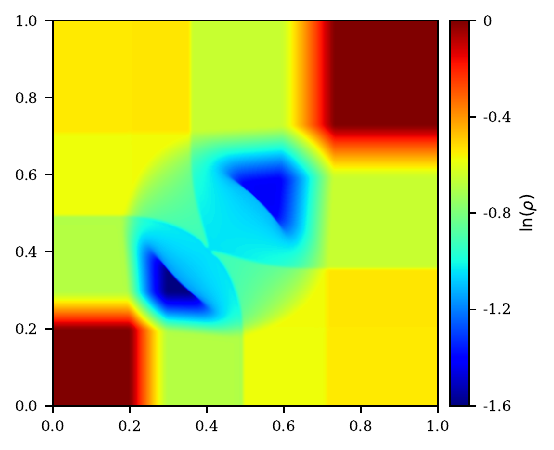}
			\caption{\tm}
		\end{subfigure}
		\caption{Results for Example~\ref{tp:rp2d3} at time $t=0.4$ on a $400\times400$ mesh. Panels show surface plot of $\ln(\rho)$ obtained using the WENO-AOI scheme with different equations of state.}
		\label{fig:rp2d3}
	\end{figure}
	\begin{figure}[ht!]
		\centering
		\begin{subfigure}{0.48\textwidth}
			\centering
			\includegraphics[width=\textwidth]{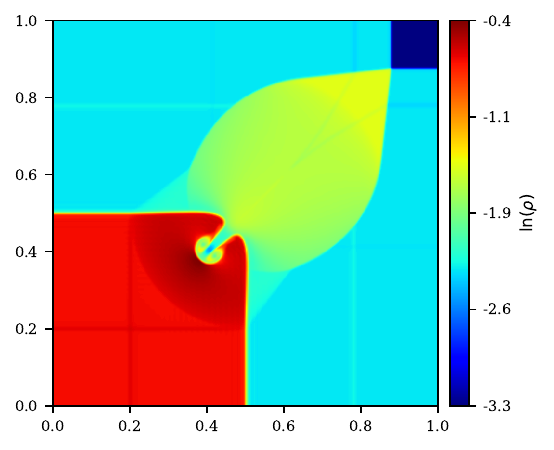}
			\caption{\id}
		\end{subfigure}
		\hfill
		\begin{subfigure}{0.48\textwidth}
			\centering
			\includegraphics[width=\textwidth]{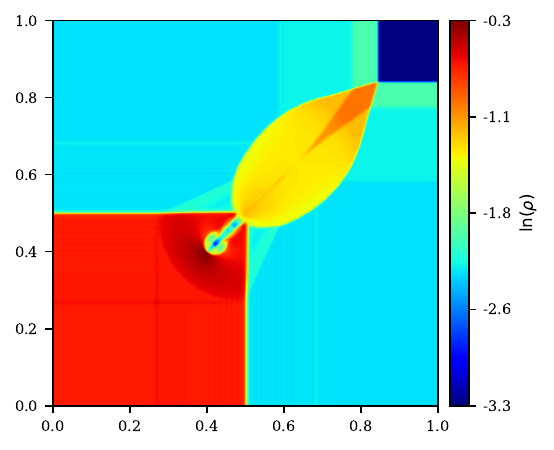}
			\caption{\ip}
		\end{subfigure}
		\\[6pt]
		\begin{subfigure}{0.48\textwidth}
			\centering
			\includegraphics[width=\textwidth]{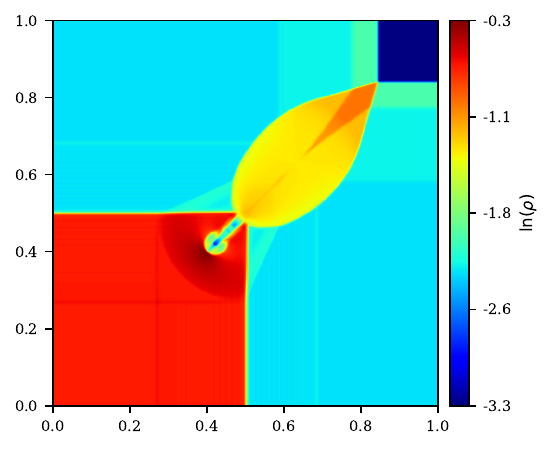}
			\caption{\rc}
		\end{subfigure}
		\hfill
		\begin{subfigure}{0.48\textwidth}
			\centering
			\includegraphics[width=\textwidth]{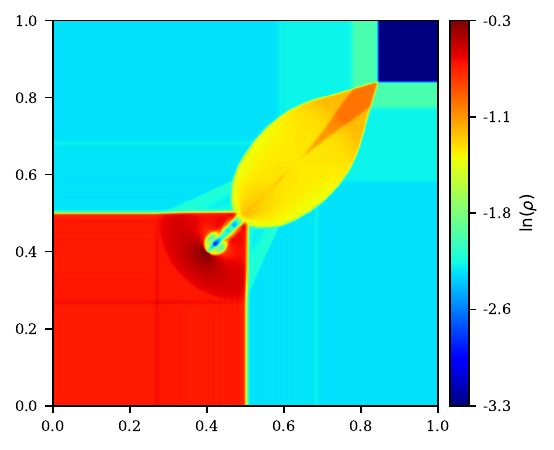}
			\caption{\tm}
		\end{subfigure}
		\caption{Results for Example~\ref{tp:rp2d4} at time $t=0.4$ on a $400\times400$ mesh. Panels show surface plot of $\ln(\rho)$ obtained using the WENO-AOI scheme with different equations of state.}
		\label{fig:rp2d4}
	\end{figure}
	\subsection{Two dimensional numerical tests}\label{sec:num2d}
	We now present two dimensional test cases. We start with standard two-dimensional Riemann problems and then move on to more complex two-dimensional problems. The numerical results are obtained using \wenoaoi schemes with all equations of state presented in Section~\ref{subsec:eos}.
	\begin{testproblem}{Two-dimensional Riemann Problem 1}
		\label{tp:rp2d1}	
		This test is a two-dimensional Riemann problem from \cite{nunez2016xtroem}. The computational domain is $[0,1]\times[0,1]$, and the initial data are given by
		
		\begin{equation*}
			\left(\rho,\,u_x,\,u_y,\,p\right)=
			\begin{cases}
				\left(0.5,\,0.5,\,-0.5,\,5\right), & \text{if } x>0.5 \text{ and } y>0.5,\\
				\left(1,\,0.5,\,0.5,\,5\right), & \text{if } x<0.5 \text{ and } y>0.5,\\
				\left(3,\,-0.5,\,0.5,\,5\right), & \text{if } x<0.5 \text{ and } y<0.5,\\
				\left(1.5,\,-0.5,\,-0.5,\,5\right), & \text{if } x>0.5 \text{ and } y<0.5.
			\end{cases}
		\end{equation*}
		
		The simulations are performed on a uniform $400\times400$ mesh with outflow boundary conditions. For the \id, the ratio of specific heats is fixed at $\gamma=5/3$. Numerical results obtained using the \wenoaoi scheme at time $t=0.4$ are presented in Figure~\ref{fig:rp2d1}.
		
		The interaction of the four initial discontinuities generates a complex wave pattern characterized by a central low-density region and a spiral-like vortex structure. While the overall flow features are similar for all four equations of state, noticeable differences appear in the detailed structure and intensity of the low-density region. In particular, the solution obtained with the \id exhibits a more pronounced low-density core and a slightly different vortex structure, whereas the results corresponding to the \ip, \rc, and \tm are qualitatively very similar. The \wenoaoi scheme captures these complex flow structures for all four equations of state.
	\end{testproblem}
	\begin{figure}[ht!]
		\centering
		\begin{subfigure}{0.48\textwidth}
			\centering
			\includegraphics[width=\textwidth]{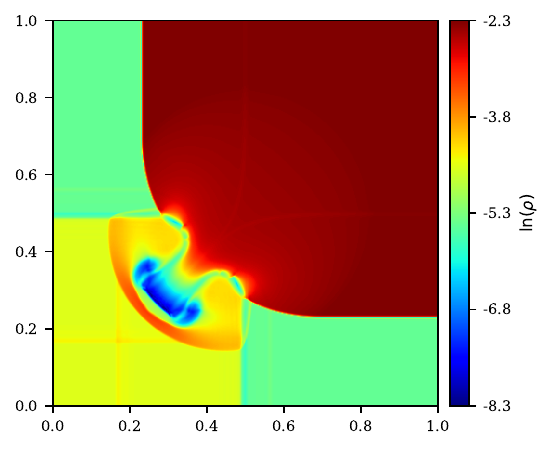}
			\caption{\id}
		\end{subfigure}
		\hfill
		\begin{subfigure}{0.48\textwidth}
			\centering
			\includegraphics[width=\textwidth]{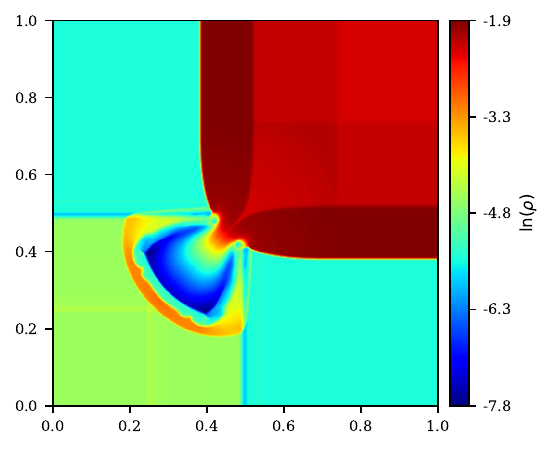}
			\caption{\ip}
		\end{subfigure}
		\\[6pt]
		\begin{subfigure}{0.48\textwidth}
			\centering
			\includegraphics[width=\textwidth]{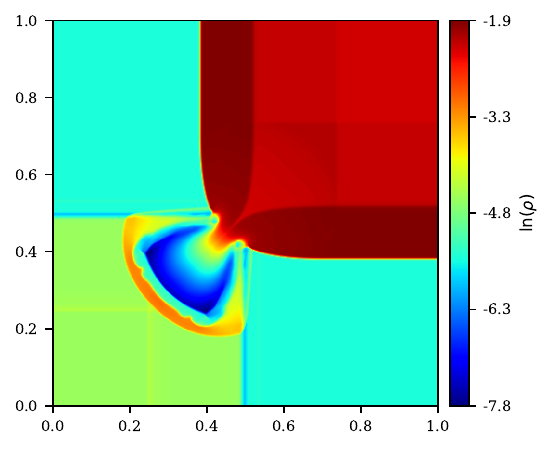}
			\caption{\rc}
		\end{subfigure}
		\hfill
		\begin{subfigure}{0.48\textwidth}
			\centering
			\includegraphics[width=\textwidth]{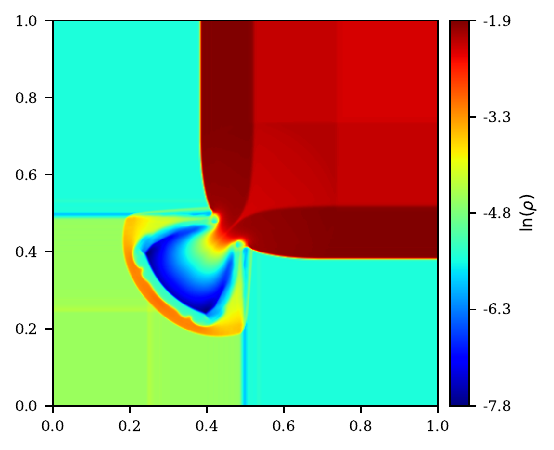}
			\caption{\tm}
		\end{subfigure}
		\caption{Results for Example~\ref{tp:rp2d5} at time $t=0.4$ on a $400\times400$ mesh. Panels show surface plot of $\ln(\rho)$ obtained using the WENO-AOI scheme with different equations of state.}
		\label{fig:rp2d5}
	\end{figure}
	\begin{testproblem}{Two-dimensional Riemann Problem 2}
		\label{tp:rp2d2}
		We consider another two-dimensional Riemann problem similar to the one in \cite{del2002efficient}. The computational domain $[0,1]\times[0,1]$ is initialized with the following states:
		\begin{equation*}
			\left(\rho,\,u_x,\,u_y,\,p\right)=
			\begin{cases}
				\left(0.1,\,0,\,0,\,0.01\right), & \text{if } x>0.5 \text{ and } y>0.5,\\
				\left(0.1,\,0.99,\,0,\,1\right), & \text{if } x<0.5 \text{ and } y>0.5,\\
				\left(0.5,\,0,\,0,\,1\right), & \text{if } x<0.5 \text{ and } y<0.5,\\
				\left(0.1,\,0,\,0.99,\,1\right), & \text{if } x>0.5 \text{ and } y<0.5.
			\end{cases}
		\end{equation*}
		The solution involves the interaction of two shocks and two contact discontinuities, resulting in a complex wave pattern. The simulations are performed on a uniform $400\times400$ mesh with outflow boundary conditions. Numerical results obtained using the \wenoaoi scheme at time $t=0.4$ are presented in Figure~\ref{fig:rp2d2}.
		
		A complex two-dimensional wave structure develops, featuring a high-density region extending from the center toward the upper-right part of the domain and a low-density region in the lower-left region. The solution with \id differs noticeably from those obtained with the \ip, \rc, and \tm, which exhibit similar overall flow features. The \wenoaoi scheme captures the complex wave interactions and the associated sharp structures consistently for all four equations of state.
		
	\end{testproblem}
	\begin{testproblem}{Two-dimensional Riemann Problem 3}
		\label{tp:rp2d3}
		We next consider a two-dimensional Riemann problem from \cite{he2012adaptive}. The computational domain $[0,1]\times[0,1]$ is initialized with the following states:
		
		\begin{equation*}
			\left(\rho,\,u_x,\,u_y,\,p\right)=
			\begin{cases}
				\left(1,\,0,\,0,\,1\right), & \text{if } x>0.5 \text{ and } y>0.5,\\
				\left(0.5771,\,-0.3529,\,0,\,0.4\right), & \text{if } x<0.5 \text{ and } y>0.5,\\
				\left(1,\,-0.3529,\,-0.3529,\,1\right), & \text{if } x<0.5 \text{ and } y<0.5,\\
				\left(0.5771,\,0,\,-0.3529,\,0.4\right), & \text{if } x>0.5 \text{ and } y<0.5.
			\end{cases}
		\end{equation*}
		
		The solution involves the interaction of two rarefaction waves, resulting in the formation of two symmetric shock waves. The simulations are performed on a uniform $400\times400$ mesh with outflow boundary conditions. Numerical results obtained using the \wenoaoi scheme at time $t=0.4$ are presented in Figure~\ref{fig:rp2d3}.
		
		The resulting flow contains both rarefaction and shock structures. All four equations of state yield nearly identical flow patterns, with no significant differences observed in the overall wave structures. The \wenoaoi scheme captures these wave interactions consistently for all four equations of state.
		
	\end{testproblem}
	\begin{figure}[ht!]
		\centering
		\begin{subfigure}{0.48\textwidth}
			\centering
			\includegraphics[width=\textwidth]{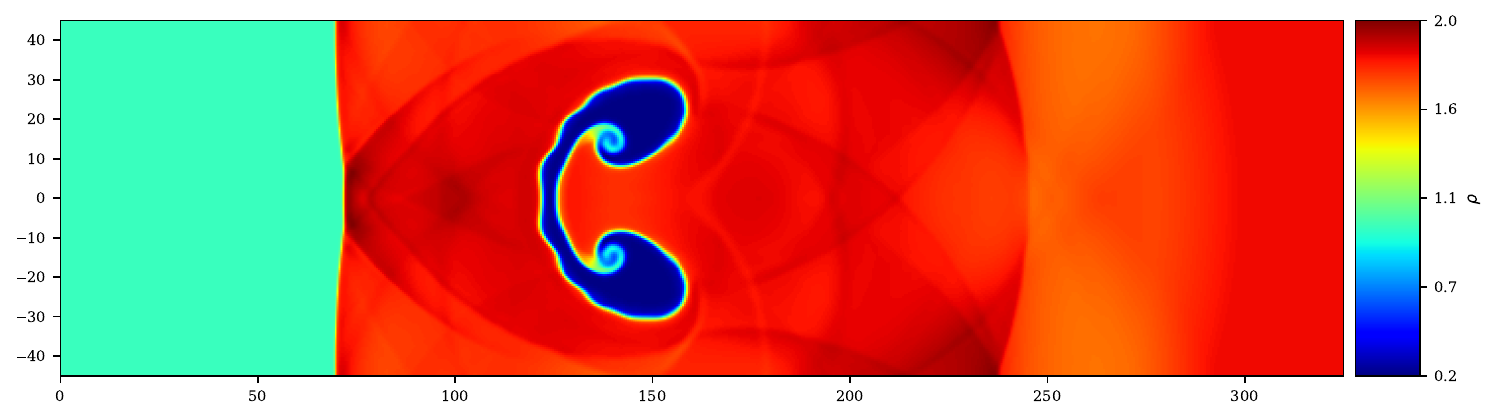}
			\caption{\id}
		\end{subfigure}
		\hfill
		\begin{subfigure}{0.48\textwidth}
			\centering
			\includegraphics[width=\textwidth]{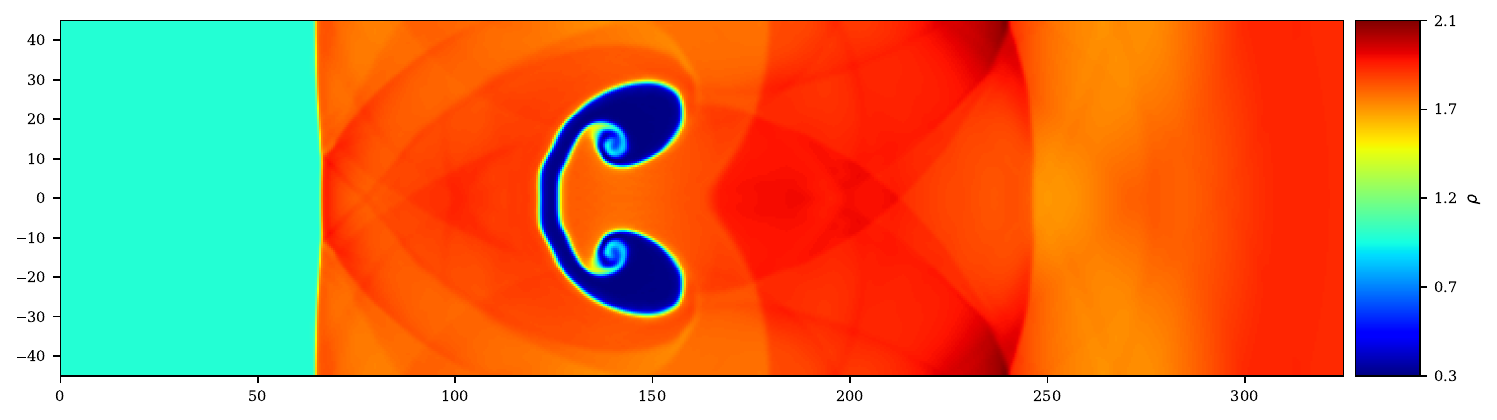}
			\caption{\ip}
		\end{subfigure}
		\\[6pt]
		\begin{subfigure}{0.48\textwidth}
			\centering
			\includegraphics[width=\textwidth]{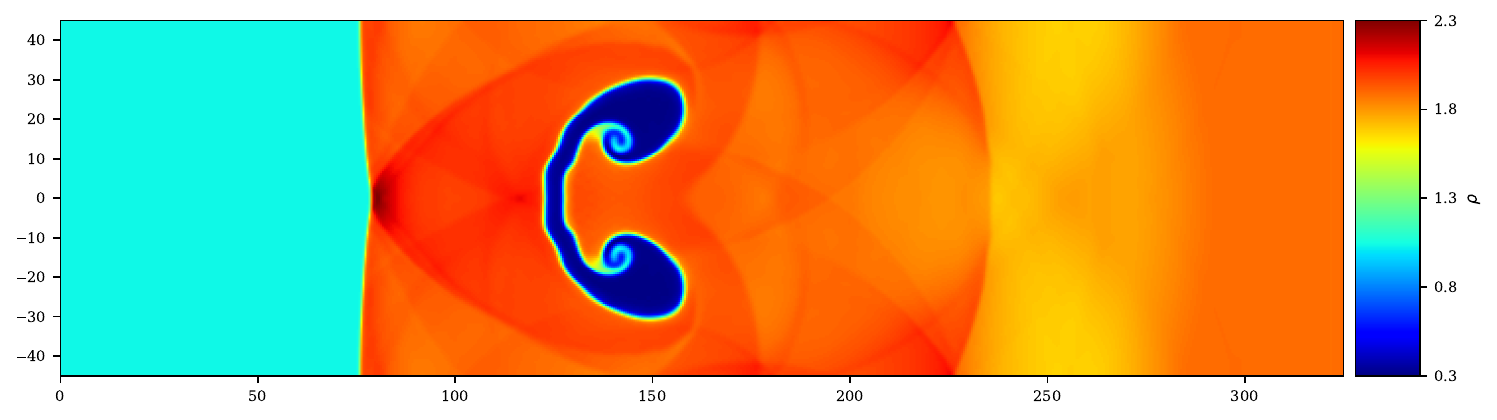}
			\caption{\rc}
		\end{subfigure}
		\hfill
		\begin{subfigure}{0.48\textwidth}
			\centering
			\includegraphics[width=\textwidth]{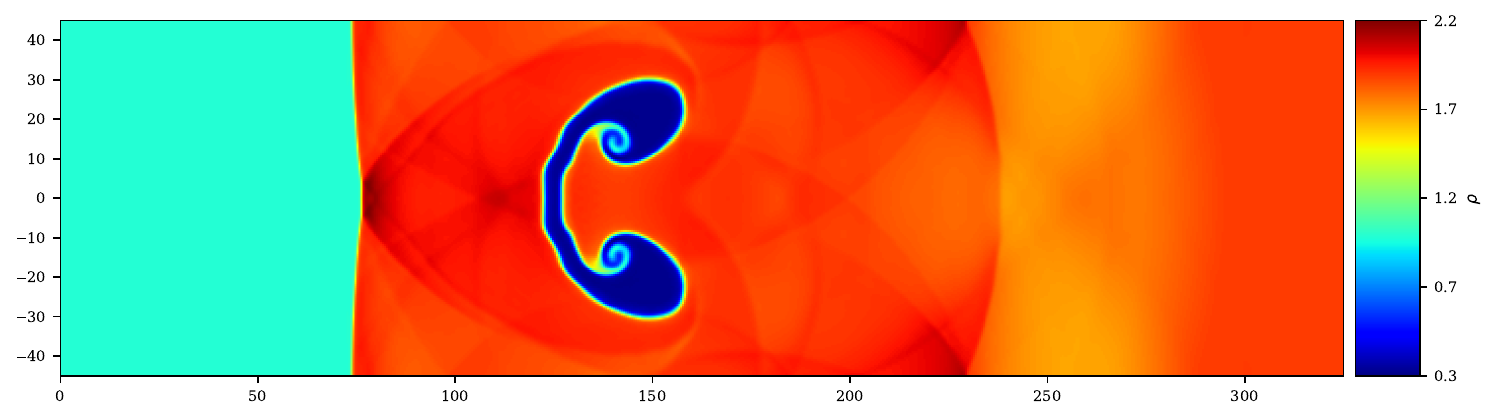}
			\caption{\tm}
		\end{subfigure}
		\caption{Density distribution for Example~\ref{tp:sbi1} at time $t=450$ on a $650\times180$ mesh, obtained using the WENO-AOI scheme with different equations of state.}
		\label{fig:pbsb1}
	\end{figure}
	\begin{testproblem}{Two-dimensional Riemann Problem 4}
		\label{tp:rp2d4}
		This test case is taken from \cite{nunez2016xtroem}. The computational domain is $[0,1]\times[0,1]$, and the initial conditions are given by
		
		\begin{equation*}
			\left(\rho,\,u_x,\,u_y,\,p\right)=
			\begin{cases}
				(0.035145216124503,\,0,\,0,\,0.162931056509027), & \text{if } x>0.5 \text{ and } y>0.5,\\
				(0.1,\,0.7,\,0,\,1), & \text{if } x<0.5 \text{ and } y>0.5,\\
				(0.5,\,0,\,0,\,1), & \text{if } x<0.5 \text{ and } y<0.5,\\
				(0.1,\,0,\,0.7,\,1), & \text{if } x>0.5 \text{ and } y<0.5.
			\end{cases}
		\end{equation*}
		
		The simulations are performed on a uniform $400\times400$ mesh with outflow boundary conditions. Numerical results obtained using the \wenoaoi scheme at time $t=0.4$ are presented in Figure~\ref{fig:rp2d4}.
		
		A complex wave pattern develops, featuring a curved high-density structure extending toward the upper-right region of the domain and a low-density region near the center. While the solutions obtained with the four equations of state capture the same overall flow features, noticeable differences are observed in the detailed wave structures. In particular, the \id solution differs visibly from the other three, whereas the results corresponding to the \ip, \rc, and \tm exhibit more similar flow patterns. The \wenoaoi scheme successfully resolves these complex flow structures for all four equations of state.
		
	\end{testproblem}
	\begin{figure}[ht!]
		\centering
		\begin{subfigure}{0.48\textwidth}
			\centering
			\includegraphics[width=\textwidth]{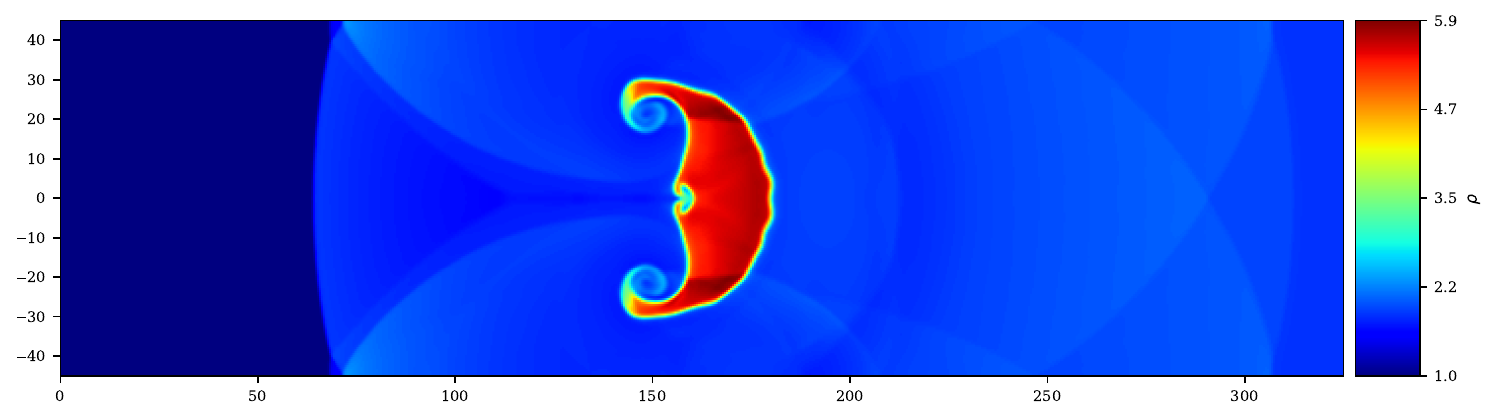}
			\caption{\id}
		\end{subfigure}
		\hfill
		\begin{subfigure}{0.48\textwidth}
			\centering
			\includegraphics[width=\textwidth]{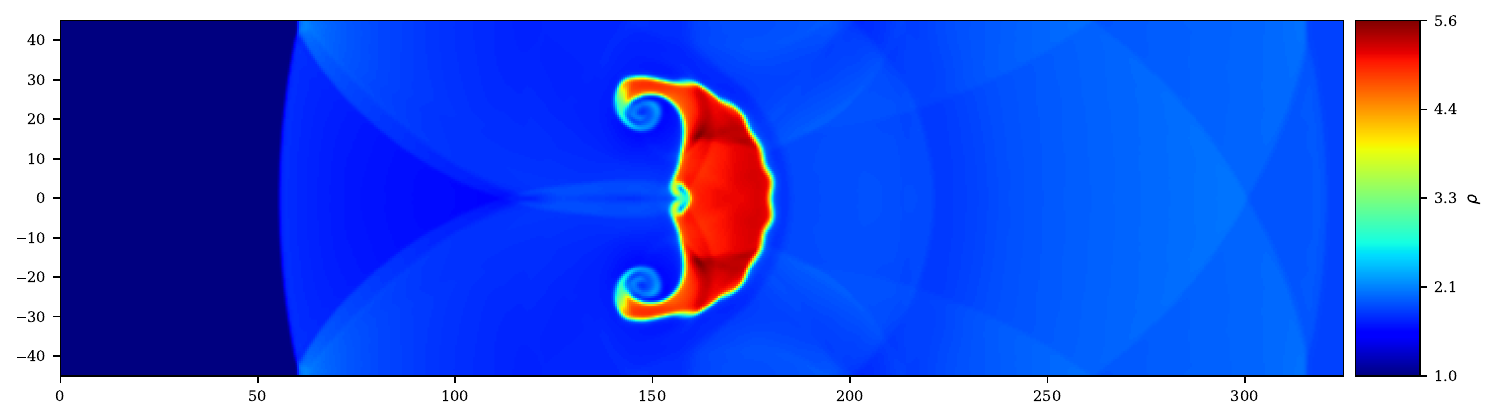}
			\caption{\ip}
		\end{subfigure}
		\\[6pt]
		\begin{subfigure}{0.48\textwidth}
			\centering
			\includegraphics[width=\textwidth]{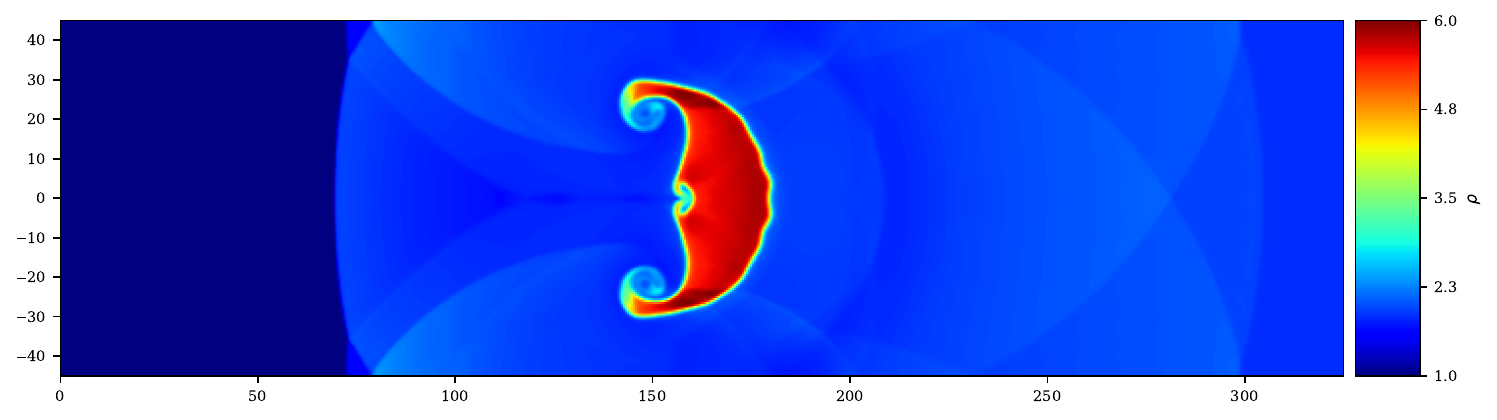}
			\caption{\rc}
		\end{subfigure}
		\hfill
		\begin{subfigure}{0.48\textwidth}
			\centering
			\includegraphics[width=\textwidth]{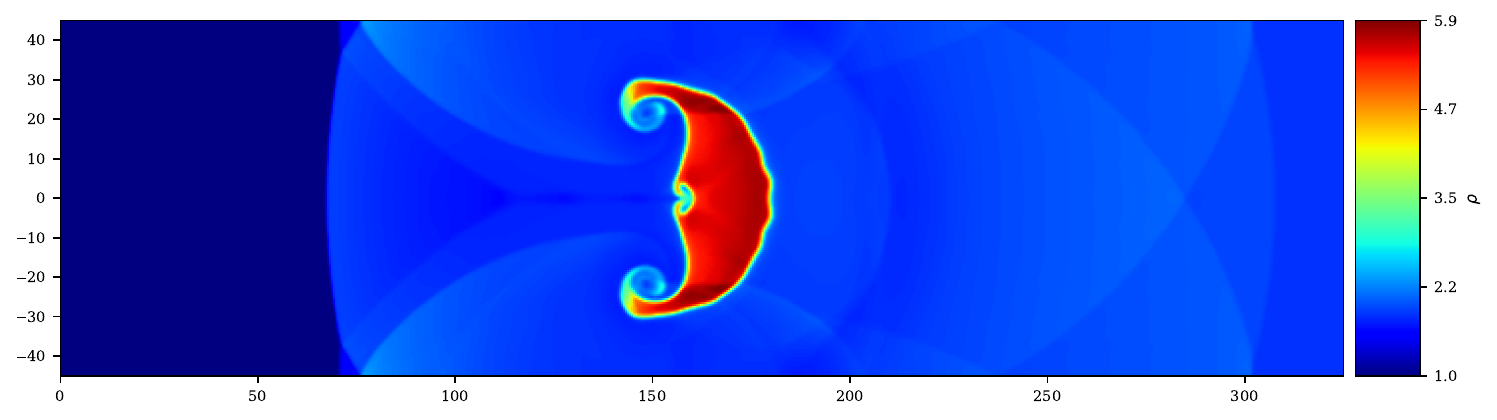}
			\caption{\tm}
		\end{subfigure}
		\caption{Density distribution for Example~\ref{tp:sbi2} at time $t=500$ on a $650\times180$ mesh, obtained using the WENO-AOI scheme with different equations of state.}
		\label{fig:sbi2}
	\end{figure}
	\begin{testproblem}{Two-dimensional Riemann Problem 5}
		\label{tp:rp2d5}	
		In this test case, we again consider the computational domain $[0,1]\times[0,1]$ with outflow boundary conditions imposed on all sides. The initial conditions are given by
		
		\begin{equation*}
			\left(\rho,\,u_x,\,u_y,\,p\right)=
			\begin{cases}
				\left(0.1,\,0,\,0,\,20.0\right), & \text{if } x\geq 0.5 \text{ and } y\geq 0.5,\\
				\left(0.00414329639576,\,0.9946418833556542,\,0,\,0.05\right), & \text{if } x<0.5 \text{ and } y\geq 0.5,\\
				\left(0.01,\,0,\,0,\,0.05\right), & \text{if } x<0.5 \text{ and } y<0.5,\\
				\left(0.00414329639576,\,0,\,0.9946418833556542,\,0.05\right), & \text{if } x\geq 0.5 \text{ and } y<0.5.
			\end{cases}
		\end{equation*}
		
		The simulations are performed on a uniform $400\times400$ mesh with outflow boundary conditions. Numerical results obtained using the \wenoaoi scheme at time $t=0.4$ are presented in Figure~\ref{fig:rp2d5}.
		
		The strong initial discontinuities give rise to a complex wave pattern, with a pronounced low-density region developing near the center of the domain. The solutions obtained with the \ip, \rc, and \tm exhibit very similar flow structures, whereas the \id solution shows noticeable differences in the shape and structure of the central low-density region. The \wenoaoi scheme captures the main flow features and resolves the complex wave interactions for all four equations of state.
		
	\end{testproblem}
	\begin{figure}[ht!]
		\centering
		\begin{subfigure}{0.48\textwidth}
			\centering
			\includegraphics[width=\textwidth]{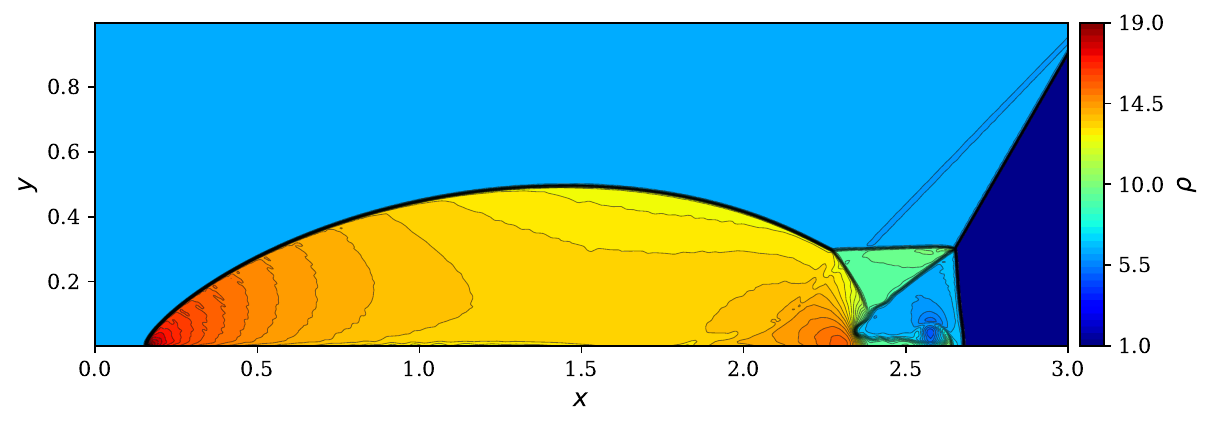}
			\caption{\id}
		\end{subfigure}
		\hfill
		\begin{subfigure}{0.48\textwidth}
			\centering
			\includegraphics[width=\textwidth]{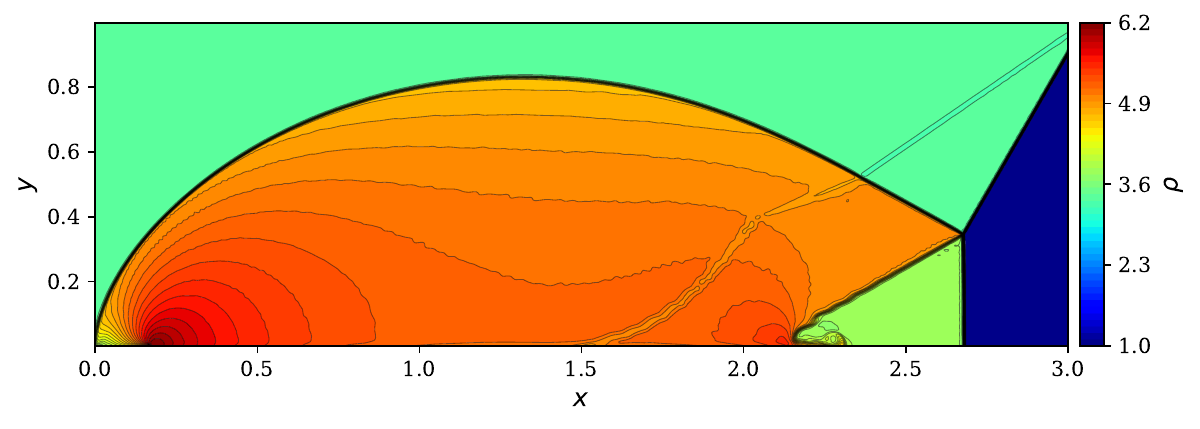}
			\caption{\ip}
		\end{subfigure}
		\\[6pt]
		\begin{subfigure}{0.48\textwidth}
			\centering
			\includegraphics[width=\textwidth]{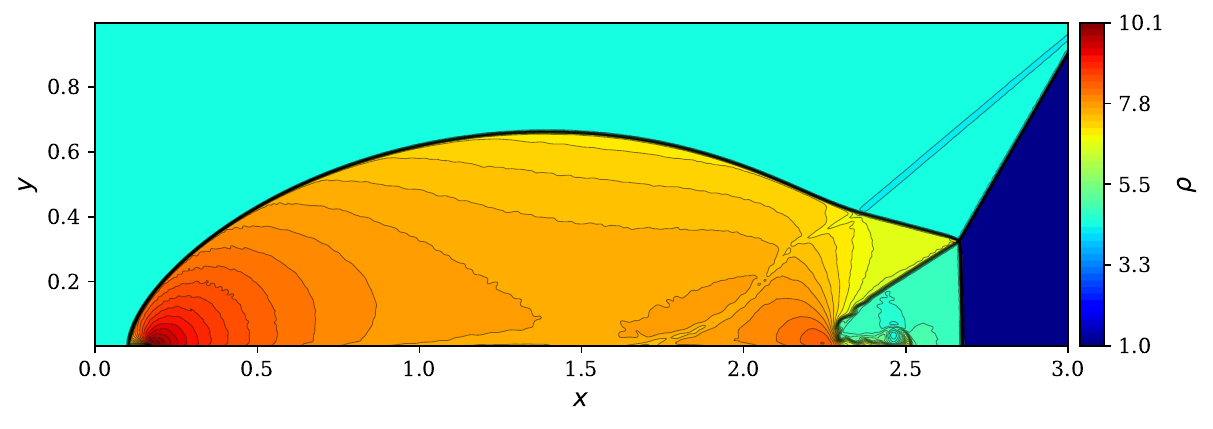}
			\caption{\rc}
		\end{subfigure}
		\hfill
		\begin{subfigure}{0.48\textwidth}
			\centering
			\includegraphics[width=\textwidth]{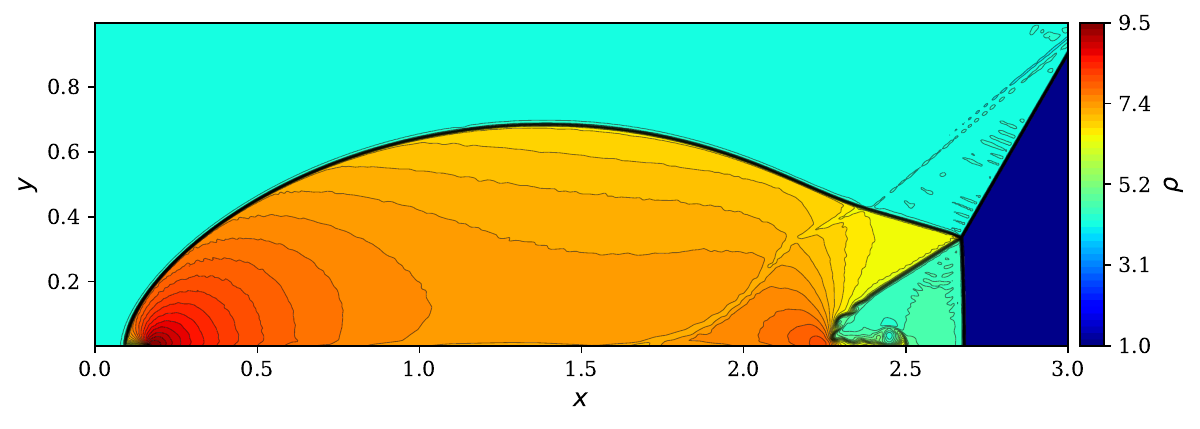}
			\caption{\tm}
		\end{subfigure}
		\caption{Density contours for the double Mach reflection problem at time $t=4.0$ obtained using the \wenoaoi scheme on a $960 \times 240$ mesh.}
		\label{fig:dmr}
	\end{figure}
	\begin{testproblem}{Shock-Bubble Interaction Problem I}
		\label{tp:sbi1}
		
		In this test case, we examine the interaction between a left-propagating shock wave and a gaseous bubble. The full problem description can be found in \cite{he2012adaptive}. The computational domain is taken as $[0,325]\times[-45,45]$, with reflective boundary conditions imposed at $y=\pm45$. At the left boundary $(x=0)$ and the right boundary $(x=325)$, we prescribe the corresponding constant post-shock and pre-shock states, respectively. The initial shock is defined by
		\begin{equation*}
			\left(\rho,\,u_x,\,u_y,\,p\right)=
			\begin{cases}
				(1,\,0,\,0,\,0.05), & \text{if } x<265,\\
				(1.865225080631180,\,-0.196781107378299,\,0,\,0.15), & \text{if } x>265.
			\end{cases}
		\end{equation*}
		A bubble of radius $25$, centered at $(215,0)$, is placed with the initial state
		\begin{equation*}
			\left(\rho,\,u_x,\,u_y,\,p\right)
			=(0.1358,\,0,\,0,\,0.05).
		\end{equation*}
		The simulations are performed on a $650\times180$ mesh. In this test, we investigate the influence of four different equations of state, namely ID-EOS, IP-EOS, RC-EOS, and TM-EOS, using the WENO-AOI scheme. Numerical results at time $t=450$ are presented in Figure~\ref{fig:pbsb1}. The interaction between the shock wave and the bubble produces complex flow structures, including transmitted and reflected waves, significant bubble deformation, and the formation of vortical structures. While the choice of EOS leads to noticeable differences in the surrounding wave patterns and density distributions, the overall bubble evolution and the dominant vortical structures remain qualitatively consistent across the four EOS. The WENO-AOI scheme accurately resolves these features for all considered EOS, demonstrating its robustness in handling shock--bubble interactions with different thermodynamic models.
	\end{testproblem}
	
	
	\begin{testproblem}{Shock-Bubble Interaction Problem II}
		\label{tp:sbi2}
		
		In this test case, we consider another variant of the shock--bubble interaction problem from \cite{he2012adaptive}. The computational setup is identical to that of Example \ref{tp:sbi1}, except that the bubble now contains a heavier fluid. The initial state inside the bubble is specified by
		\begin{equation*}
			\left(\rho,\,u_x,\,u_y,\,p\right)
			=(3.1538,\,0,\,0,\,0.05).
		\end{equation*}
		The simulations are performed on a $650\times180$ mesh. In this test, we again investigate the influence of the four equations of state using the WENO-AOI scheme. Numerical results at time $t=500$ are presented in Figure~\ref{fig:sbi2}. The presence of the heavier fluid substantially alters the shock--bubble dynamics, leading to a strongly deformed bubble with pronounced roll-up of the upper and lower portions of the interface. The density contours also reveal differences in the transmitted wave structure and the surrounding flow field for the different EOS. Nevertheless, the main features of the interaction, including the deformed bubble interface and the associated vortical structures, are consistently resolved by the WENO-AOI scheme for all four EOS.
	\end{testproblem}
	
	\begin{testproblem}{Double Mach Reflection Problem}
		\label{tp:dmr}
		This problem represents a standard, highly challenging test for shock-capturing schemes. Originally introduced by Woodward and Colella \cite{woodward1984numerical} for classical hydrodynamics, it was extended to RHD with the \id in \cite{wu2015high}, and later it was used in several articles including \cite{basak2025constraints,cao2025robust,he2012adaptive,zhao2013runge}. It features a strong shock propagating in a two-dimensional channel that hits a wedge, initiating a complex double Mach reflection process with fine-scale structures. 
		
		The computational domain is defined as $[0, 4] \times [0, 1]$. Initially, a planar shock front is set at an angle of $60^\circ$ to the $x$-axis, starting from $x_s(y) = \frac{1}{6} + \frac{y}{\sqrt{3}}$ at $t=0$. The shock propagates to the right with a speed $V_s = 0.5$. The pre-shock (right state) is set globally for all EOS as
		\begin{align*}
			\left(\rho_R, u_{x,R}, u_{y,R}, p_R\right) = \left(1.0, 0.0, 0.0, 0.001\right).
		\end{align*}
		The post-shock (left state) values are computed using the shock speed $V_s = 0.5$ and the Rankine-Hugoniot (RH) jump conditions. In this work, these post-shock states are tabulated in Table \ref{tab:dmr_postshock}, extending the numerical test to \tm, \ip, and \rc under these initial conditions.
		
		For the boundary conditions, the left boundary ($x=0$) is an inflow boundary assigned to the post-shock state, while the right boundary ($x=4$) is an outflow boundary. On the bottom boundary ($y=0$), a reflective wall condition is applied for $x \geq 1/6$, and an inflow boundary condition assigned to the post-shock state is applied for $x < 1/6$. The top boundary ($y=1$) is configured to track the exact motion of the shock front by assigning the post-shock state for $x < x_s(1, t)$ and the pre-shock state for $x \ge x_s(1, t)$, where the shock front position is $x_s(y, t) = \frac{1}{6} + \frac{y + 2 V_s t}{\sqrt{3}}$. The computations are run up to the final time $t = 4.0$ and plotted in Figure \ref{fig:dmr}.
		\begin{table}[ht]
			\centering
			\footnotesize
			\setlength{\tabcolsep}{4pt}
			\caption{Post-shock states (left states) for the DMR problem under different EOS.}
			\label{tab:dmr_postshock}
			\begin{tabular}{lcccc}
				\toprule
				Equation of State & $\rho_L$ & $p_L$ & $u_{x,L}$ & $u_{y,L}$ \\
				\midrule
				\id ($\Gamma=1.4$) & $6.230163124916983$ & $0.2737681594569726$ & $0.3701795013577267$ & $-0.2137232347573649$ \\
				\tm                & $4.283403143384633$ & $0.2458911152705976$ & $0.3400404782693554$ & $-0.1963224616641815$ \\
				\ip                & $3.374066413843201$ & $0.2223337777773334$ & $0.3133737731750582$ & $-0.1809264322995886$ \\
				\rc                & $4.395707430079081$ & $0.2481330900558431$ & $0.3425375832788305$ & $-0.1977641659135966$ \\
				\bottomrule
			\end{tabular}
		\end{table}
		\par The different equations of state lead to appreciable variations in the post-shock state and consequently in the density distribution within the reflected flow. As seen in Figure~\ref{fig:dmr}, all four EOS produce the characteristic double Mach reflection pattern, including the strong incident shock, reflected shock, and the complex flow structure near the wedge. However, the magnitude and distribution of density within the shocked region differ significantly among the EOS. In particular, the ID-EOS yields substantially higher density values, whereas the IP-EOS produces a comparatively lower-density field, with the TM-EOS and RC-EOS giving intermediate distributions. These results illustrate the sensitivity of the strong shock structure to the choice of EOS while demonstrating that the WENO-AOI scheme remains effective across all four thermodynamic models.
	\end{testproblem}
	\begin{figure}[ht!]
		\centering
		\begin{subfigure}{0.24\textwidth}
			\centering
			\includegraphics[width=\textwidth]{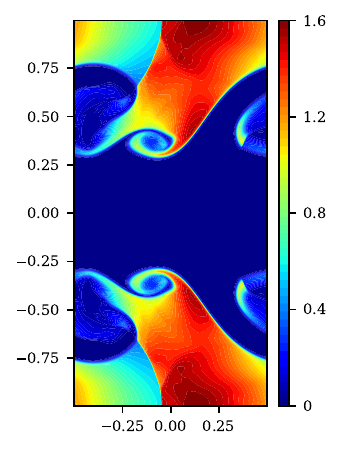}
			\caption{\id}
		\end{subfigure}
		\hfill
		\begin{subfigure}{0.24\textwidth}
			\centering
			\includegraphics[width=\textwidth]{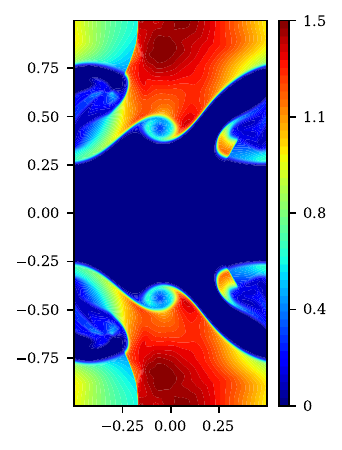}
			\caption{\ip}
		\end{subfigure}
		\hfill
		\begin{subfigure}{0.24\textwidth}
			\centering
			\includegraphics[width=\textwidth]{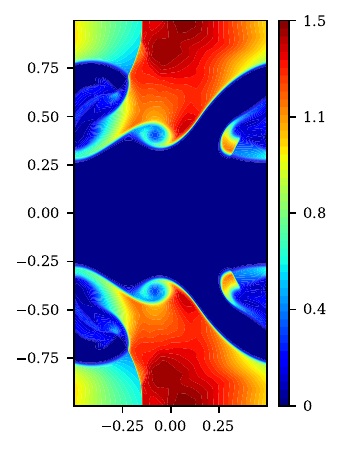}
			\caption{\rc}
		\end{subfigure}
		\hfill
		\begin{subfigure}{0.24\textwidth}
			\centering
			\includegraphics[width=\textwidth]{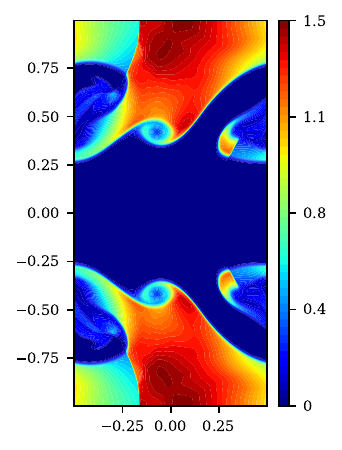}
			\caption{\tm}
		\end{subfigure}
		\caption{Rest-mass density distribution for the relativistic Kelvin--Helmholtz instability at time $t = 3.0$ obtained using the \wenoaoi scheme on a $640 \times 320$ mesh for different equations of state.}
		\label{fig:kh}
	\end{figure} 
	\begin{testproblem}{Relativistic Kelvin--Helmholtz Instability}
		\label{tp:kh}
		The Kelvin--Helmholtz (KH) instability is a classical instability of fluid dynamics,
		invoked in the astrophysical context to explain the observed phenomenology of extended
		radio-jets. Originally studied in the relativistic regime by
		\cite{mignone2009five,beckwith2011second,radice2012thc}, it was used as a benchmark
		for high-order AMR schemes in \cite{zanotti2015high}. It features two shear layers
		that roll up into vortices, transitioning from a linear to a nonlinear regime, and
		serves as a sensitive test for the ability of a scheme to capture contact
		discontinuities and fine-scale structures.
		
		The computational domain is $\Omega = [-0.5, 0.5] \times [-1.0, 1.0]$ with periodic
		boundary conditions in both directions and adiabatic index $\gamma = 4/3$. The
		pre-shock (right) state is not applicable here; instead, the initial conditions
		consist of two shear layers at $y = \pm 0.5$ with a transverse velocity perturbation
		to trigger the instability. Specifically, the initial primitive variables are
		\begin{align*}
			u_x &= \begin{cases}
				v_s \tanh\!\left(\dfrac{y - 0.5}{a}\right), & y > 0, \\[6pt]
				-v_s \tanh\!\left(\dfrac{y + 0.5}{a}\right), & y \leq 0,
			\end{cases} \\[8pt]
			u_y &= \begin{cases}
				\eta_0\, v_s \sin(2\pi x)\exp\!\left(-\dfrac{(y-0.5)^2}{\sigma}\right), & y > 0, \\[6pt]
				-\eta_0\, v_s \sin(2\pi x)\exp\!\left(-\dfrac{(y+0.5)^2}{\sigma}\right), & y \leq 0,
			\end{cases} \\[8pt]
			\rho &= \begin{cases}
				\rho_0 + \rho_1 \tanh\!\left(\dfrac{y - 0.5}{a}\right), & y > 0, \\[6pt]
				\rho_0 - \rho_1 \tanh\!\left(\dfrac{y + 0.5}{a}\right), & y \leq 0,
			\end{cases}
		\end{align*}
		with uniform pressure $p = 1$ everywhere. The parameters are $v_s = 0.5$,
		$a = 0.01$, $\sigma = 0.1$, $\eta_0 = 0.1$, $\rho_0 = 0.505$, and $\rho_1 = 0.495$.
		
		As the system evolves, the shear layers roll up into vortices and the flow transitions
		from the linear to the nonlinear regime for $t \in [2, 3]$. The computations are run
		up to the final time $t = 3.0$ and the rest-mass density is plotted in Figure \ref{fig:kh}. The results obtained using the WENO-AOI scheme show the formation of one primary vortex at each shear layer and the absence of prominent secondary vortices, with similar structures across the four EOS, consistent with the findings of \cite{zanotti2015high} that secondary vortices produced by less dissipative schemes are numerical artifacts rather than physical structures.
	\end{testproblem}
	
	\begin{figure}[ht!]
		\centering
		\begin{subfigure}{0.32\textwidth}
			\centering
			\includegraphics[width=\textwidth]{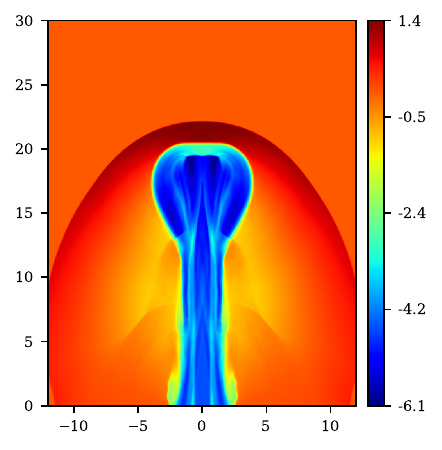}
			\caption{$v_b=0.99$ with $M_b=1.72$}
		\end{subfigure}
		\begin{subfigure}{0.32\textwidth}
			\centering
			\includegraphics[width=\textwidth]{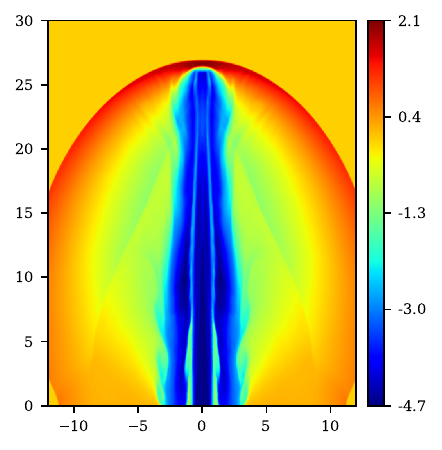}
			\caption{$v_b=0.999$ with $M_b=1.74$}
		\end{subfigure}
		\begin{subfigure}{0.32\textwidth}
			\centering
			\includegraphics[width=\textwidth]{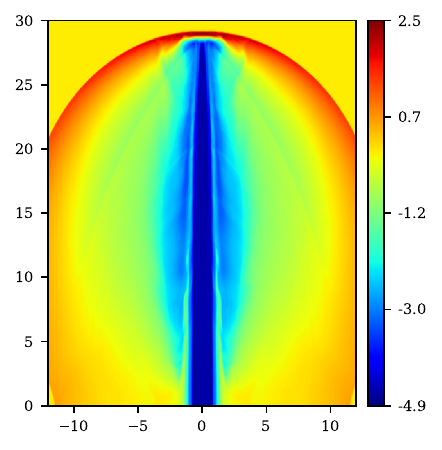}
			\caption{$v_b=0.9999$ with $M_b=1.74$}
		\end{subfigure}
		\caption{Rest-mass density for the pressure-matched hot jet model of Test Problem~\ref{tp:rjet} at time $t=30.0$, obtained using the \wenoaoi scheme on a $240\times600$ mesh.}
		\label{fig:hot_jet}
	\end{figure}
	
	\begin{figure}[ht!]
		\centering
		\begin{subfigure}{0.32\textwidth}
			\centering
			\includegraphics[width=\textwidth]{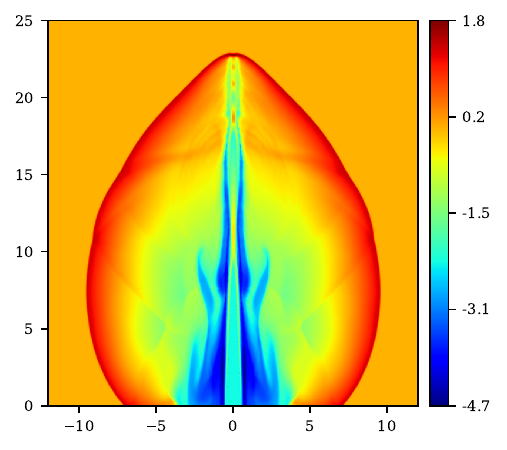}
			\caption{$v_b=0.99$ with $M_b=50$}
		\end{subfigure}
		\begin{subfigure}{0.32\textwidth}
			\centering
			\includegraphics[width=\textwidth]{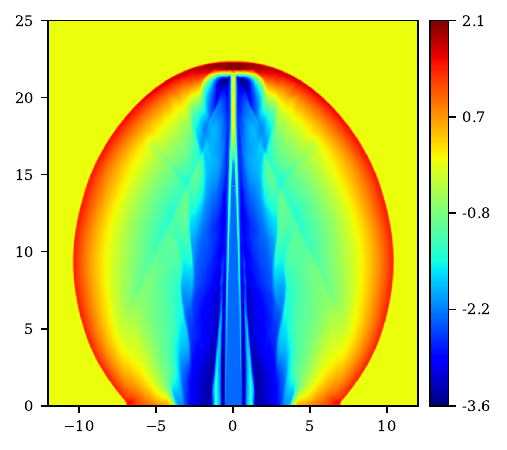}
			\caption{$v_b=0.999$ with $M_b=50$}
		\end{subfigure}
		\begin{subfigure}{0.32\textwidth}
			\centering
			\includegraphics[width=\textwidth]{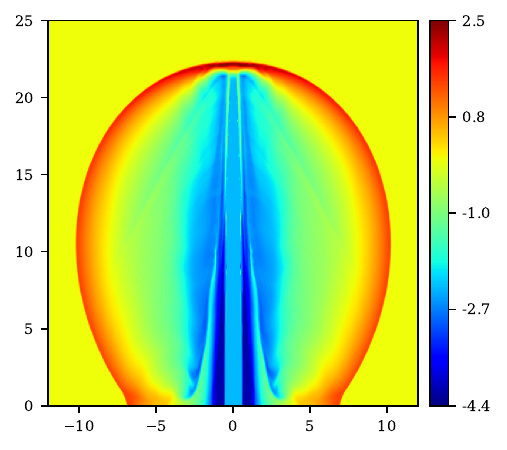}
			\caption{$v_b=0.9999$ with $M_b=50$}
		\end{subfigure}
		\caption{Rest-mass density for the pressure-matched highly supersonic (cold) jet model of Test Problem~\ref{tp:rjet} at times $t=30.0$, $25.0$, and $23.0$, respectively, obtained using the \wenoaoi scheme on a $240\times500$ mesh.}
		\label{fig:cold_jet}
	\end{figure}
	
	
	\begin{testproblem}{Relativistic Jet Problem \label{tp:rjet}}
		
		This problem models high-speed relativistic jet flows and serves as a challenging benchmark for multidimensional relativistic hydrodynamic solvers. Such jets are ubiquitous in extragalactic radio sources associated with active galactic nuclei, where strong relativistic shock waves, shear waves, interface instabilities, and ultra-relativistic flow regions naturally arise. The present setup follows the classical relativistic jet benchmark introduced by \cite{marti1997morphology,aloy1999genesis}, with the parameter choices inspired by \cite{wu2017physical}.
		
		The ambient medium is initially at rest with unit rest-mass density, and a relativistic jet is continuously injected in the positive $y$-direction through the inlet region $|x|\le0.5$ on the bottom boundary ($y=0$). A reflective boundary condition is imposed along $x=0$, a fixed inflow condition is prescribed on the nozzle, and outflow boundary conditions are applied on all remaining boundaries.
		
		The first test considers the pressure-matched hot jet model. The computational domain is $[0,12]\times[0,30]$. The injected jet has a rest-mass density $\rho_b=0.01$, a pressure equal to the ambient pressure, and a beam velocity $v_b$. The RC EOS is employed. Three configurations are considered: (i) $v_b=0.99$ with $M_b=1.72$, (ii) $v_b=0.999$ with $M_b=1.74$, and (iii) $v_b=0.9999$ with $M_b=1.74$.
		
		The second test considers the pressure-matched highly supersonic (cold) jet model. The setup is identical to the hot jet model except that the jet density is increased to $\rho_b=0.1$, the computational domain is changed to $[0,12]\times[0,25]$, and the beam Mach number is taken as $M_b=50$ for all three configurations. The three cases considered are: (i) $v_b=0.99$, (ii) $v_b=0.999$, and (iii) $v_b=0.9999$.

		The relativistic cold jet with configuration (iii) represents the most extreme regime in our test suite. Driven by a nearly light-speed inflow velocity ($v_b=0.9999$), it exhibits a massive Lorentz factor of $W_b\approx 70.7$. When coupled with a cold thermal profile ($c_s = 0.02$ for $M_b= 50$) this yields an exceptionally large relativistic Mach number of $M_r \approx 3.53\times 10^3$. Under these conditions, the kinetic energy flux dominates the thermal pressure by several orders of magnitude. The resulting impact with the ambient medium triggers a violent shock deceleration, generating the most severe pressure and density compression jumps in this study, which strictly tests the robustness of the PCP mechanism.

		The resulting rest-mass density distributions are shown in Figures~\ref{fig:hot_jet} and~\ref{fig:cold_jet}. For both jet models, increasing the beam velocity produces a more extended and collimated jet structure. The cold-jet configurations exhibit stronger axial compression and sharper internal structures, with these features becoming more pronounced as $v_b$ approaches the speed of light. The WENO-AOI scheme remains stable and well resolved even for the most extreme relativistic configuration.
		
	\end{testproblem}

	\section{Conclusion}
	\label{sec:conclusions}
	We have developed a high-order accurate, physical-constraint-preserving AFD-WENO scheme for solving the RHD equations with a general EOS. The PCP property is rigorously proven for the proposed scheme, ensuring that the numerical solutions remain within the physically admissible set. The scheme is designed to handle a wide range of EOS. Extensive numerical experiments, including one-dimensional and two-dimensional test problems demonstrate the robustness, accuracy, and effectiveness of the proposed scheme. 
	
	\section*{Acknowledgments}
	Rakesh Kumar is supported by the Prime Minister Early Career Research Grant (PMECRG) of the Anusandhan National Research Foundation (ANRF), India, under Grant No. ANRF/ECRG/2025/004846/PMS. Biswarup Biswas is supported by the State University Research Excellence (SURE) Scheme of the ANRF, India, under File No. SUR/2022/001786. The authors acknowledge the use of the High-Performance Computing (HPC) facility at Mahindra University, Hyderabad, India, for the numerical simulations presented in this work.
	
	\bibliographystyle{plain}
	\bibliography{references}
\end{document}